\documentclass[11pt]{amsart}
\usepackage[margin=1in]{geometry}

\usepackage{amssymb}
\usepackage{amsthm}
\usepackage{amsmath}
\usepackage{mathrsfs}
\usepackage{amsbsy}
\usepackage[all]{xy}
\usepackage{bm}
\usepackage{hyperref}
\usepackage{tikz}
\usetikzlibrary{quotes}
\usepackage{array}
\usepackage{enumerate}
\usepackage{xcolor}
\usepackage{bbm}
\usepackage{comment}
\usepackage{dynkin-diagrams}
\usepackage{ifthen}
\usepackage{thmtools, thm-restate}

\usepackage[noabbrev,capitalise,nameinlink]{cleveref}

\definecolor{lavender}{rgb}{0.4,0,1.0}
\definecolor{LightPink}{rgb}{1.0,0.784,1.0}
\definecolor{DarkPink}{rgb}{1.0,0,0.627}
\definecolor{Pink}{rgb}{1.0,0,1}
\definecolor{MyGreen}{rgb}{0,0.784,0} 
\definecolor{DarkRed}{rgb}{0.35,0,0}
\definecolor{FaintRed}{rgb}{1,0.64,0.64}
\definecolor{MyOrange}{rgb}{1,0.7,0}

\hypersetup{colorlinks=true, citecolor=lavender, linkcolor=lavender,urlcolor=lavender}

\usepackage{hhline}
\allowdisplaybreaks
\usepackage[noadjust]{cite}
\usepackage{asymptote}

\usepackage{caption}
\usepackage{tabu}
\usepackage{diagbox}
\usepackage[noabbrev,capitalise,nameinlink]{cleveref}
\crefname{conjecture}{Conjecture}{Conjectures}

\newtheorem{theorem}{Theorem}[section]
\newtheorem{proposition}[theorem]{Proposition}
\newtheorem{corollary}[theorem]{Corollary}

\newtheorem{lemma}[theorem]{Lemma}

\theoremstyle{definition}
\newtheorem{definition}[theorem]{Definition}
\newtheorem{remark}[theorem]{Remark}
\newtheorem{example}[theorem]{Example}

\newcommand{\CC}{\mathbb{C}}

\newcommand{\RR}{\mathbb{R}}

\newcommand{\ZZ}{\mathbb{Z}}

\newcommand{\cA}{\mathcal{A}}

\newcommand{\cC}{\mathcal{C}}
\newcommand{\cD}{\mathcal{D}}

\newcommand{\cF}{\mathcal{F}}

\newcommand{\cO}{\mathcal{O}}
\newcommand{\cP}{\mathcal{P}}
\newcommand{\cQ}{\mathcal{Q}}

\newcommand{\cT}{\mathcal{T}}
\newcommand{\cU}{\mathcal{U}}

\newcommand{\tA}{\widetilde{A}}

\newcommand{\fD}{\mathfrak{D}}
\newcommand{\sT}{\mathscr{T}}

\usepackage{etoolbox}

\newcommand{\redbullet}{{\color{DarkRed}\bullet}} 
\newcommand{\syl}{\mathrm{syl}}
\newcommand{\po}{\sqsubseteq}
\newcommand{\TIBIT}{\mathsf{CTam}_{\mathsf{Tree}}}
\newcommand{\ATIBIT}{\mathsf{ATam}_{\mathsf{Tree}}}
\newcommand{\coATIBIT}{\mathsf{ATam}^{\mathsf{Tree}}}
\newcommand{\spine}{\mathrm{spine}}
\newcommand{\Krew}{\mathrm{Krew}} 
\renewcommand{\Row}{\mathrm{Row}}
\newcommand{\pr}{\mathrm{par}}

\newcommand{\CDO}{\mathsf{CDyer}}
\newcommand{\DO}{\mathsf{Dyer}}
\newcommand{\flip}{\mathrm{flip}}
\newcommand{\ATam}{\mathsf{ATam}}
\newcommand{\CTam}{\mathsf{CTam}}
\newcommand{\T}{v}
\newcommand{\Inv}{\mathrm{Inv}}

\newcommand{\pispine}{\pi_{\mathrm{spine}}} 
\newcommand{\Tors}{\mathsf{Tors}} 
\newcommand{\JIrr}{\mathrm{JIrr}}
\newcommand{\MIrr}{\mathrm{MIrr}}
\renewcommand{\j}{\mathfrak{j}} 
\newcommand{\Po}{\mathbf{P}}
\newcommand{\Pob}{\mathbf{P}^{\redbullet}}
\newcommand{\Cat}{\mathrm{Cat}}
\newcommand{\NC}{\mathrm{NC}} 
\newcommand{\TINC}{\mathrm{TINC}} 
\newcommand{\uw}{\mathcal{R}}

\newcommand{\Orn}{\mathrm{Orn}}
\newcommand{\SSS}{\mathfrak{S}}
\newcommand{\Ext}{\mathrm{Ext}}
\newcommand{\Hom}{\mathrm{Hom}}
\newcommand{\cone}{\mathrm{cone}}
\newcommand{\tbe}{\widetilde{\mathbf{e}}}
\newcommand{\bg}{\mathbf{g}}

\newcommand{\varpiv}{\varpi}
\newcommand{\omm}{\boldsymbol{\omega}}

\newcommand{\dfn}[1]{\textcolor{blue}{\emph{#1}}}

\begin{document}

\title[]{The Affine Tamari Lattice}
\subjclass[2010]{}

\author[]{Grant Barkley}
\address[]{Department of Mathematics, Harvard University, Cambridge, MA 02138, USA}
\email{gbarkley@math.harvard.edu}

\author[]{Colin Defant}
\address[]{Department of Mathematics, Harvard University, Cambridge, MA 02138, USA}
\email{colindefant@gmail.com}

\begin{abstract} 
Given a fixed integer $n\geq 2$, we construct two new finite lattices that we call the \emph{cyclic Tamari lattice} and the \emph{affine Tamari lattice}. The cyclic Tamari lattice is a sublattice and a quotient lattice of the \emph{cyclic Dyer lattice}, which is the infinite lattice of \emph{translation-invariant total orders} under containment of inversion sets. The affine Tamari lattice is a quotient of the \emph{Dyer lattice}, which in turn is a quotient of the cyclic Dyer lattice and is isomorphic to the collection of biclosed sets of the root system of type $\widetilde{A}_{n-1}$ under inclusion. We provide numerous combinatorial and algebraic descriptions of these lattices using translation-invariant total orders, translation-invariant binary in-ordered trees, noncrossing arc diagrams, torsion classes, triangulations, and translation-invariant noncrossing partitions. The cardinalities of the cyclic and affine Tamari lattices are the Catalan numbers of types $B_n$ and $D_n$, respectively. We show that these lattices are self-dual and semidistributive, and we describe their decompositions coming from the Fundamental Theorem of Finite Semidistributive Lattices. We also show that the rowmotion operators on these lattices have well-behaved orbit structures, which we describe via the cyclic sieving phenomenon. Our new combinatorial framework allows us to prove that the lengths of maximal green sequences for the completed path algebra of the oriented $n$-cycle are precisely the integers in the interval $[2n-1,\binom{n+1}{2}]$. 
\end{abstract}

\maketitle

\section{Introduction}\label{sec:intro}

\begin{figure}[ht]
\begin{center}{\includegraphics[height=14.5cm]{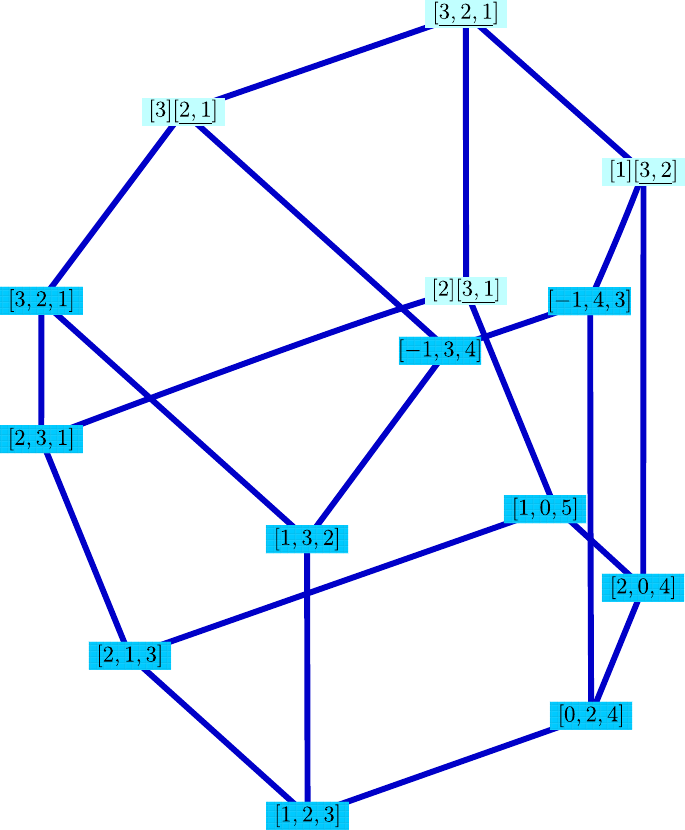}}
\end{center}
\caption{The Hasse diagram of the affine Tamari lattice for $n=3$, with each element represented as a window notation of a $312$-avoiding translation-invariant total order. The darker elements are $312$-avoiding affine permutations; the subposet that they induce is not a lattice.} 
\label{fig:intro_polytope} 
\end{figure}

The \emph{Tamari lattice}, introduced by Tamari in 1962 \cite{Tamari}, is a fundamental object with several remarkable properties and numerous incarnations in combinatorics, polyhedral geometry, quiver representation theory, cluster algebras, and category theory \cite{vonBellUnifying,BousquetRep,CeballosGeometry,DefantLin,MPS,PP,PrevilleViennot,ReadingCambrian,Rognerud,WhyTheFuss}. It arises naturally as a quotient of the weak order on the symmetric group $\SSS_n$ via the \emph{sylvester congruence}. It is also the sublattice of the weak order induced by the set of $312$-avoiding permutations. 

It is natural to try defining an affine analogue of the Tamari lattice using the weak order on the affine symmetric group $\widetilde \SSS_n$. One possible approach is to consider the subposet induced by $312$-avoiding affine permutations (as defined in \cite{Crites}). However, the resulting poset lacks many desirable properties. Most notably, neither the weak order on $\widetilde \SSS_n$ nor the induced subposet of $312$-avoiding affine permutations is a lattice (they are, however, meet-semilattices). In order to rectify these issues, we turn to the \emph{Dyer order}, a partial order on biclosed sets of roots of $\widetilde \SSS_n$ introduced by Dyer in \cite{Dyer1,Dyer2} and other unpublished works. Barkley and Speyer \cite{BarkleySpeyerCombinatorics,BarkleySpeyerLattice}, settling a conjecture of Dyer, proved that the Dyer order is a lattice.\footnote{In the article \cite{BarkleySpeyerLattice}, the Dyer order was called the \emph{extended weak order}.}  In this article, we introduce the \emph{affine Tamari lattice}, a finite lattice that we obtain naturally as a lattice quotient of the Dyer lattice. In fact, we will obtain the affine Tamari lattice as a quotient of another new finite lattice that we call the \emph{cyclic Tamari lattice}, which is a quotient of a lattice called the \emph{cyclic Dyer lattice}. 

Our primary goal is to begin developing the theory of cyclic and affine Tamari lattices, which appears to be quite rich. Let us present an outline of our article, highlighting our main contributions. 

\cref{sec:preliminaries} presents necessary background on lattices. 

In \cref{sec:TITOs}, we define \emph{translation-invariant total orders} (TITOs) and the Dyer order. The \emph{cyclic Dyer lattice} is the set of TITOs under the Dyer order, while the \emph{Dyer lattice} is the set of \emph{real TITOs} under the Dyer order. We define the \emph{cyclic Tamari lattice} and the \emph{affine Tamari lattice} as the subposets of the cyclic Dyer lattice and the Dyer lattice, respectively, consisting of \emph{$312$-avoiding TITOs}.  

In \cref{sec:trees}, we introduce \emph{translation-invariant binary in-ordered trees} (TIBITs) and a natural partial order on them that we also call the \emph{Dyer order}. By considering \emph{translation-invariant linear extensions} (TILEs) of TIBITs, we construct an isomorphism between the cyclic Tamari lattice and the poset of TIBITs under the Dyer order; as a by-product, we also find that the cyclic Tamari lattice is self-dual (\cref{cor:isomorphism}). We then use TILEs to define a lattice congruence on the cyclic Dyer lattice that we call the \emph{cyclic sylvester congruence}. This allows us to show that the cyclic Tamari lattice is indeed a lattice and is in fact a quotient of the cyclic Dyer lattice (\cref{thm:cyclic_quotient}). In \cref{thm:cyclic_sublattice}, we prove that the cyclic Tamari lattice is also a sublattice of the cyclic Dyer lattice. We then use TIBITs to prove that the affine Tamari lattice is a quotient (but not a sublattice) of the cyclic Tamari lattice. Because the cyclic Dyer lattice is semidistributive, we deduce that the cyclic and affine Tamari lattices are as well. The antiautomorphism of the cyclic Tamari lattice that we construct permutes fibers of the quotient map that we construct from the cyclic Tamari lattice to the affine Tamari lattice, so we deduce that the affine Tamari lattice is also self-dual. 

Reading \cite{ReadingNoncrossing} interpreted the join-irreducible elements of the Tamari lattice graphically as \emph{arcs}, and he showed that canonical join representations of elements of the Tamari lattice correspond to certain \emph{noncrossing arc diagrams}. In \cref{sec:Arcs_FTFSDL}, we describe a similar story for the cyclic and affine Tamari lattices using arcs drawn in an annulus. We also provide an explicit description of the decompositions of the cyclic and affine Tamari lattices coming from the Fundamental Theorem of Finite Semidistributive Lattices \cite{FTFSDL}. As a corollary, we find that the sizes of the cyclic and affine Tamari lattices are given by the Catalan numbers of types $B_n$ and $D_n$, respectively (\cref{cor:Catalan}). 

\cref{sec:quivers} provides a representation-theoretic interpretation for the cyclic Tamari lattice and the affine Tamari lattice. We define two finite-dimensional algebras called the \emph{cyclic Tamari algebra} and the \emph{affine Tamari algebra}. The cyclic Tamari algebra is the quotient of the path algebra of an oriented $n$-cycle quiver by the paths of length $n$. The affine Tamari algebra is the quotient of the path algebra of an oriented $n$-cycle quiver by the paths of length $n-1$.  We show that their lattices of torsion classes are isomorphic to the cyclic Tamari lattice and the affine Tamari lattice, respectively.

In \cref{sec:triangulations}, we consider a standard model for the exchange graph of the cluster algebra of type $D_n$ using triangulations of a punctured $n$-gon. Using $\tau$-tilting theory, we provide two different descriptions of an isomorphism from the Hasse diagram of the affine Tamari lattice to this exchange graph. The first description computes the s$\tau$-tilting module generating a torsion class. We describe this process combinatorially in terms of collections of arcs called \emph{affine arc torsion classes}. The second description is geometric, making use of the stability fan (equivalently, the {\bf g}-vector fan) of the affine Tamari algebra, though it can also be understood purely combinatorially in terms of TIBITs.   

A maximal chain in the cyclic (respectively, affine) Tamari lattice is the same thing as a maximal green sequence for the cyclic (respectively, affine) Tamari algebra. In \cref{sec:green} (specifically, \cref{thm:max_chains}), we use our novel combinatorial models to prove that the maximum length of a maximal chains in the cyclic (respectively, affine) Tamari lattice is $\binom{n+1}{2}$ (respectively, $\binom{n+1}{2}-1$). By known results, this implies that the lengths of maximal green sequences in the cyclic (respectively, affine) Tamari algebra are the integers in the interval $[2n-1,\binom{n+1}{2}]$ (respectively, $[2n-2,\binom{n+1}{2}-1]$). In the case of the affine Tamari lattice, this yields a new proof of a result due to Apruzzese and Igusa. Our result is new for the cyclic Tamari algebra. In fact, our result implies that the lengths of maximal green sequences for the completed path algebra of the oriented $n$-cycle are the integers in the interval $[2n-1,\binom{n+1}{2}]$, which is also new.   

Each finite semidistributive lattice $L$ has an associated invertible operator $\Row\colon L\to L$ called \emph{rowmotion}. \cref{sec:rowmotion} considers the dynamics of rowmotion on the cyclic and affine Tamari lattices. We find that rowmotion on the cyclic Tamari lattice is dynamically equivalent to the Kreweras complement operator on noncrossing partitions of type~$B$. Thus, by invoking a result of Armstrong, Stump, and Thomas about Kreweras complement \cite{AST}, we are able to describe the orbit structure of rowmotion on the cyclic Tamari lattice via the cyclic sieving phenomenon (\cref{thm:CSP_cyclic}). Rowmotion on the affine Tamari lattice is a bit more complicated, but we can still describe its orbit structure via the cyclic sieving phenomenon (\cref{thm:CSP_affine}). As far as we are aware, the cyclic and affine Tamari lattices are the first examples of semidistributive lattices that are not trim yet still have nice rowmotion dynamics (see \cref{rem:trim}). 

Finally, \cref{sec:future} collects several suggestions for future work.

\section{Preliminaries}\label{sec:preliminaries}  

\subsection{Lattices}\label{subsec:lattices} 

Let $P$ be a poset with partial order relation $\leq$. For $x,y\in P$ with $x\leq y$, the \dfn{interval} between $x$ and $y$ is the set $[x,y]=\{z\in P\mid x\leq z\leq y\}$. If $[x,y]$ has cardinality $2$, then we say $y$ \dfn{covers} $x$ and write $x\lessdot y$. In this case, we also say that $x\leq y$ is a \dfn{cover relation}. 

A \dfn{lattice} is a poset $L$ such that any two elements $x,y\in L$ have a meet (i.e., greatest lower bound), which is denoted $x\wedge y$, and a join (i.e., least upper bound), which is denoted $x\vee y$. If $L$ is a lattice and $X$ is a finite subset of $L$, then $X$ has a meet $\bigwedge X$ and a join $\bigvee X$. A \dfn{complete lattice} is a poset $L$ such that every (possibly infinite) subset $X$ of $L$ has a meet $\bigwedge X$ and a join $\bigvee X$. Every finite lattice is automatically a complete lattice. Thus, when referring to finite complete lattices, we will often drop the adjective \emph{complete}. 

Suppose $L$ and $L'$ are complete lattices. A map $\varphi\colon L\to L'$ is a \dfn{homomorphism} (of complete lattices) if for every $X\subseteq L$, we have $\bigwedge\varphi(X)=\varphi\left(\bigwedge X\right)$ and $\bigvee\varphi(X)=\varphi\left(\bigvee X\right)$. We say $L'$ is a \dfn{quotient} of $L$ if there is a surjective homomorphism from $L$ to $L'$. 

Suppose $\equiv$ is an equivalence relation on $L$ whose equivalence classes are intervals. For $x\in L$, let $\pi_\equiv^\downarrow(x)$ and $\pi_\equiv^\uparrow(x)$ be the minimum and maximum elements, respectively, of the equivalence class containing $x$. Let $L'=\{\pi_\equiv^\downarrow(x)\mid x\in L\}$ be the set of minimum elements of equivalence classes. If $L'$ is a complete lattice and the map $\pi_\equiv^\downarrow\colon L\to L'$ is a complete lattice homomorphism, then we say $\equiv$ is a \dfn{complete lattice congruence}. We will need the following useful proposition.\footnote{This result is stated in \cite{ReadingBook} for finite lattices, but the same proof works \emph{mutatis mutandis} for arbitrary complete lattices.} 

\begin{proposition}[{\cite[Proposition~9-5.2]{ReadingBook}}]\label{prop:9-5.2}
An equivalence relation $\equiv$ on a complete lattice $L$ is a complete lattice congruence if and only if each equivalence class of $\equiv$ is an interval and the maps $\pi_\equiv^\downarrow,\pi_\equiv^\uparrow\colon L\to L$ are order-preserving. 
\end{proposition}

Let $L$ be a complete lattice, and let $L'\subseteq L$ be a subposet of $L$ that is also a complete lattice. We say $L'$ is a \dfn{complete sublattice} of $L$ if for every subset $X\subseteq L'$, the meet and join of $X$ in $L$ belong to $L'$. 

\begin{proposition}\label{prop:join-subsemilattice}
Let $\equiv$ be a complete lattice congruence on a complete lattice $L$, and consider ${L'=\{\pi_\equiv^\downarrow(x)\mid x\in L\}}$. For every subset $X\subseteq L'$, the join of $X$ in $L$ is in $L'$. 
\end{proposition} 

In general, the statement analogous to \cref{prop:join-subsemilattice}, but with joins replaced by meets, is false.

Let $L$ be a complete lattice. An element $j\in L$ is \dfn{completely join-irreducible} if every set $X\subseteq L$ satisfying $j=\bigvee X$ contains $j$. Similarly, an element $m\in L$ is \dfn{completely meet-irreducible} if every set ${X\subseteq L}$ satisfying $m=\bigwedge X$ contains $m$. Let $\JIrr(L)$ and $\MIrr(L)$ be the set of completely join-irreducible elements of $L$ and the set of completely meet-irreducible elements of $L$, respectively. We say $L$ is \dfn{completely meet-semidistributive} if for all $x,y\in L$ with $x\leq y$, the set ${\{z\in L\mid z\wedge y=x\}}$ has a maximum element. We say $L$ is \dfn{completely join-semidistributive} if its dual is completely meet-semidistriutive. We say $L$ is \dfn{completely semidistributive} if it is both completely meet-semidistributive and completely join-semidistributive.  

Let $L$ be a finite semidistributive lattice. Each join-irreducible element $j$ of $L$ covers a unique element, which we denote by $j_*$, and each meet-irreducible element $m$ is covered by a unique element, which we denote by $m^*$. For $j\in\JIrr(L)$, let $\kappa(j)=\kappa_L$ be the maximum element of $\{z\in L\mid z\wedge j=j_*\}$. It is known \cite{Barnard,Freese} that $\kappa(j)$ is meet-irreducible and that the map $\kappa=\kappa_L\colon\JIrr(L)\to\MIrr(L)$ is a bijection. For each $m\in\MIrr(L)$, the minimum element of $\{z\in L\mid z\vee m=m^*\}$ is $\kappa^{-1}(m)$. For every edge $x\lessdot y$ in the Hasse diagram of $L$, there is a unique join-irreducible $\j(x\lessdot y)\in\JIrr(L)$ such that 
\[\j(x\lessdot y)\vee x=y\quad\text{and}\quad \kappa(\j(x\lessdot y))\wedge y=x.\] We imagine labeling the edge $x\lessdot y$ with $\j(x\lessdot y)$. This provides a canonical way to label the edges of the Hasse diagram of $L$ with elements of $\JIrr(L)$. For $u\in L$, let 
\[\cD(u)=\{\j(x\lessdot u)\mid x\lessdot u\}\quad\text{and}\quad\cU(u)=\{\j(u\lessdot y)\mid u\lessdot y\}.\] The sets $\cD(u)$ and $\cU(u)$ are called the \dfn{canonical join representation} of $u$ and the \dfn{canonical meet representation} of $u$, respectively. It is known that 
\[u=\bigvee\cD(u)=\bigwedge\kappa(\cU(u)).\] 
There is a unique bijection $\Row\colon L\to L$, called the \dfn{rowmotion} operator, that satisfies 
\[\cD(u)=\cU(\Row(u))\] for all $u\in L$ \cite{Barnard}.

The Fundamental Theorem of Finite Semidistributive Lattices (FTFSDL) \cite{FTFSDL} gives a description of a finite semidistributive lattice $L$ in terms of its meet-irreducibles and join-irreducibles. The \dfn{FTFSDL factorization} of $L$ is a triple ($\twoheadrightarrow,\hookrightarrow,\to$) of relations on $\JIrr(L)$. The relation $\twoheadrightarrow$ is defined so that $j_1 \twoheadrightarrow j_2$ if and only if $j_1\geq j_2$. The relation $\hookrightarrow$ is defined so that $j_1 \hookrightarrow j_2$ if and only if $\kappa(j_1)\geq \kappa(j_2)$. The relation $\to$ is defined so that $j_1 \to j_2$ if and only if $j_1\not\leq \kappa(j_2)$. The lattice $L$ can be recovered from the data $(\JIrr(L), \twoheadrightarrow,\hookrightarrow,\to)$ as the lattice of \dfn{maximal orthogonal pairs} of $\JIrr(L)$. A maximal orthogonal pair is a pair $(T,F)$ of subsets of $\JIrr(L)$ such that \[F=\{r\in \JIrr(L) \mid \text{there does not exist }\ell\in T\text{ satisfying }\ell\to r\}\] and \[T=\{\ell\in \JIrr(L) \mid \text{there does not exist } r \in F\text{ satisfying }\ell\to r\}.\] The collection of maximal orthogonal pairs is partially ordered so that $(T,F)\leq T',F'$ if and only if $T\subseteq T'$ (equivalently, $F\supseteq F'$). The map \[x\mapsto(\{j\in \JIrr(L)\mid j\leq x\}, \kappa^{-1}\{m\in \MIrr(L) \mid m\geq x\})\] is an isomorphism from $L$ to the poset of maximal orthogonal pairs; the inverse of this isomorphism sends a maximal orthogonal pair $(T,F)$ to $\bigvee T$. In particular, two finite semidistributive lattices are isomorphic if and only if they have isomorphic FTFSDL factorizations. 
We have $x\to z$ if and only if there exists a $y$ with $x\twoheadrightarrow y \hookrightarrow z$. We have $x\twoheadrightarrow y$ if and only if for all $y\to z$, we also have $x\to z$. We have $x\hookrightarrow y$ if and only if for all $z\to x$, we also have $z\to y$. Therefore, the relations $\twoheadrightarrow$ and $\hookrightarrow$ determine the relation $\to$ and vice versa.

\subsection{Enumeration}

The $N$-th \dfn{Catalan number} (i.e., the Catalan number of type $A_{N-1}$) is \[\Cat_{A_{N-1}}=\frac{1}{N+1}\binom{2N}{N}.\] We will also need the Catalan numbers of types $B_N$ and $D_N$, which are 
\begin{equation}\label{eq:Catalan_BD}
\Cat_{B_N}=\binom{2N}{N}\quad\text{and}\quad \Cat_{D_N}=\frac{3N-2}{N}\binom{2N-2}{N-1}.
\end{equation}
Let 
\[[N]_q=\frac{q^N-1}{q-1}\quad\text{and}\quad[N]_q!=[N]_q[N-1]_q\cdots[1]_q.\]
A $q$-analogue of the $N$-th Catalan number is \[\Cat_{A_{N-1}}(q)=\prod_{i=1}^{N-1}\frac{[N+i+1]_q}{[i+1]_q};\] note that $\Cat_{A_{N-1}}(1)=\Cat_{A_{N-1}}$. A $q$-analogue of the Catalan number of type $B_N$ is \[\Cat_{B_{N}}(q)=\prod_{i=1}^{N}\frac{[2N+2i]_q}{[2i]_q};\] note that $\Cat_{B_N}(1)=\Cat_{B_N}$.

\section{TITOs}\label{sec:TITOs} 
\subsection{The Cyclic and Affine Dyer Lattices}
Let us fix once and for all an integer $n\geq 2$. 

Recall that a \dfn{total order} on a set $Y$ is a partial order $\preceq$ on $Y$ such that for all $a,b\in Y$, either $a\preceq b$ or $b\preceq a$. Given a total order $\preceq$ on $Y$, we say a subset $I\subseteq Y$ is \dfn{order-convex} if $a\preceq b\preceq c$ and $a,c\in I$ imply that $b\in I$. As before, for $a,c\in Y$ with $a\preceq c$, we let $[a,c]=\{b\in Y\mid a\preceq b\preceq c\}$. 

\begin{definition}
A \dfn{translation-invariant total order} (TITO) is a total order $\preceq$ on $\ZZ$ such that
\[ a \preceq b \text{ if and only if } a + n \preceq b + n \]
for all $a,b\in \ZZ$.
\end{definition}

A \dfn{block} of a TITO $\preceq$ is an order-convex subset $I$ of $\ZZ$ with no minimal or maximal element such that for all $a,b\in I$, the interval $[a,b]$ is finite. Equivalently, a block of $\preceq$ is an interval that is order-isomorphic to the usual ordering on $\ZZ$. The \dfn{size} of a block $I$ is the number of distinct residue classes of integers modulo $n$ appearing in $I$. A \dfn{window} for $I$ consists of $k$ consecutive elements $a_1,\ldots,a_k$ of $I$, where $k$ is the size of $I$. If $a_1\preceq a_1+n$, then we call $I$ a \dfn{waxing block} and denote its window as $[a_1,\ldots,a_k]$; if $a_1\succeq a_1+n$, then we call $I$ a \dfn{waning block} and denote its window as $[\underline{a_1,\ldots,a_k}]$ (note that the definitions of waxing and waning do not depend on the element $a_1$ chosen to start the window). A \dfn{window notation} of $\preceq$ consists of one window for each block, listed from left to right in the order they appear in $\preceq$. We say a TITO $\preceq$ is \dfn{compact} if it does not have two consecutive waxing blocks.

\begin{example}\label{ex:TITO}
Let $n=5$, and consider the TITO $\preceq$ defined so that 
\[ \cdots \preceq 11 \preceq 6 \preceq 1 \preceq \cdots \preceq \cdots \preceq 3 \preceq 2 \preceq 8 \preceq 7 \preceq \cdots\preceq \cdots \preceq 10 \preceq  4\preceq 5 \preceq -1 \preceq \cdots. \]
This TITO has three blocks; the first and third are waning, while the second is waxing. Hence, $\preceq$ is compact. The TITO $\preceq$ has many window notations, which depend on where we start the windows. For example, three window notations representing this TITO are \[[\underline{1}][3,2][\underline{4,5}],\quad [\underline{1}][2,8][\underline{10,4}],\quad\text{and}\quad[\underline{11}][-3,3][\underline{5,-1}].\] 
\end{example}

In general, we are allowed to ``slide'' a window $[a_1,\ldots,a_k]$ of a waxing block one step to the right to obtain a window $[a_2,\ldots,a_k,a_1+n]$ representing the same block; similarly, we can slide a window $[\underline{a_1,\ldots,a_k}]$ of a waning block one step to the right to get the equivalent window $[\underline{a_2,\ldots,a_k,a_1-n}]$. 

If $\preceq$ is a TITO with a unique waxing block and no waning blocks, then there is a unique bijection $\pi\colon\ZZ\to\ZZ$ such that the following conditions hold:
\begin{itemize}
\item For all $i,j\in\ZZ$, we have $i\preceq j$ if and only if $\pi^{-1}(i)\leq\pi^{-1}(j)$ (in $\ZZ$). 
\item We have $\sum_{i=1}^n\pi(i)=\binom{n+1}{2}$. 
\end{itemize}
Such a bijection $\pi$ is called an \dfn{affine permutation}. The set of affine permutations forms a group under composition, which is called the \dfn{affine symmetric group} and is denoted $\widetilde{\mathfrak{S}}_n$. 


Define a \dfn{reflection index} to be a pair $(a,b)$ of integers with $a<b$, considered up to simultaneous translation by multiples of $n$. For example, if $n=5$, then $(1,6)$, $(6,11)$, and $(11,16)$ are all the same reflection index, but $(1,6)$ and $(1,11)$ are different reflection indices. 

An \dfn{inversion} of a TITO $\preceq$ is a reflection index $(a,b)$ such that $a\succeq b$. We write $\Inv(\preceq)$ for the set of inversions of $\preceq$. We say a TITO is \dfn{real} if every block of size $1$ is waxing. Likewise, we say a TITO is \dfn{co-real} if every block of size $1$ is waning. 

\begin{definition}
Let $\CDO$ denote the collection of all TITOs, equipped with the partial order $\leq$ defined so that for all $\preceq_1,\preceq_2$ in $\CDO$, we have $\preceq_1\, \leq\, \preceq_2$ if and only if $\Inv(\preceq_1)\subseteq \Inv(\preceq_2)$. The poset $\CDO$ is called the (type $\tA_{n-1}$) \dfn{cyclic Dyer lattice}. The subposet of $\CDO$ consisting of real TITOs is called the (type $\tA_{n-1}$) \dfn{Dyer lattice} and is denoted $\DO$. 
\end{definition} 

There is a map $\pi^\downarrow_{\DO}\colon\CDO \to \DO$ sending a TITO $\preceq$ to the maximum real TITO bounded above by $\preceq$. Equivalently, $\pi^\downarrow_{\DO}(\preceq)$ is the real TITO obtained from $\preceq$ by reversing all waning blocks of size $1$ so that they become waxing blocks. For instance, if $\preceq$ is the TITO from \cref{ex:TITO}, which has window notation $[\underline{1}][3,2][\underline{4,5}]$, then $\pi^\downarrow_{\DO}(\preceq)$ has window notation $[1][3,2][\underline{4,5}]$. We also define $\pi_{\DO}^\uparrow(\preceq)$ to be the minimum co-real TITO bounded below by $\preceq$. Equivalently, $\pi_{\DO}^\uparrow(\preceq)$ is the co-real TITO obtained from $\preceq$ by reversing all waxing blocks of size $1$ so that they become waning blocks. 

\begin{proposition}[\cite{Barkley}]\label{prop:quotient_Dyer}
    The posets $\CDO$ and $\DO$ are completely semidistributive lattices. Moreover, the map $\pi^\downarrow_{\DO}\colon\CDO\to \DO$ is a surjective homomorphism of complete lattices. 
\end{proposition}

A \dfn{wall} of a TITO $\preceq$ is a reflection index $(a,b)$ such that $a\preceq b$ is a cover relation or $b\preceq a$ is a cover relation. We say a wall $(a,b)$ is an \dfn{upper wall} if $a\preceq b$; otherwise, we say it is a \dfn{lower wall}.

Given a wall $(a,b)$ of $\preceq$, there is a unique TITO $\flip_{a,b}(\preceq)$ such that the only reflection index in the symmetric difference $\Inv(\preceq) \triangle \Inv(\flip_{a,b}(\preceq))$ is $(a,b)$. If $(a,b)$ is an upper wall of $\preceq$, then $\flip_{a,b}(\preceq)$ covers $\preceq$ in the cyclic Dyer lattice, and if $(a,b)$ is a lower wall, then $\flip_{a,b}(\preceq)$ is covered by $\preceq$ in the cyclic Dyer lattice. The following converse to this is sometimes called Dyer's Conjecture~A.

\begin{proposition}[\cite{BarkleySpeyerCombinatorics, Barkley}]
    If $\preceq_1$ is covered by $\preceq_2$ in the Dyer lattice or cyclic Dyer lattice, then there is a unique upper wall $(a,b)$ of $\preceq_1$ such that $\flip_{a,b}(\preceq_1) = \preceq_2$.
\end{proposition}


Let us record two natural antiautomorphisms of $\CDO$. The first is the map $\Psi_{\leftrightarrow}$ that simply reverses each TITO. More precisely, for every TITO $\preceq$ and for all $a,b\in\ZZ$, we have $a\preceq b$ if and only if $b\,\Psi_{\leftrightarrow}(\preceq)\, a$. The second antiautomorphism, denoted $\Psi_{\updownarrow}$, is induced by the negation map on $\ZZ$. More precisely, for every TITO $\preceq$ and for all $a,b\in\ZZ$, we have $a\preceq b$ if and only if $-a\,\Psi_{\updownarrow}(\preceq)\, -b$. Note that the map $\Psi_{\leftrightarrow}\circ\Psi_{\updownarrow}=\Psi_{\updownarrow}\circ\Psi_{\leftrightarrow}$ is an automorphism of $\CDO$ that restricts to an automorphism of $\DO$.  

\subsection{Pattern-avoiding TITOs}
In this subsection, we introduce our first definitions of the cyclic and affine Tamari lattices, as subposets of the Dyer lattice. 

Let $\preceq$ be a total order on a set $X\subseteq \ZZ$. We say three integers $a<b<c$ in $X$ form a \dfn{$312$-pattern} (respectively, \dfn{$132$-pattern}) in $\preceq$ if $c\preceq a\preceq b$ (respectively, $a\preceq c\preceq b$). Furthermore, we say the integers $a<b<c$ form a \dfn{$\overline{31}2$-pattern} if they form a 312-pattern and $c \preceq a$ is a cover relation. We say $\preceq$ is \dfn{$312$-avoiding} (respectively, \dfn{$\overline{31}2$-avoiding}) if no three integers form a $312$-pattern (respectively, a $\overline{31}2$-pattern) in $\preceq$. Similarly, $\preceq$ is \dfn{$132$-avoiding} (respectively, $\overline{13}2$-avoiding) if no three integers form a $312$-pattern (respectively, a $\overline{13}2$-pattern) in $\preceq$.    


\begin{definition}\label{def:312_Tamaris}
    The \dfn{cyclic Tamari lattice}, denoted $\CTam_{312}$, is the subposet of $\CDO$ consisting of $312$-avoiding TITOs. The \dfn{affine Tamari lattice}, denoted $\ATam_{312}$, is the subposet of $\DO$ consisting of $312$-avoiding real TITOs. 
\end{definition}

We have yet to justify calling $\CTam_{312}$ and $\ATam_{312}$ \emph{lattices}. We will do so in \cref{thm:cyclic_quotient,thm:affine_quotient}. 

In analogy with \cref{def:312_Tamaris}, we write $\CTam_{132}$ and $\ATam_{132}$ for the set of $132$-avoiding TITOs and the set of $132$-avoiding co-real TITOs, respectively, under the Dyer order. We will see in \cref{cor:isomorphism,cor:affine_incarnations} that $\CTam_{132}$ and $\ATam_{312}$ are isomorphic to $\CTam_{312}$ and $\ATam_{132}$, respectively. 

The following lemma will be used to study TITOs using the machinery in \cite{Barkley}, for which compact TITOs play a pivotal role.

\begin{lemma}\label{lem:312iscompact}
    Every $312$-avoiding TITO is compact.
\end{lemma}
\begin{proof}
    Let $\preceq$ be a TITO with two waxing blocks $I_1 \prec I_2$. For any $c\in I_1$ and any small enough $a\in I_2$, the integers $a< a+n < c$ form a $312$-pattern. 
\end{proof}

We note one more fact, which shows that $312$-avoiding TITOs have simpler order type than more general TITOs.

\begin{lemma}\label{lem:two_blocks}
    Let $\preceq$ be a 
    $312$-avoiding TITO. Then $\preceq$ has at most two blocks. If $\preceq$ has two blocks, then the left block is waxing and the right block is waning. If $\preceq$ has a waning block, then the elements of the waning block appear in decreasing order. 
\end{lemma}
\begin{proof}
Suppose $I$ is a waning block of $\preceq$. If the elements of $I$ do not appear in decreasing order, then there is some cover relation $c\preceq a$ in $I$ with $c-a> n$. But $c\preceq c-n$, so in this case, the integers $a<c-n < c$ form a $\overline{31}2$-pattern, which is a contradiction. This shows that the elements of $I$ appear in decreasing order. We claim that $I$ cannot appear to the left of another block. To see this, suppose instead that there is a block $I'$ to the right of $I$, and let $b\in I'$. Since the elements of $I$ appear in decreasing order, there is in fact a unique cover relation $c\preceq a$ in $I$ with $a<b<c$; this is a contradiction. 

It follows from the preceding paragraph that there is at most one waning block and that, if such a waning block exists, it must appear to the right of all waxing blocks. Therefore, it follows quickly from \Cref{lem:312iscompact} that there must also be at most one waxing block. 
\end{proof}


\section{Binary Trees}\label{sec:trees}
We now aim to construct a poset on certain infinite binary trees that is isomorphic to the cyclic Tamari lattice; this will help us prove that the cyclic Tamari lattice is a sublattice and a quotient of the cyclic Dyer lattice. 

\subsection{Binary Trees} 

\begin{definition}
A \dfn{binary tree} is a (possibly infinite) plane tree $T$ together with the following data. Each node $v$ has most one \dfn{left child} and at most one \dfn{right child}. Each node $v$ is the child (either left or right, but not both) of at most one other node, which is called the \dfn{parent} of $v$ and denoted $\pr(v)$ if it exists. We say that a node is an \dfn{ancestor} of $v$ if it is of the form $\pr^k(v)$ for some integer $k\geq 0$ (in particular, each node is an ancestor of itself). If $x$ is an ancestor of $v$, then we call $v$ a \dfn{descendant} of $x$. Any two nodes must have a common ancestor. 
\end{definition}

In a binary tree $T$, the \dfn{left subtree} of a node $v$ is the (possibly empty) binary tree induced by the descendants of the left child of $v$. Similarly, the \dfn{right subtree} of $v$ is the binary tree induced by the descendants of the right child of $v$.  

We will need to consider three natural total orders on the node set of a binary tree. 
\begin{itemize}
\item The \dfn{in-order} is the unique total order such that for every node $v$, all of the nodes in the left subtree of $v$ precede $v$, which in turn precedes all of the nodes in the right subtree of $v$. 
\item The \dfn{post-order} is the unique total order such that for every node $v$, all of the nodes in the left subtree of $v$ precede all of the nodes in the right subtree of $v$, which in turn precede $v$. 
\item The \dfn{reverse-post-order} is the unique total order such that for every node $v$, all of the nodes in the right subtree of $v$ precede all of the nodes in the left subtree of $v$, which in turn precede $v$. 
\end{itemize}

Let $X\subseteq\ZZ$. A \dfn{binary in-ordered tree} on $X$ is a binary tree $T$ with node set $\{\T_a\mid a\in X\}$ such that for all distinct $a,b\in\ZZ$, we have $a<b$ if and only if $\T_a$ precedes $\T_b$ in the in-order of $T$. 

\begin{definition}\label{def:TIBIT}
A \dfn{translation-invariant binary in-ordered tree} (TIBIT) is a binary in-ordered tree on $\ZZ$ such that for all $a,b\in\ZZ$, the node $\T_a$ is the left (respectively, right) child of $\T_b$ if and only if $\T_{a+n}$ is the left (respectively, right) child of $\T_{b+n}$. Let $\TIBIT$ denote the set of TIBITs. 
\end{definition} 



\cref{fig:one_tree,fig:two_trees} show examples of TIBITs. In each figure, we represent each node $\T_a$ as a circle with the number $a$ inside. \Cref{fig:two_trees} illustrates two different TIBITs whose underlying infinite binary trees are isomorphic. Note that a TIBIT can be recovered uniquely if we know its underlying infinite binary tree (up to isomorphism) and we know the node $\T_a$ for a single integer $a$. Define the \dfn{spine} of a TIBIT $T$ to be the set of nodes that have infinitely many descendants. Equivalently, $\T_a$ is in the spine if and only if $\T_{a+n}$ is an ancestor or descendant of $\T_a$. 

\begin{figure}[ht]
\begin{center}{\includegraphics[height=8.267cm]{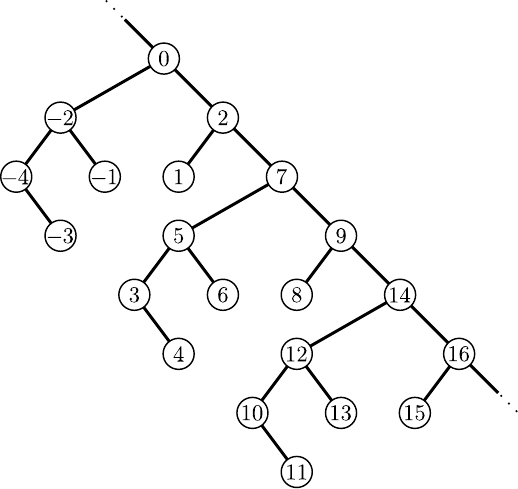}}
\end{center}
\caption{A translation-invariant binary in-ordered tree for $n=7$. } 
\label{fig:one_tree} 
\end{figure} 

\begin{figure}[ht]
\begin{center}{\includegraphics[height=6.267cm]{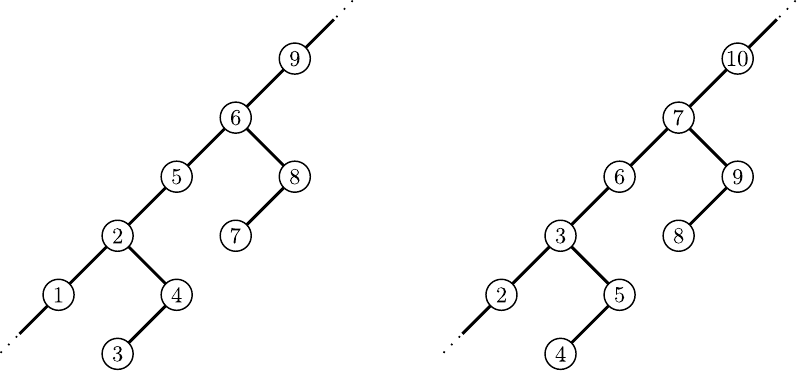}}
\end{center}
\caption{Two different translation-invariant binary in-ordered trees for $n=4$ whose underlying infinite binary trees are isomorphic.} 
\label{fig:two_trees} 
\end{figure}

The next lemma is immediate from the definitions. 

    \begin{lemma}\label{lem:TIBITstructure}
        Let $T$ be a binary in-ordered tree. If a node $\T_b$ has a right child $\T_c$, then the left subtree of $\T_c$ is exactly the set
        \[ \{\T_x \mid  b< x< c \}. \]
        Similarly, if $\T_b$ has a left child $\T_a$, then the right subtree of $\T_a$ is exactly the set
        \[ \{\T_x \mid  a < x < b \}. \]
    \end{lemma}

If $T$ is a finite binary tree and $X\subseteq \ZZ$ is equinumerous with the node set of $T$, then it follows from \cref{lem:TIBITstructure} that there is a unique binary in-ordered tree on $X$ that is isomorphic to $T$ as a binary tree. 

\subsection{The Cyclic Tamari Lattice via Binary Trees}\label{sec:CTam_Trees} 

Let $T$ be a binary in-ordered tree on a set $X\subseteq \ZZ$. We can naturally view $T$ as a partial order $\leq_T$ on $\{\T_a\mid a\in\ZZ\}$ such that $\T_a\leq_T \T_b$ if and only if $\T_b$ is an ancestor of $\T_a$. A \dfn{linear extension} of $T$ is a total order $\preceq$ on $X$ such that $a\preceq b$ whenever $\T_a$ is a child of $\T_b$. When $X$ is finite, it is well known that $T$ has a unique $312$-avoiding linear extension, which is obtained by reading the indices of the nodes in post-order. We wish to generalize this fact to the context of TIBITs. 

Define a \dfn{translation-invariant linear extension} (TILE) of a TIBIT $T$ to be a linear extension of $T$ that is also a TITO.\footnote{We will see in \cref{lem:TIBITcone} and its proof that certain cones associated to the TILEs of $T$ literally tile a cone naturally associated to $T$.} Let $\po_T$ and $\po^T$ be the TITOs obtained by reading the indices of the nodes of $T$ in post-order and in reverse-post-order, respectively. For example, if $T$ is the TIBIT shown in \cref{fig:one_tree}, then $\po_T$ has window notation $[1,4,3,6,5][\underline{7,2}]$, while $\po^T$ has window notation $[\underline{6,4,3,5,7,1,2}]$. 

There is a natural involution $\omm$ on the set of TIBITs that acts by flipping a tree across a vertical axis and then negating the indices of the nodes. More precisely, $v_a$ is a left child of $v_b$ in a TIBIT $T$ if and only if $v_{-a}$ is a right child of $v_{-b}$ in $\omm(T)$. Observe that 
\begin{equation}\label{eq:PR}
\po^T=\Psi_{\updownarrow}(\po_{\omm(T)}).
\end{equation}

    \begin{lemma}\label{lem:TIBIT312lift}
        Let $T$ be a TIBIT. Then $\po_T$ is the unique $312$-avoiding TILE of $T$, and $\po^T$ is the unique $132$-avoiding TILE of $T$.  
    \end{lemma}
    \begin{proof}
We prove the first statement; the second statement will then follow from \eqref{eq:PR}. It is straightforward to see that the linear extension of $T$ obtained by reading the indices of the nodes in post-order is indeed a $312$-avoiding TILE. Hence, we just need to prove uniqueness. 


Let $\preceq$ be a $312$-avoiding TILE of $T$. Suppose $a<c$; we will show that $c\preceq a$ if and only if $\T_c$ is in the right subtree of $\T_a$. Indeed, if $\T_c$ is in the right subtree of $\T_a$, then $c\preceq a$ by the definition of a linear extension. Conversely,
        assume that $\T_c$ is not in the right subtree of $\T_a$. Then $\T_c$ is not a descendant of $\T_a$ because $a<c$. If $\T_c$ is an ancestor of $\T_a$, then $a\preceq c$, as desired. Otherwise, let $\T_b$ be the lowest common ancestor of $\T_a$ and $\T_c$. Since $a<c$, it must be the case that $\T_a$ is in the left subtree of $\T_b$ and that $\T_c$ is in the right subtree of $\T_b$. In particular, we have that $a<b<c$, $a\preceq b$, and $c\preceq b$. Note that we cannot have $c\preceq a\preceq b$ since $a<b<c$ do not form a $312$-pattern in $\preceq$. It follows that $a\preceq c \preceq b$, as desired. This shows that $\preceq$ must be the TILE obtained by reading the indices of the nodes of $T$ in post-order. 
    \end{proof}


When $X\subseteq\ZZ$ is finite, it is well known that each total order $\preceq$ on $X$ is a linear extension of a unique binary in-ordered tree on $X$, which we denote by $\sT(\preceq)$. In fact, when $X=[n]$, we can naturally identify total orders on $[n]$ with elements of the symmetric group $\SSS_n$, and the map sending a total order $\preceq$ to the unique $312$-avoiding linear extension of $\sT(\preceq)$ is a quotient map from the weak order on $\SSS_n$ to the Tamari lattice. 
    
Let $\preceq$ be a TITO. We aim to construct a TIBIT $\sT(\preceq)$ of which $\preceq$ is a TILE; we will call $\sT(\preceq)$ the \dfn{binary insertion tree} of $\preceq$. Let $I$ be the rightmost block of $\preceq$. If $I$ is waxing, define $\spine(\preceq)$ to be the set of all integers $s\in I$ such that $x\geq s$ whenever $s\preceq x$. If $I$ is waning, define $\spine(\preceq)$ to be the set of all integers $s\in I$ such that $x\leq s$ whenever $s\preceq x$. The spine of $\sT(\preceq)$ will be $\{\T_s\mid s\in\spine(\preceq)\}$. For $s\in\spine(\preceq)$, let $\pr(s)$ be the unique integer such that $s\preceq\pr(s)$ is a cover relation, and let $J_s$ be the set of integers that are strictly between $s$ and $\pr(s)$ in the usual order on $\ZZ$. Let $\preceq_s$ be the restriction of $\preceq$ to $J_s$. As mentioned above, $\preceq_s$ is a linear extension of a unique binary in-ordered tree $\sT(\preceq_s)$ on $J_s$. If $I$ is a waxing block, then for each $s\in\spine(\preceq)$, we make $\T_s$ the left child of $\T_{\pr(s)}$ and make $\sT(\preceq_s)$ the right subtree of $\T_{s}$. If $I$ is a waxing block, then for each $s\in\spine(\preceq)$, we make $\T_s$ the right child of $\T_{\pr(s)}$ and make $\sT(\preceq_s)$ the left subtree of $\T_{s}$.
    
\begin{example}
Let $n=7$, and let $\preceq$ be the TITO with window notation $[4,3,15][\underline{6}][\underline{9,5,7}]$. The rightmost block $I$ of $\preceq$ is waning, so $\spine(\preceq)$ is the set of integers $s\in I$ such that $x\leq s$ whenever $s\preceq x$. Thus, $\spine(\preceq)$ is the set of integers $s$ that are congruent to $0$ or $2$ modulo $7$. We have $\pr(7)=2$, so $J_7=\{3,4,5,6\}$. Now, $\preceq_7$ is the restriction of $\preceq$ to $J_7$, which is the total order $4\preceq_7 3\preceq_7 6\preceq_7 5$. The unique binary in-ordered tree on $J_7$ that has $\preceq_7$ as a linear extension is 
\[\sT(\preceq_7)=\begin{array}{l}\includegraphics[height=2.550cm]{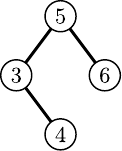}\end{array}\] (as before, $v_a$ is represented by a circle with $a$ inside). On the other hand, $\pr(9)=7$, so we have $J_9=\{8\}$ and \[\sT(\preceq_9)=\begin{array}{l}\includegraphics[height=0.550cm]{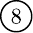}\end{array}.\] The TIBIT $\sT(\preceq)$ is shown in \cref{fig:one_tree}. 
\end{example}

\begin{lemma}\label{lem:insertionextension}
    If $\preceq$ is a TITO, then $\preceq$ is a linear extension of $\sT(\preceq)$. 
\end{lemma}
\begin{proof}
    By definition of $\sT(\preceq)$, the ordering $\preceq$ restricts to a linear extension of the spine of $\sT(\preceq)$ and to a linear extension on each subtree $\sT(\preceq_s)$ for $s\in\spine(\preceq)$. Hence, we just need to check that if $c\in \spine(\preceq)$ and $\T_c$ is an ancestor of $\T_b$ with $\T_b\in \sT(\preceq_a)$, then $b\preceq c$. We have that $a\preceq c$, so we just need to show that $b\preceq a$. We may also assume that $\T_b$ is the right child of $\T_a$.
    
    We assume that $a<b<c$; the case in which $c<b<a$ is similar. We assume to the contrary that $a\preceq b$. We cannot have $\pr(a) \preceq b$ since this would imply that $\pr(a)$ is not in the spine. Therefore, 
    $b\preceq \pr(a)$. We claim that $b\in\spine(\preceq)$. Indeed, if $b\preceq x$, then $a\preceq x$, so $a<x$. But if $a<x<b$, then $x\in \sT(\preceq_a)$, and then $\T_b$ is in the right subtree of $\T_x$, contradicting the fact that $\T_b$ is the right child of $\T_a$. Consequently, $b<x$. We find that $b\in\spine(\preceq)$, which is a contradiction. Thence, $b\preceq a$. 
\end{proof}

\begin{lemma}\label{lem:uniqueTIBIT}
    Let $\preceq$ be a TITO that is a linear extension of a TIBIT $T$. If $b$ and $c$ are integers with $b\preceq c$, then $\T_b$ is a descendant of $\T_c$ in $T$ if and only if there does not exist $x\succ c$ with ${\min_{<}(b,c) < x < \max_{<}(b,c)}$. Consequently, $\sT(\preceq)$ is the unique TIBIT admitting $\preceq$ as a linear extension.
\end{lemma}
\begin{proof}
We will assume that $b<c$; a similar argument handles the case when $c<b$. 
    
First assume that there exists an integer $x\succ c$ such that $b<x<c$. We may assume that $x$ is chosen maximally with respect to the total order $\preceq$. If $v_x$ and $v_b$ are incomparable in the partial order $\leq_T$, then they have a least common ancestor $v_y$. However, if this is the case, then we have $b<y<x<c$ and $x\prec y$, which contradicts the maximality of $x$. It follows that $v_x$ and $v_b$ are comparable in $\leq_T$; since $\preceq$ is a linear extension of $T$ and $b\preceq x$, the node $v_b$ must be a descendant of $v_x$. A similar argument shows that $v_c$ must be a descendant of $v_x$. Because $T$ is a TIBIT and $b<x<c$, we know that $v_b$ is in the left subtree of $v_x$ and that $v_c$ is in the right subtree of $v_x$. Hence, $v_b$ cannot be a descendant of $v_c$. 

To prove the converse, assume that $v_b$ is not a descendant of $v_c$. Let $v_x$ be the least common ancestor of $v_b$ and $v_c$. Then we have $b<x<c$ and $x\succ x$. 

Now, since a binary in-ordered tree $T'$ is uniquely determined by the partial order $\leq_{T'}$, it follows that $\preceq$ is a linear extension of at most one TIBIT. Therefore, the final statement of the lemma follows from \Cref{lem:insertionextension}. 
\end{proof}

The next proposition is immediate from \cref{lem:TIBIT312lift,lem:insertionextension,lem:uniqueTIBIT}. 

\begin{proposition}\label{prop:312TIBIT_bijection}
The map $\sT$ restricts to a bijection from $\CTam_{312}$ to $\TIBIT$ whose inverse is the map $T\mapsto\,\po_T$. 
\end{proposition}

\cref{prop:312TIBIT_bijection} allows us to transfer the Dyer order from $312$-avoiding TITOs to TIBITs. Thus, given two TIBITs $T_1,T_2$, we write $T_1\leq T_2$ if $\po_{T_1}\,\leq\,\po_{T_2}$ (in the Dyer order). 

\begin{lemma}\label{lem:interval} 
    Let $T$ be a TIBIT. For every TILE $\preceq$ of $T$, we have $\po_T\,\leq\,\preceq\,\leq\,\po^T$ in the Dyer order. 
    \end{lemma}
\begin{proof}
    We show that $\po_T\,\leq\, \preceq$; the proof that $\preceq\,\leq\, \po^T$ is similar. Let $a<b$ be integers, and assume that $b\po_T a$. We wish to show that $b\preceq a$. In fact, we will show that $\T_b$ is a descendant of $\T_a$. To see this, assume otherwise. Then the lowest common ancestor of $\T_a$ and $\T_b$ is some node $\T_x$, where $x$ is distinct from $a$ and $b$. The node $\T_a$ is in the left subtree of $\T_x$, while the node $\T_b$ is in the right subtree of $\T_x$. Hence, by reading the indices of the nodes of $T$ in post-order, we find that $a\po_T b \po_T x$, which is a contradiction. 
\end{proof} 

An \dfn{inversion} of a TIBIT $T$ is a reflection index $(a,b)$ such that $\T_b$ is a descendant of $\T_a$ in $T$. Likewise, a \dfn{version} of $T$ is a reflection index $(a,b)$ such that $\T_a$ is a descendant of $\T_b$ in $T$. 

\begin{proposition}\label{prop:6equivalent} 
    Let $T_1$ and $T_2$ be two TIBITs. The following are equivalent.
    \begin{enumerate}[\normalfont(a)]
        \item\label{item:a} $T_1\leq T_2$.
        \item\label{item:a'} $\omm(T_2)\leq\omm(T_1)$. 
        \item\label{item:b} Every inversion of $T_1$ is an inversion of $T_2$.
        \item\label{item:c} Every version of $T_2$ is a version of $T_1$.  
        \item\label{item:d} For every TILE $\preceq_2$ of $T_2$, we have $\po_{T_1}\,\leq\,\preceq_2$ in the Dyer order.
        \item\label{item:e} For every TILE $\preceq_1$ of $T_1$, we have $\preceq_1\,\leq\, \po^{T_2}$ in the Dyer order. 
    \end{enumerate}
\end{proposition}
\begin{proof}
For every TIBIT $T$, the set $\Inv(\po_T)$ of inversions of $\po_T$ is equal to the set of inversions of $T$. This immediately implies the equivalence of items \eqref{item:a} and \eqref{item:b}. Similarly, the inversion set $\Inv(\po_{\omm(T)})$ is equal to the set of pairs $(-b,-a)$ such that $(a,b)$ is a version of $T$. This shows that \eqref{item:a'} is equivalent to \eqref{item:c}. The equivalence of \eqref{item:a} and \eqref{item:d} follows directly from \cref{lem:interval}. Since $\Psi_{\updownarrow}$ is an antiautomorphism of $\CDO$, we know that \eqref{item:a'} holds if and only if $\Psi_{\updownarrow}(\po_{\omm(T_1)})\leq\Psi_{\updownarrow}(\po_{\omm(T_2)})$; by \eqref{eq:PR}, this occurs if and only if $\po^{T_1}\,\leq\,\po^{T_2}$. Using \cref{lem:interval}, we deduce that \eqref{item:a'} and \eqref{item:e} are equivalent. 

To complete the proof, we will show that \eqref{item:b}, and \eqref{item:c} are equivalent. We will prove that \eqref{item:b} implies \eqref{item:c}; the proof of the reverse implication is similar. Suppose \eqref{item:b} holds, and let $a<b$ be such that $v_a\leq_{T_2}v_b$. We want to show that $v_a\leq_{T_1}v_b$. Suppose instead that $v_a\not\leq_{T_1}v_b$. We know by (the contrapositive of) \eqref{item:b} that $v_b\not\leq_{T_1}v_a$, so the least common ancestor of $v_a$ and $v_b$ in $T_1$ is some node $v_x$ with $a<x<b$. Then $v_b\leq_{T_1} v_x$, so we know by \eqref{item:b} that $v_b\leq_{T_2} v_x$. This implies that $v_b$ is in the right subtree of $v_x$ in $T_2$, so $v_a$ must also be in the right subtree of $v_x$ in $T_2$. But this is impossible because $a<x$. 
\end{proof}

%

\begin{figure}[]
\begin{center}{\includegraphics[height=21cm]{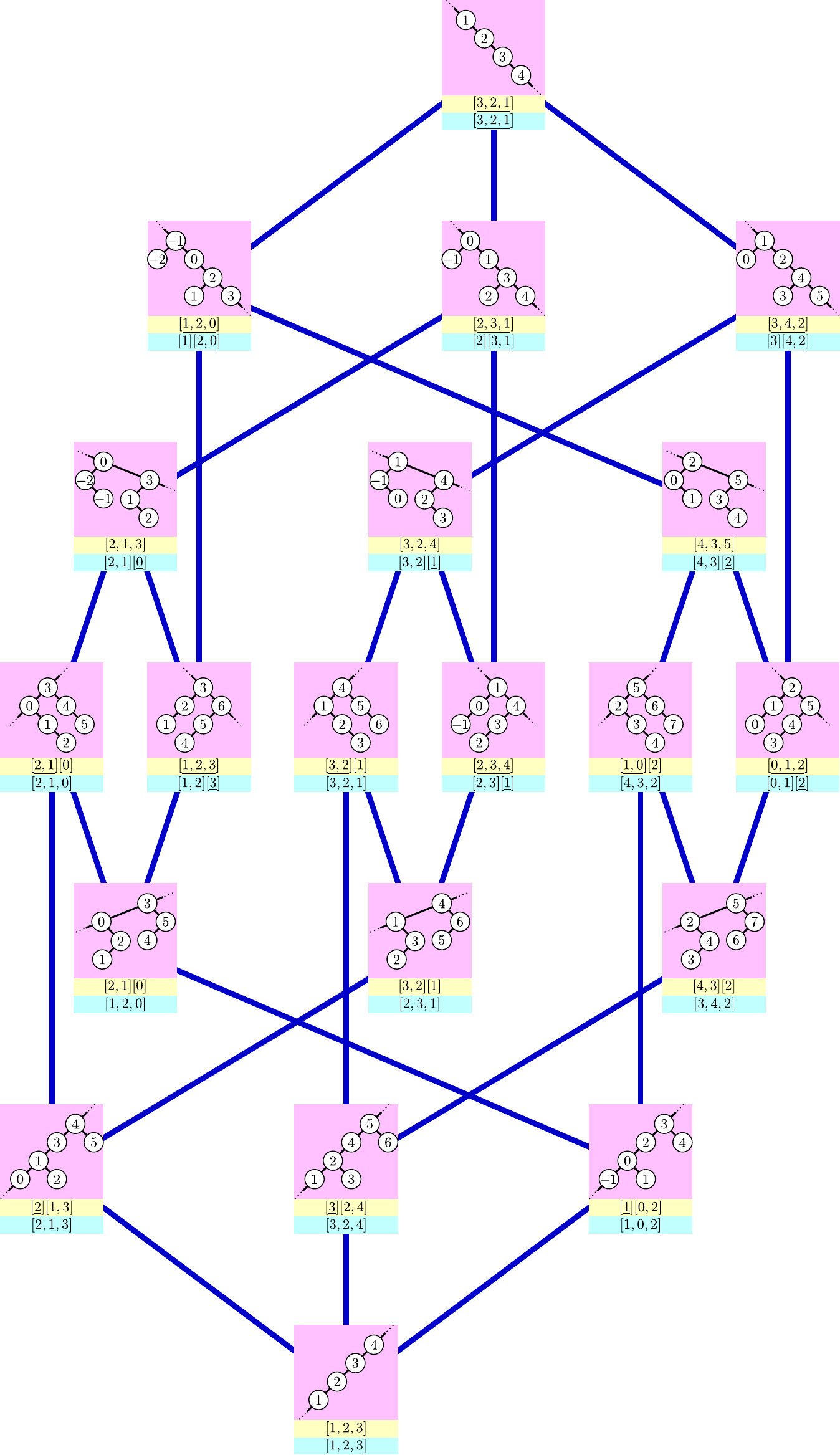}}
\end{center}
\caption{The cyclic Tamari lattice for $n=3$. Each element is represented by a box containing a TIBIT $T$ (on top), the corresponding $312$-avoiding TITO $\po_T$ (on bottom), and the corresponding $132$-avoiding TITO $\po^T$ (in the middle). } 
\label{fig:CTam3_TIBITs} 
\end{figure} 

\begin{corollary}\label{cor:isomorphism}
The posets $\CTam_{312}$, $\CTam_{132}$, and $\TIBIT$ are finite and self-dual, and they are pairwise-isomorphic. 
\end{corollary}

\begin{proof}
The Dyer order on $\TIBIT$ was defined precisely so that the map $T\mapsto\,\po_T$ is an isomorphism from $\TIBIT$ to $\CTam_{312}$. \cref{prop:6equivalent} implies that the map $T\mapsto\,\po^T$ is an isomorphism from $\TIBIT$ to $\CTam_{132}$; it also implies that $\omm$ is an antiautomorphism of $\TIBIT$. To see that $\TIBIT$ is finite, note that there are only finitely many ways to choose the spine of a TIBIT; once the spine is chosen, there are only finitely many ways to choose the finite subtrees of nodes in the spine. 
\end{proof}

While we originally defined the \emph{cyclic Tamari lattice} to be $\CTam_{312}$, \cref{cor:isomorphism} allows us to alternatively view $\CTam_{132}$ and $\TIBIT$ as other ``incarnations'' of the cyclic Tamari lattice. In some contexts, when the particular incarnation is not important, we will simply write $\CTam$ for the cyclic Tamari lattice considered as an abstract poset up to isomorphism.  

Define the \dfn{cyclic sylvester congruence} to be the equivalence relation on $\CDO$ in which two TITOs $\preceq_1$ and $\preceq_2$ are equivalent if and only if $\sT(\preceq_1)=\sT(\preceq_2)$. Given a TITO $\preceq$, let ${\pi_\syl^\downarrow(\preceq)=\,\po_{\sT(\preceq)}}$ and $\pi_\syl^\uparrow(\preceq)=\,\po^{\sT(\preceq)}$. 

\begin{theorem}\label{thm:cyclic_quotient} 
    The cyclic Tamari lattice is a finite lattice. Moreover, the cyclic sylvester congruence is a complete lattice congruence, and the map $\sT\colon\CDO\to\TIBIT$ is a surjective homomorphism of complete lattices. 
\end{theorem}
\begin{proof}
The finiteness was already proven in \cref{cor:isomorphism}. 

Let us show that $\sT$ is order-preserving. Suppose $\preceq_1$ and $\preceq_2$ are TITOs such that $\preceq_1\,\leq\,\preceq_2$ in $\CDO$. We wish to show that $\sT(\preceq_1)\leq\sT(\preceq_2)$. We will use the equivalence of \eqref{item:a} and \eqref{item:b} in \cref{prop:6equivalent}. Suppose $a<b$ are integers such that $v_b\leq_{\sT(\preceq_1)} v_a$; we aim to show that $v_b\leq_{\sT(\preceq_2)} v_a$. Consider an integer $x$ such that $a<x<b$. Since $v_b\leq_{\sT(\preceq_1)} v_a$ (and, therefore, $b\preceq_1 a$), we know by \cref{lem:uniqueTIBIT} that we cannot have $a\preceq_1 x$. This implies that $x\preceq_1 a$, so ${(a,x)\in\Inv(\po_{\sT(\preceq_1)})\subseteq\Inv(\po_{\sT(\preceq_2)})}$. It follows that $x\preceq_2 a$. As $x$ was arbitrary, we can use \cref{lem:uniqueTIBIT} once again to find that $v_b\leq_{\sT(\preceq_2)} v_a$, as desired. 

Now, because $\sT$ is order-preserving, it follows from \cref{prop:6equivalent} that the maps \[{\pi_\syl^\downarrow,\,\pi_\syl^\uparrow\colon\CDO\to\CDO}\] are also order-preserving. Moreover, \cref{prop:6equivalent} tells us that the cyclic sylvester class containing a TITO $\preceq$ is precisely the interval between $\pi_\syl^\downarrow(\preceq)$ and $\pi_\syl^\uparrow(\preceq)$ in $\CDO$. The desired result now follows from \cref{prop:9-5.2}. 
\end{proof} 

\cref{thm:cyclic_quotient} tells us that the cyclic Tamari lattice is a quotient of the cyclic Dyer lattice. It turns out that it is also a sublattice. 

\begin{theorem}\label{thm:cyclic_sublattice} 
The cyclic Tamari lattice $\CTam_{312}$ is a sublattice of the cyclic Dyer lattice $\CDO$. 
\end{theorem}
    
\begin{proof}
Let $X\subseteq\CTam_{312}$; we aim to show that $\bigvee X$ and $\bigwedge X$ are both in $\CTam_{312}$, where we write $\bigvee$ and $\bigwedge$ for the join and meet operations in $\CDO$. We already know by \cref{prop:join-subsemilattice} that $\bigvee X$ is in $\CTam_{312}$, so we just need to consider the meet. 


Define a relation $\preceq^*$ on $\ZZ$ as follows. For all $a\in\ZZ$, we have $a\preceq^* a$. For $a,b\in\ZZ$ with $a<b$, if $(a,b)\in \bigcap_{\preceq\,\in\, X}\Inv(\preceq)$, then $b\preceq^* a$; otherwise, $a\preceq^* b$. Note that $\preceq^*$ is antisymmetric. We claim that $\preceq^*$ is also transitive. Once this claim is established. it will follow immediately that $\preceq^*$ is a $312$-avoiding TITO whose inversion set is exactly $\bigcap_{\preceq\,\in\, X}\Inv(\preceq)$. But then it will follow that $\preceq^*$ must be equal to $\bigwedge X$, so we will deduce that $\bigwedge X\in\CTam_{312}$, as desired. 

To prove the claim, suppose $p,q,r$ are distinct integers such that $p\preceq^* q\preceq^*r$; we will show that $p\preceq^* r$. We consider four cases. 
\medskip

\noindent {\bf Case 1.} Assume $p<q$ and $p<r$. Since $p\preceq^* q$, there exists a $312$-avoiding TITO $\preceq^\#$ in $X$ such that $(p,q)\not\in\Inv(\preceq^\#)$. Thus, $p\preceq^\# q$. If $p<r<q$, then $(r,q)\in\bigcap_{\preceq\,\in\, X}\Inv(\preceq)\subseteq\Inv(\preceq^\#)$, so $p\preceq^\# q\preceq^\# r$. Now assume $p<q<r$. We must have $p\preceq^\# r$ since, otherwise, $p<q<r$ would form a $312$-pattern in $\preceq^\#$. This implies that $(p,r)\not\in\Inv(\preceq)$, so $p\preceq^* r$.  

\medskip 

\noindent {\bf Case 2.} Assume $r<p$ and $r<q$. We need to show that $(r,p)\in \bigcap_{\preceq\,\in\, X}\Inv(\preceq)$. If $r<q<p$, then $(q,p)$ and $(r,q)$ are both inversions of every TITO in $X$, so $(r,p)$ is also an inversion of every TITO in $X$. Now assume $r<p<q$. Consider a TITO $\preceq$ in $X$. We know that $(r,q)\in\Inv(\preceq)$. Since $r<p<q$ cannot form a $312$-pattern in $\preceq$, we must have $(r,p)\in\Inv(\preceq)$.  

\medskip 

\noindent {\bf Case 3.} Assume $q<p<r$. Since $q<r$ and $q\preceq^* r$, there must exist a TITO $\preceq'$ in $X$ such that $q\preceq' r$. Since $q<p$ and $p\preceq^* q$, we know that $(q,p)\in\Inv(\preceq')$. It follows that $p\preceq' r$, so $p\preceq^* r$. 

\medskip 

\noindent {\bf Case 4.} Assume $q<r<p$. Since $q<r$ and $q\preceq^* r$, there exists a TITO $\preceq''$ in $X$ such that $(q,r)\not\in\Inv(\preceq'')$. Since $q<p$ and $p\preceq^* q$, we must have $(q,p)\in\Inv(\preceq'')$. But then $q<r<p$ form a $312$-pattern in $\preceq''$, which is a contradiction. Hence, this case cannot occur. 
\end{proof} 

The elements covered by a TIBIT $T$ in $\TIBIT$ correspond to the lower walls of $\po_T$, which correspond to the right edges in $T$ modulo translation by $n$. Similarly, the elements covering $T$ in $\TIBIT$ correspond to the upper walls of $\po^T$, which correspond to the left edges in $T$ modulo translation by $n$. Since edges of $T$ up to translation are determined by their lower element modulo $n$, this implies the following proposition. 

\begin{proposition}\label{prop:regular_cyclic} 
The Hasse diagram of the cyclic Tamari lattice is an $n$-regular graph. 
\end{proposition}

\subsection{The Affine Tamari Lattice via Binary Trees}\label{subsec:Affine_TIBITs}

Our next goal is to interpret the affine Tamari lattice in terms of binary trees and prove that it is a quotient of the cyclic Tamari lattice. To this end, let us say a TIBIT $T$ is \dfn{real} if there does not exist an integer $k$ whose parent in $T$ is $k-n$. Let $\ATIBIT$ denote the subposet of $\TIBIT$ consisting of real TIBITs. The map $T\mapsto\po_T$ is an isomorphism from $\ATIBIT$ to $\ATam_{312}$; thus, we will sometimes refer to $\ATIBIT$ as the \emph{affine Tamari lattice}. Likewise, say $T$ is \dfn{co-real} if there does not exist an integer $k$ whose parent in $T$ is $k+n$, and let $\coATIBIT$ be the subposet of $\TIBIT$ consisting of co-real TITOs. Note that every TIBIT is either real or co-real. Moreover, a TIBIT is both real and co-real if and only if its spine contains more than one residue class modulo $n$. The map $T\mapsto\po^T$ is an isomorphism from $\coATIBIT$ to $\ATam_{132}$. We will soon see that $\coATIBIT$ and $\ATam_{132}$ are also isomorphic to the affine Tamari lattice. 

Let us define a map $\pispine^\downarrow\colon\TIBIT\to\ATIBIT$ as follows. If $T$ is a real TIBIT, simply let ${\pispine^\downarrow(T)=T}$. Now suppose $T$ is a TIBIT that is not real, and consider an integer $k$ whose parent is $k-n$. Then $k+n$ must be the right child of $k$; let $k'$ be the left child of $k$. We can obtain the TIBIT $\pispine^\downarrow(T)$ by ``flipping the spine'' of $T$. More precisely, if $j\not\equiv k\pmod n$, then the children of $j$ in $\pispine^\downarrow(T)$ are the same as the children of $j$ in $T$. On the other hand, if $j=k+mn$ for some integer $m$, then the left and right children of $j$ in $\pispine^\downarrow(T)$ are $j-n$ and $k'+(m+1)n$, respectively. See \cref{fig:pispine}.   

We also have a map $\pispine^\uparrow\colon\TIBIT\to\coATIBIT$ defined as follows. If $T$ is a co-real TIBIT, then $\pispine^\downarrow(T)=T$. Now suppose $T$ is a TIBIT that is not co-real, and consider an integer $k$ whose parent is $k+n$. Then $k-n$ must be the left child of $k$; let $k'$ be the right child of $k$. We obtain $\pispine^\uparrow(T)$ by ``flipping the spine'' of $T$. More precisely, if $j\not\equiv k\pmod n$, then the children of $j$ in $\pispine^\uparrow(T)$ are the same as the children of $j$ in $T$. On the other hand, if $j=k+mn$ for some integer $m$, then the left and right children of $j$ in $\pispine^\uparrow(T)$ are $k'+(m-1)n$ and $j+n$, respectively. See \cref{fig:pispine}. 
Alternatively, we have 
\begin{equation}\label{eq:pispineup}
\pispine^\uparrow=\omm\circ\pispine^\downarrow\circ\omm.
\end{equation} 

\begin{figure}[]
\begin{center}{\includegraphics[height=5.054cm]{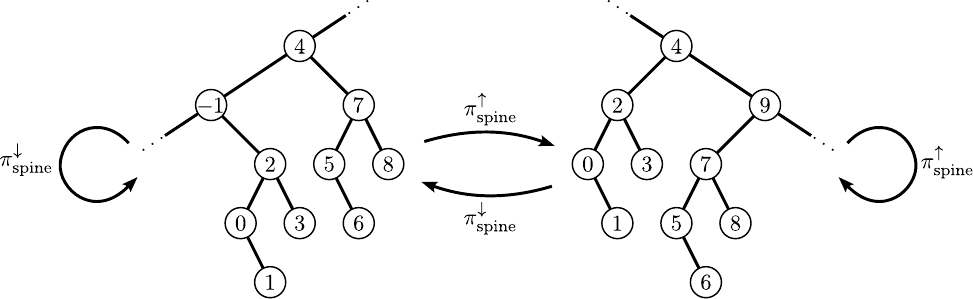}}
\end{center}
\caption{Applying $\pispine^\downarrow$ and $\pispine^\uparrow$ to two TIBITs for $n=5$. } 
\label{fig:pispine} 
\end{figure} 

For any TIBIT $T$, a reflection index $(a,b)$ is an inversion of $\pispine^\downarrow(T)$ if and only if $(a,b)$ is an inversion of $T$ and $a\not\equiv b\pmod n$. Similarly, $(a,b)$ is a version of $\pispine^\uparrow(T)$ if and only if $(a,b)$ is a version of $T$ and $a\not\equiv b\pmod n$. 

\begin{theorem}\label{thm:affine_quotient} 
The affine Tamari lattice is a finite lattice, and ${\pispine^\downarrow\colon\TIBIT\to\ATIBIT}$ is a surjective lattice homomorphism. 
\end{theorem} 

\begin{proof}
Consider the equivalence relation $\equiv_{\mathrm{spine}}$ on $\TIBIT$ in which $T\equiv_{\mathrm{spine}} T'$ if and only if $\pispine^\downarrow(T)=\pispine^\downarrow(T')$. Each equivalence class of this relation is an interval of cardinality $1$ or $2$; furthermore, for each TIBIT $T$, the minimum and maximum elements of the equivalence class containing $T$ are $\pispine^\downarrow(T)$ and $\pispine^\uparrow(T)$, respectively. The image of $\pispine^\downarrow$ is $\ATIBIT$, so by \cref{prop:9-5.2}, we just need to show that $\pispine^\downarrow$ and $\pispine^\uparrow$ are order-preserving. We will prove that $\pispine^\downarrow$ is order-preserving; since $\omm$ is an antiautomorphism of $\TIBIT$ (by \cref{prop:6equivalent}), it will then follow from \eqref{eq:pispineup} that $\pispine^\uparrow$ is also order-preserving. 

Let $T_1,T_2$ be TIBITs with $T_1 \leq T_2$. Suppose $(a,b)$ is an inverison of $\pispine^\downarrow(T_1)$. Then $(a,b)$ is an inversion of $T_1$, so it is also an inversion of $T_2$. Moreover, $a\not\equiv b\pmod n$, so $(a,b)$ is an inversion of $\pispine^\downarrow(T_2)$. This shows that $\pispine^\downarrow(T_1)\leq\pispine^\downarrow(T_2)$, which proves that $\pispine^\downarrow$ is order-preserving. 
\end{proof}

\begin{corollary}\label{cor:affine_incarnations}
The posets $\ATam_{312}$, $\ATam_{132}$, $\ATIBIT$, and $\coATIBIT$ finite and self-dual, and they are pairwise-isomorphic. 
\end{corollary}

\begin{proof}
We already know that $\ATam_{312}$ and $\ATIBIT$ are isomorphic and finite and that $\ATam_{132}$ and $\coATIBIT$ are isomorphic. It follows from the proof of \cref{thm:affine_quotient} that $\pispine^\downarrow$ restricts to an isomorphism from $\coATIBIT$ to $\ATIBIT$ whose inverse is the restriction of $\pispine^\uparrow$. The self-duality of $\ATIBIT$ follows from \eqref{eq:pispineup} and \cref{thm:affine_quotient}. 
\end{proof} 

\cref{fig:ATam3_TIBITs} shows the affine Tamari lattice for $n=3$. More accurately, this figure shows the four incarnations of the affine Tamari lattice listed in \cref{cor:affine_incarnations} superimposed on one another so that each element is represented as a $312$-avoiding real TITO, a $132$-avoiding co-real TITO, a real TIBIT, and a co-real TIBIT (the two TIBITs are drawn as one when they are equal). 

\begin{figure}[]
\begin{center}{\includegraphics[height=20.559cm]{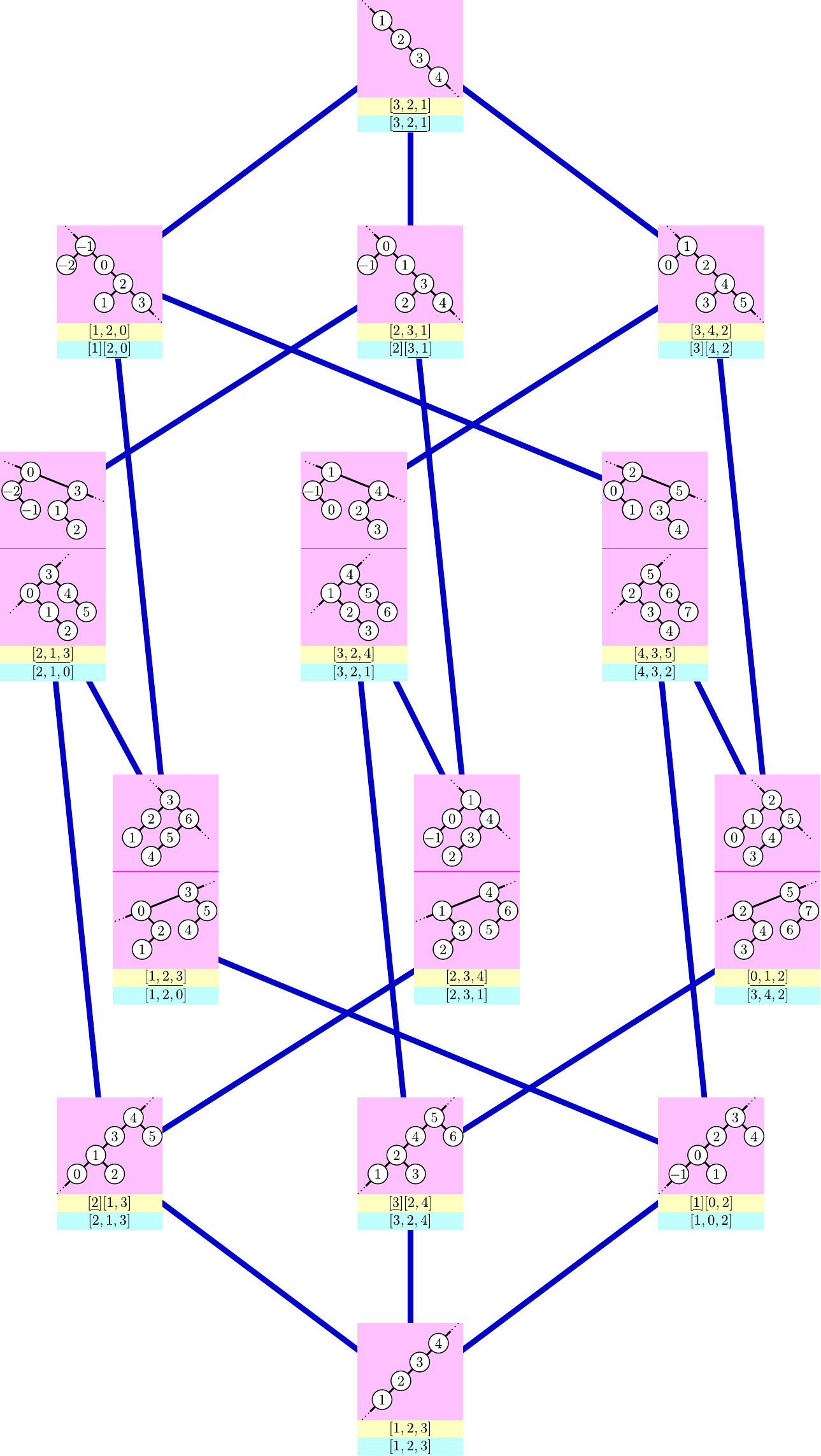}}
\end{center}
\caption{The affine Tamari lattice for $n=3$. Each element is represented by a box containing one or two TIBITs (on top), a $312$-avoiding real TITO (on the bottom), and a $132$-avoiding co-real TITO (in the middle). } 
\label{fig:ATam3_TIBITs} 
\end{figure} 

\begin{remark}
For $n\geq 3$, the affine Tamari lattice $\ATam_{312}$ is \emph{not} a sublattice of the cyclic Tamari lattice $\CTam_{312}$, the cyclic Dyer lattice $\CDO$, or the Dyer lattice $\DO$. For example, suppose $n=3$, and consider the TITOs $\preceq_1$ and $\preceq_2$ with window notations $[1][\underline{2,0}]$ and $[2][\underline{3,1}]$, respectively. In each of the lattices $\CTam_{312}$ and $\CDO$, the meet of $\preceq_1$ and $\preceq_2$ is the TITO with window notation $[1,2][\underline{3}]$, which is not real. The meet of $\preceq_1$ and $\preceq_2$ in $\DO$ is the TITO with window notation $[1,2][3]$, which is not $312$-avoiding. 
\end{remark} 

We end this section with the following companion to \cref{prop:regular_cyclic}. 

\begin{proposition}\label{prop:regular_affine} 
The Hasse diagram of the affine Tamari lattice is an $n$-regular graph. 
\end{proposition}
\begin{proof}
By \cref{thm:affine_quotient}, it suffices to show that for each TIBIT $T$, the number of elements covered by $\pispine^\downarrow(T)$ in $\TIBIT$ and the number of elements covering $\pispine^\uparrow(T)$ in $\TIBIT$ sum to $n$. By \cref{prop:regular_cyclic}, we just need to show that the TIBITs $\pispine^\downarrow(T)$ and $\pispine^\uparrow(T)$ cover the same number of elements in $\TIBIT$. This is obvious if $\pispine^\downarrow(T)$ and $\pispine^\uparrow(T)$ are equal, so assume they are not. This assumption implies that the set of nodes in the spine of $\pispine^\downarrow(T)$, which equals the set of nodes in the spine of $\pispine^\uparrow(T)$, is of the form $\{v_{a}\mid a\equiv k\pmod{n}\}$ for some $k$. We must show that $\pispine^\downarrow(T)$ and $\pispine^\uparrow(T)$ have the same number of right edges modulo translation by $n$. This follows from two observations. First, each node $v_b$ for $b\not\equiv k\pmod{n}$ has the same set of children in $\pispine^\downarrow(T)$ as in $\pispine^\uparrow(T)$ (and each such child is either a left child in both trees or a right child in both trees). Second, each node $v_a$ with $a\equiv k\pmod{n}$ has one left child and one right child in each of the trees $\pispine^\downarrow(T)$ and $\pispine^\uparrow(T)$. 
\end{proof}

\section{Arc Diagrams and the Fundamental Theorem of Finite Semidistributive Lattices}\label{sec:Arcs_FTFSDL} 

Consider an annulus whose outer boundary has $n$ marked points numbered $1,\ldots,n$ in clockwise order. We will think of the marked points as representing elements of $\ZZ/n\ZZ$ (so each label really represents a residue class modulo $n$). An \dfn{arc} is a clockwise-directed curve inside the annulus that starts at one marked boundary point and ends at another marked boundary point (which could be the same as the starting point) and does not intersect itself in its relative interior. Arcs are considered up to isotopy, so they are in one-to-one correspondence with reflection indices $(a,b)$ such that $b\leq a+n$. The arc corresponding to a reflection index $(a,b)$, which we denote by $\gamma_{a,b}$, has starting point $a$ and final point $b$. If $b=a+n$, we call $\gamma_{a,b}$ an \dfn{imaginary arc}; otherwise, $\gamma_{a,b}$ is a \dfn{real arc}. A \dfn{noncrossing arc diagram} is a collection of arcs that can be drawn so that they do not cross in their relative interiors. 

Let $\preceq$ be a $312$-avoiding TITO. Recall that a \emph{lower wall} of $\preceq$ is a reflection index $(a,b)$ such that $b\preceq a$ is a cover relation. The \dfn{arc diagram} of $\preceq$, denoted $\cA(\preceq)$, is the set of arcs $\gamma_{a,b}$ such that $(a,b)$ is a lower wall of $\preceq$.

\begin{figure}[ht]
\begin{center}{\includegraphics[height=6.960cm]{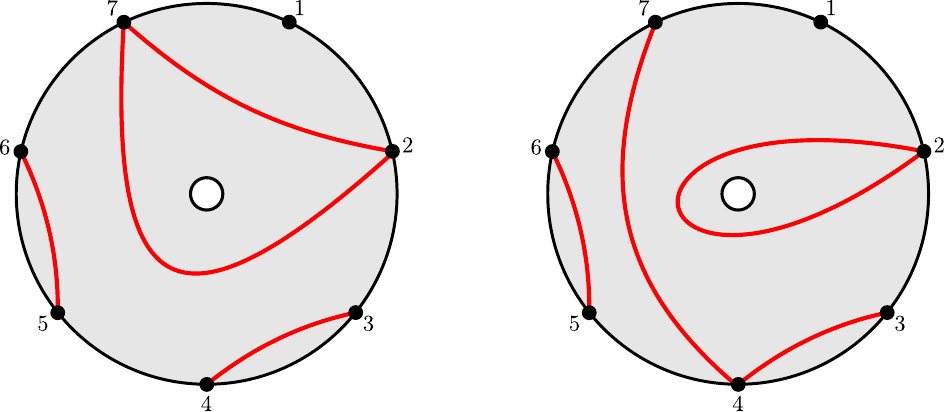}}
\end{center}
\caption{The arc diagrams of the $312$-avoiding TITOs with window notations $[1,4,3,6,5][\underline{7,2}]$ (left) and $[1,6,5,7,4,3][\underline{2}]$ (right).} 
\label{fig:arc_diagrams} 
\end{figure} 

\begin{proposition}\label{lem:arcdiagram312}
    The map $\cA$ is a bijection from the set of $312$-avoiding TITOs to the set of noncrossing arc diagrams. In particular, a $312$-avoiding TITO is determined by its collection of lower walls.
\end{proposition}
\begin{proof}
    In \cite{Barkley}, every TITO is given an associated arc diagram, using a more general notion of arc diagram than we have defined here. The TITOs with arc diagrams matching our definition are exactly the $\overline{31}2$-avoiding TITOs. Since each $312$-avoiding TITO is in particular $\overline{31}2$-avoiding, the diagram of a $312$-avoiding TITO $\preceq$ is the diagram $\cA(\preceq)$ we define above. In the general context of \cite{Barkley}, taking the diagram of a TITO is a many-to-one map from TITOs to noncrossing arc diagrams. However, \cite[Theorem 4.12]{Barkley} asserts that there is a unique compact TITO with any given noncrossing arc diagram. By \Cref{lem:312iscompact}, every $312$-avoiding TITO is compact, so this finishes the proof.
\end{proof}

For the next corollary, we recall that the Catalan numbers of types $B_n$ and $D_n$ are, respectively, 
\[\mathrm{Cat}_{B_n}=\binom{2n}{n}\quad\text{and}\quad\mathrm{Cat}_{D_n}=\frac{3n-2}{n}\binom{2n-2}{n-1}.\]  

\begin{corollary}\label{cor:Catalan} 
    We have
    \[ |\CTam| = \mathrm{Cat}_{B_n} \quad\text{and}\quad |\ATam| = \mathrm{Cat}_{D_n}. \]
\end{corollary}
\begin{proof}
     Let $\fD$ be the set of noncrossing arc diagrams with no imaginary arcs. It follows from \cref{lem:arcdiagram312} that $|\fD|=|\ATam|$.
    We will repeatedly use the fact that if $i<j$ and $j-i\leq n-1$, then the noncrossing arc diagrams using only arcs of the form $\gamma_{a,b}$ with $i\leq a<b\leq j$ are in natural bijection with noncrossing partitions of the set $\{i,i+1,\ldots,j\}$ (as defined in \cref{subsec:rowmotion_semidistributive}). Thus, the number of such noncrossing arc diagrams is $\mathrm{Cat}_{A_{j-i-1}}$. If $j-i=n-1$, then these noncrossing arc diagrams are also in bijection with noncrossing arc diagrams that use the arc $\gamma_{i-1,i-1+n}$. By \cref{lem:arcdiagram312}, this implies that
    \[ |\CTam|=|\ATam| + n \mathrm{Cat}_{A_{n-2}}, \]
    so if we can show that $|\ATam| = \mathrm{Cat}_{D_n}$, then it will follow that $|\CTam| = \mathrm{Cat}_{B_n}$. Hence, we must show that $|\fD|=\mathrm{Cat}_{D_n}$. Given a noncrossing arc diagram in $\fD$, either there is no arc of the form $\gamma_{i,j}$ with $i\leq n<j$, or there is an arc $\gamma_{j,i+n}$ with $1\leq i<j\leq n$. This implies that \[|\fD|=\mathrm{Cat}_{A_{n-1}} + \sum_{1\leq i < j\leq n}\mathrm{Cat}_{A_{j-i}}\mathrm{Cat}_{A_{n+i-j-2}}, 
    \]
    and it is straightforward to show (e.g., using generating functions) that this last expression is exactly $\mathrm{Cat}_{D_{n}}$.
\end{proof}

We now describe the map $\cA^{-1}$ sending a noncrossing arc diagram to a $312$-avoiding TITO. From the proof of \Cref{lem:arcdiagram312}, it follows that if the union of two noncrossing arc diagrams $D_1,D_2$ is also noncrossing, then $\cA^{-1}(D_1\cup D_2) = \cA^{-1}(D_1)\vee \cA^{-1}(D_2)$. 
One can then work out the value of $\cA^{-1}$ on any diagram from its values on diagrams with only one arc. The TITO $\cA^{-1}(\{\gamma_{a,b}\})$ is the unique join-irreducible element $j_{a,b}$ of $\CTam_{312}$ with lower wall $(a,b)$. If $b<a+n$, then
\[ j_{a,b} \coloneqq [b+1-n,b+2-n,\ldots,\widehat{a},\ldots,b-1,b,a]\] (the hat diacritic represents omission).
If $b=a+n$, then
\[ j_{a,b} \coloneqq [a-n+1,\ldots,a-1][\underline{a}]. \]

\begin{corollary}
    Let $\preceq$ be a 312-avoiding TITO, and let $\cA(\preceq) = \{ \gamma_{a_1,b_1},\ldots,\gamma_{a_k,b_k}\}$. Then
    \[ \preceq\, =\, j_{a_1,b_1} \vee \cdots \vee j_{a_k,b_k}, \]
    and $\{j_{a_1,b_1}, \ldots, j_{a_k,b_k}\}$ is the canonical join representation of $\preceq$.
\end{corollary}

We can also give a different description of $\cA^{-1}$ using TIBITs.

\begin{lemma}
    If $D$ is a noncrossing arc diagram, then $\sT\circ \cA^{-1}(D)$ is the unique TIBIT such that for all $a<b$, $\T_b$ is a right child of $\T_a$ if and only if $\gamma_{a,b}$ is in $D$.
\end{lemma}

We now will describe the lattice factorizations of $\CTam$ and $\ATam$ using the Fundamental Theorem of Finite Semidistributive Lattices. For each of these lattices, we need to describe the bijection $\kappa$ from the set of join-irreducibles to the set of meet-irreducibles, which sends $j$ to the maximum element $z$ satisfying $j\wedge z = j_*$. Because $\ATam$ is a quotient of $\CTam$ (by \cref{thm:affine_quotient}), the map $\kappa_{\ATam}$ is the restriction of $\kappa_{\CTam}$ to the set of join-irreducibles of $\ATam$. Thus, we will slightly abuse notation and denote both of these maps by $\kappa$. We will find it convenient to view the join-irreducibles as elements of $\CTam_{312}$ and the meet-irreducibles as elements of $\CTam_{132}$. There is a unique element $m_{a,b}$ of $\CTam_{132}$ whose unique upper wall is $(a,b)$. If $b<a+n$, let
\[ m_{a,b} \coloneqq [\underline{a+n-1,\ldots,\widehat{b},\ldots,a+1,a,b}], \]
and if $b = a+n$, let
\[ m_{a,b} \coloneqq [a][\underline{a+n-1,\ldots,a+1}]. \]
We then have that $\kappa(j_{a,b}) = m_{a,b}$. Furthermore, $j_{a,b}\leq j_{c,d}$ if and only if (after translating $(c,d)$ by a multiple of $(n,n)$) $a=c$ and $b\leq d$. Similarly, $m_{a,b}\leq m_{c,d}$ if and only if (after translating $(a,b)$ by a multiple of $(n,n)$) $b=d$ and $a\leq c$. We deduce the following.

\begin{theorem}\label{thm:FTFSDL_CTam}
    Let $(a,b)$ and $(c,d)$ be reflection indices such that $b-a\leq n$ and $d-c\leq n$. In the FTFSDL factorization for $\CTam$, we have
    \begin{itemize}
        \item $j_{a,b} \twoheadleftarrow j_{c,d}$ if and only if the integers $a,b,c,d$ can be chosen so that $a=c$ and $b\leq d$; 
        \item $j_{a,b} \hookrightarrow j_{c,d}$ if and only if the integers $a,b,c,d$ can be chosen so that $b=d$ and $a\geq c$;
        \item $j_{a,b} \rightarrow j_{c,d}$ if and only if the integers $a,b,c,d$ can be chosen so that $c\leq a < d \leq b$.
    \end{itemize}
    Furthermore, the restriction of these relations to join-irreducibles in $\ATam$ gives the FTFSDL factorization of $\ATam$. 
\end{theorem}

\begin{figure}[ht]
\begin{center}{\includegraphics[height=3.465cm]{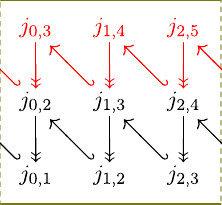}}\qquad\qquad{\includegraphics[height=3.465cm]{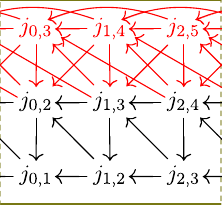}}
\end{center}
\caption{For $n=3$, each image shows a cylinder containing the join-irreducible elements of $\CTam$. The left image shows the non-loop arrows of the form $\twoheadrightarrow$ and $\hookrightarrow$ in the FTFSDL factorization, while the right image shows all non-loop arrows of the form $\to$. Deleting the parts of the images in {\color{red}red} yields the FTFSDL factorization of $\ATam$.}  
\label{fig:FTFSDL1} 
\end{figure} 

According to \cref{thm:FTFSDL_CTam}, the lattice $\CTam$ is isomorphic to the lattice of pairs $(\cT, \cF)$ with ${\cT,\cF\subseteq \mathrm{JIrr}(\CTam)}$ satisfying
\[ \cT = \{j \in \mathrm{JIrr}(\CTam)\mid j\not\to j'\text{ for all }j'\in\cF\}\quad \text{and}\quad \cF = \{j \in \mathrm{JIrr}(\CTam)\mid  j'\not\to j\text{ for all }j'\in\cT\}, \] ordered by containment of left sets (or reverse containment of right sets).  
It remains to describe the pair $(\cT,\cF)$ associated to an element of $\CTam$. 

\begin{lemma}\label{lem:JIeqinv}
Let $\preceq$ be a $312$-avoiding TITO. Then $j_{a,b}\leq\, \preceq$ if and only if $(a,b)\in \Inv(\preceq)$. Similarly, if $\preceq'$ is $132$-avoiding, then $m_{a,b}\geq\, \preceq'$ if and only if $(a,b)\not\in \Inv(\preceq')$.   
\end{lemma}
\begin{proof}
    We prove the statement about $312$-avoiding TITOs; the other statement is similar.
    Evidently, if $j_{a,b} \leq\, \preceq$, then $(a,b)\in \Inv(j_{a,b})\subseteq \Inv(\preceq)$. Conversely, assume $(a,b)\in \Inv(\preceq)$. If $a<x<b$, then also $(a,x)\in \Inv(\preceq)$ since otherwise the integers $a<x<b$ would form a $312$-pattern in $\preceq$. The desired result follows by observing that $j_{a,b}$ is the unique minimal TITO with ${(a,a+1),(a,a+2),\ldots, (a,b-1),(a,b)}$ in its inversion set. 
\end{proof}

In the following corollary, recall the maps $\pi_\syl^\downarrow\colon\CDO\to\CTam_{312}$ and $\pi_\syl^\uparrow\colon\CDO\to\CTam_{132}$ associated to the cyclic sylvester congruence, as defined in \cref{sec:CTam_Trees}. 

\begin{corollary}
    Under the FTFSDL for $\CTam_{312}$, the orthogonal pair $(\cT,\cF)$ corresponding to a $312$-avoiding TITO $\preceq$ is given by 
    \[ \cT = \{ j_{a,b} \mid (a,b)\in \Inv(\pi_{\syl}^\downarrow(\preceq)) \} \quad\text{and}\quad \cF = \{ j_{a,b} \mid (a,b) \not\in \Inv(\pi_\syl^{\uparrow}(\preceq)) \}.   \]
Under the FTFSDL for $\ATam_{312}$, the orthogonal pair $(\cT,\cF)$ corresponding to a $312$-avoiding real TITO $\preceq$ is given by 
    \[ \cT = \{ j_{a,b} \mid (a,b)\in \Inv(\pi_{\DO}^\downarrow(\pi_{\syl}^\downarrow(\preceq))) \} \quad\text{and}\quad \cF = \{ j_{a,b} \mid (a,b) \not\in \Inv(\pi_{\DO}^\uparrow(\pi_{\syl}^{\uparrow}(\preceq))) \}.   \]
\end{corollary}  

Using arcs, we can describe intrinsically the sets $\cT$ that arise as the left sets in orthogonal pairs. 

\begin{definition}
    A \dfn{cyclic arc torsion class} is a set $D$ of arcs such that
    \begin{itemize}
        \item if $\gamma_{a,c}\in D$ and $a<b<c$, then $\gamma_{a,b}\in D$;
        \item if $\gamma_{a,b},\gamma_{b,c}\in D$ and $c-a\leq n$, then $\gamma_{a,c}\in D$.
    \end{itemize}
    An \dfn{affine arc torsion class} is a set $D$ of real arcs such that
    \begin{itemize}
        \item if $\gamma_{a,c}\in D$ and $a<b<c$, then $\gamma_{a,b}\in D$;
        \item if $\gamma_{a,b},\gamma_{b,c}\in D$ and $c-a<n$, then $\gamma_{a,c}\in D$.
    \end{itemize}
\end{definition} 
Given an arbitrary set $D$ of arcs, we denote by $\overline{D}$ the cyclic arc torsion class generated by $D$, which is simply the smallest cyclic arc torsion class containing $D$ (see \cref{fig:arc_torsion_classes}). The following proposition implies that the map $D\mapsto \{j_{a,b} \mid \gamma_{a,b}\in D\}$ is a bijection from the collection of cyclic arc torsion classes to the collection of left sets in orthogonal pairs. One can prove this proposition directly, but we defer to the next section to give an algebraic proof.


\begin{proposition} \label{prop:noncrosstoclass}
    The map $D\mapsto \overline{D}$ is a bijection from the set of noncrossing arc diagrams to the set of cyclic arc torsion classes. If $\preceq$ is a 312-avoiding TITO and $D=\cA(\preceq)$, then 
    \[ \overline{D} =  \{\gamma_{a,b} \mid (a,b)\in \Inv(\preceq) \}. \]
\end{proposition} 

\begin{figure}[ht]
\begin{center}{\includegraphics[height=6.960cm]{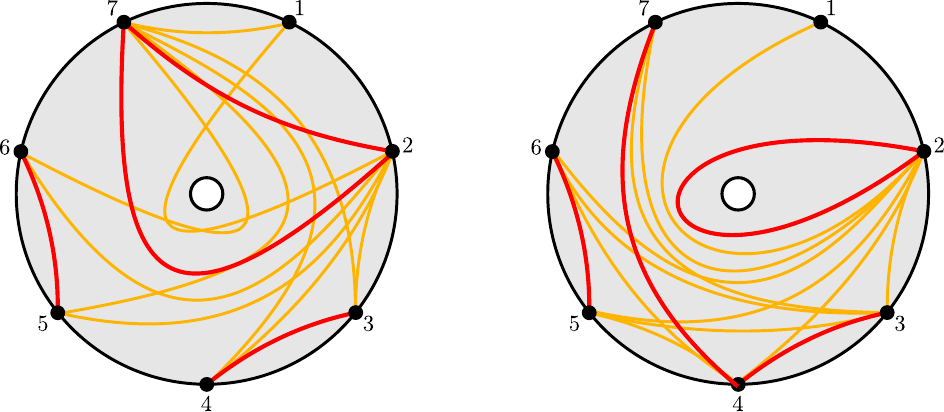}}
\end{center}
\caption{On the left is the affine arc torsion class generated by the noncrossing arc diagram on the left of \cref{fig:arc_diagrams}; it is not a cyclic arc torsion class. On the right is the cyclic arc torsion class generated by the noncrossing arc diagram on the right of \cref{fig:arc_diagrams}; it is not an affine arc torsion class. Each arc torsion class is obtained by adding extra arcs (drawn in {\color{MyOrange}orange}) to the corresponding noncrossing arc diagram.} 
\label{fig:arc_torsion_classes} 
\end{figure} 

\begin{remark}\label{rem:arc_torsion_containment} 
It follows from \cref{prop:noncrosstoclass} that the set of cyclic arc torsion classes, when partially ordered by containment, is isomorphic to $\CTam$. In addition, the map \[\preceq\,\mapsto\{\gamma_{a,b}\mid(a,b)\in\Inv(\preceq),\, b-a<n\}\] is a bijection from $\ATam_{312}$ to the set of affine arc torsion classes. It follows that the set of affine arc torsion classes, when partially ordered by containment, is isomorphic to $\ATam$. 
\end{remark}

\section{Quiver Representations}\label{sec:quivers}

Let $\cQ$ denote the oriented cycle quiver with node set $\ZZ/n\ZZ$ and arrow set ${\{i\to i+1\mid i\in\ZZ/n\ZZ\}}$. That is, 
\[\cQ=\begin{array}{l}\includegraphics[height=2.936cm]{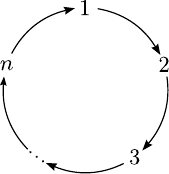}\end{array}.\] 

Given integers $a\leq b$, considered up to simultaneous translation by a multiple of $n$, we define the \dfn{path} $\alpha_{a,b}$ to be the sequence $a\to a+1 \to \cdots \to b-1 \to b$ of edges of $\cQ$. The \dfn{path algebra} of $\cQ$, denoted $\CC[\cQ]$, is a (non-commutative) algebra with linear basis given by the paths in $\cQ$. 
For each $a\in\ZZ/n\ZZ$, there is a \dfn{trivial path} $\alpha_{a,a}$. Multiplication in the path algebra is given by concatenation of paths:
\[ \alpha_{c,d}\cdot \alpha_{a,b} = \begin{cases}
    \alpha_{a,d} & \text{if $b\equiv c\pmod n$}; \\
    0 & \text{if $b\not\equiv c\pmod n$}.
\end{cases} \]

\begin{definition}
    The \dfn{cyclic Tamari algebra} is the quotient algebra
    \[ C_\cQ \coloneqq \CC[\cQ]/(\alpha_{a,a+n} \mid a \in [n]). \]
    The \dfn{affine Tamari algebra} is the quotient algebra
    \[ A_\cQ \coloneqq \CC[\cQ]/(\alpha_{a,a+n-1} \mid a \in [n]). \]
\end{definition}

It is well known that the affine Tamari algebra $A_\cQ$ is a cluster-tilted algebra of type $D$; in particular, the representation theory of $A_\cQ$ is closely tied to the combinatorics of type $D$ cluster algebras. A \dfn{representation} of $\CC[\cQ]$, $C_\cQ$, or $A_\cQ$ is a module that is finite-dimensional over $\CC$. A representation $M$ of $\CC[\cQ]$ is equivalent to the data of a vector space $M^i$ for each $i\in \ZZ/n\ZZ$ and a linear map $L^i:M^i\to M^{i+1}$ for each $i\in\ZZ/n\ZZ$. Moreover, $M$ is a representation of $C_\cQ$ if it additionally satisfies $L^{i-1}\cdots L^{i+1}L^i=0$ for all $i\in \ZZ/n\ZZ$, and $M$ is a representation of $A_\cQ$ if it additionally satisfies $L^{i-2}L^{i-3}\cdots L^{i+1}L^i=0$ for all $i\in \ZZ/n\ZZ$. 

For $a\in[n]$ and $a<b\leq a+n$, let $M_{a,b}$ be the module of $\CC[\cQ]$ that assigns to each vertex $i$ the vector space 
\[M^i_{a,b}=\begin{cases}
    \CC & \text{if $a\leq i\leq b-1$}; \\
    0 & \text{otherwise}
    \end{cases}
    \]
and the identity map $L^i=\text{id}\colon\CC\to\CC$ for all $a\leq i\leq b-2$ (all other maps are necessarily $0$). (These are examples of \emph{string modules}; see, e.g., \cite[Section~2.3.1]{Laking}.) For example, when $n=5$, the module $M_{2,5}$ is 
\[\begin{array}{l}\includegraphics[height=4.597cm]{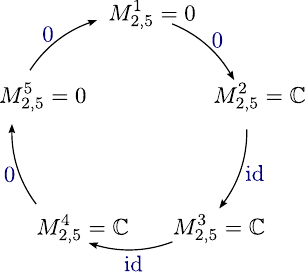}\end{array}.\] 
We also write $M_{a,b}=M_{\gamma_{a,b}}=M_{j_{a,b}}$, where $\gamma_{a,b}$ and $j_{a,b}$ are the arc and the join-irreducible element of $\CTam_{312}$ defined in \cref{sec:Arcs_FTFSDL}. 

Both $C_\cQ$ and $A_\cQ$ are \emph{Nakayama algebras}, which in particular implies that they have finitely many indecomposable representations up to isomorphism (they are \emph{representation-finite}). In fact, indecomposable representations of $C_\cQ$ are precisely the modules isomorphic to $M_{a,b}$ for $a\in[n]$ and $a<b\leq a+n$ \cite{Adachi2,Baur}. The indecomposable modules of $A_\cQ$ are precisely the modules isomorphic to $M_{a,b}$ for $a\in[n]$ and $a<b\leq a+n-1$. For $C_\cQ$ and $A_\cQ$, the indecomposable modules coincide with the \dfn{brick} modules, which are modules $M$ such that $\mathrm{End}(M)$ is a division algebra (since we are working over $\CC$, $\mathrm{End}(M)$ is a division algebra if and only if it is $\CC$). We write \[ \mathrm{Bricks}(C_\cQ) = \{M_{a,b} \mid a<b\leq a+n \} \text{ and } \mathrm{Bricks}(A_\cQ) = \{M_{a,b} \mid a < b < a+n \}. \]

A \dfn{torsion class} is a collection of modules closed under isomorphisms, quotients, and extensions. A set of modules $S$ \dfn{generates} a torsion class $\cT$ if $\cT$ is the minimum torsion class containing $S$. We define $\CTam_{\Tors}$ to be the poset of torsion classes for $C_\cQ$ under inclusion order. Likewise, $\ATam_{\Tors}$ is defined to be the poset of torsion classes for $A_\cQ$ under inclusion order. A torsion class $\cT$ is determined by the isomorphism classes of bricks contained in $\cT$. We thus freely identify $\cT$ with $\cT\cap \mathrm{Bricks}(C_\cQ)$. We recall the following result from \cite[Section~8.2]{FTFSDL}.

\begin{lemma}[\cite{FTFSDL}]\label{lem:torsionftfsdl}
    Let $A$ be a finite-dimensional algebra of finite representation type. Then the poset of torsion classes for $A$ is a semidistributive lattice. The set of join-irreducible torsion classes can be identified with $\mathrm{Bricks}(A)$. Let $M_1,M_2$ be bricks. In the FTFSDL factorization, 
    \begin{itemize}
    \item if $M_1$ is a quotient of $M_2$, then $M_1\twoheadleftarrow M_2$; 
    \item if $M_1$ is a submodule of $M_2$, then $M_1\hookrightarrow M_2$; 
    \item there is a nonzero map from $M_1$ to $M_2$ if and only if $M_1\to M_2$.
    \end{itemize} 
\end{lemma}

\begin{proposition}\label{prop:torsion_isomorphic}
    Let $j_1,j_2$ be join-irreducible elements of $\CTam$. Then 
    \begin{itemize}
    \item $j_1\twoheadleftarrow j_2$ if and only if $M_{j_1}$ is a quotient of $M_{j_2}$; \item $j_1\hookrightarrow j_2$ if and only if $M_{j_1}$ is a submodule of $M_{j_2}$; 
    \item $j_1\to j_2$ if and only if there is a nonzero map from $M_{j_1}$ to $M_{j_2}$.
    \end{itemize} 
    As a result, $\CTam_{\Tors}$ is isomorphic to $\CTam_{312}$, and $\ATam_{\Tors}$ is isomorphic to $\ATam_{312}$.
\end{proposition}
\begin{proof}
    Upon inspection, we see that the conditions of being a submodule or quotient are equivalent to the conditions on the reflection indices described in \Cref{thm:FTFSDL_CTam}. 

    By \Cref{lem:torsionftfsdl}, the $\to$ relation on $\CTam_{\Tors}$ is given by existence of a nonzero map. Since the FTFSDL factorization is determined by the $\to$ relation, the lattices $\CTam_\Tors$ and $\CTam_{312}$ are isomorphic.  
\end{proof}

Given a collection $D$ of arcs, let \[ \psi(D) = \{ M_\gamma \mid \gamma \in D \}. \] 
In the following proposition, we will use the fact (special to Nakayama algebras like $C_\cQ$ and $A_{\cQ}$) that the submodules of indecomposable modules are also indecomposable.

\begin{proposition} 
    The map $\psi$ restricts to a bijection from the set of cyclic arc torsion classes to the set $\CTam_{\mathsf{Tors}}$ of torsion classes for $C_\cQ$.  
\end{proposition} 
\begin{proof}
    First, we note that $\psi$ is a bijection between sets of arcs and subsets of $\mathrm{Bricks}(C_\cQ)$. We can see that if $S\subseteq \mathrm{Bricks}(C_\cQ)$ is a torsion class, then $\psi^{-1}(S)$ is a cyclic arc torsion class: if $M_{a,c}$ is in $S$ and $a<b<c$, then $M_{a,b}$ is a quotient of $M_{a,c}$, so it is also in $S$. Similarly, if $M_{a,b},M_{b,c}\in S$, then there is a short exact sequence
    \[ 0\to M_{a,b} \to M_{a,c} \to M_{b,c} \to 0,  \]
    so $M_{a,c}$ is also in $S$. We conclude that $\psi^{-1}$ sends torsion classes to cyclic arc torsion classes. 

    Now let $D$ be a cyclic arc torsion class. If $M$ is an indecomposable module that is a quotient of an element $M_{a,c}$ of $\psi(D)$, then $M\cong M_{a,b}$ for some $a<b\leq c$, so in fact $M\in \psi(D)$. Now assume $M=M_{a,c}$ is a brick that has a filtration $0=M^0\subset M^1 \subset \cdots \subset M^k = M$ such that $M^{i+1}/M^i$ is a direct sum of elements of $\psi(D)$. Then $M^1$ is a submodule of $M$, so it coincides with $M_{a,b}$ for some $a<b\leq c$. Similarly, $M/M^1$ is a quotient of $M$, so it must also be indecomposable and hence isomorphic to $M_{b,c}$. By assumption, $M^1\in \psi(D)$, and by induction on the length of the filtration, $M/M^1\in \psi(D)$. Therefore, $\gamma_{a,b},\gamma_{b,c}\in D$, and we conclude that $M\cong M_{a,c}$ is in $\psi(D)$. 
\end{proof}

\begin{proof}[Proof of \Cref{prop:noncrosstoclass}]
    We first show that $D\mapsto \overline{D}$ is a bijection from the set of noncrossing arc diagrams to the set of cyclic arc torsion classes. The join of the set $S=\{ M_\gamma \mid \gamma \in D\}$ in $\CTam_\Tors$ is by definition the torsion class generated by $S$, which is the intersection of all torsion classes containing $S$. This coincides with the image under $\psi$ of the cyclic arc torsion class generated by $D$. 

    Now, if $\preceq$ is a 312-avoiding TITO and $D= \cA(\preceq)$, then by \Cref{lem:JIeqinv}, we have
    \[ \{\gamma_{a,b} \mid (a,b)\in\Inv(\preceq)\} = \{\gamma_{a,b} \mid j_{a,b} \leq\, \preceq \}. \]
    Applying $\psi$, we obtain the set $\{ M_{a,b} \mid j_{a,b} \leq\, \preceq \}$, which is exactly the torsion class associated to $\preceq$, viewed as an element of $\CTam_{\Tors}$. This coincides with $\psi(\overline{D})$, as claimed. 
\end{proof}

\section{Type D Triangulations}\label{sec:triangulations}
In this section, we focus on the affine Tamari lattice. We saw in \cref{prop:torsion_isomorphic} that $\ATam$ is the lattice of torsion classes for the algebra $A_\cQ$, which is a cluster-tilted algebra of type $D_n$. Hence, the Hasse diagram of $\ATam$ coincides with the cluster exchange graph of type $D_n$. The type $D_n$ exchange graph is modeled by type $D_n$ triangulations \cite{LaurenBook}. In this section, we give two isomorphisms between the Hasse diagram of $\ATam$ and the exchange graph of the type $D_n$ cluster algebra. Because of their algebraic interpretations, the two isomorphisms coincide.

Consider a regular $n$-gon $\Po$. Let $\redbullet$ be the point at the center of $\Po$. Let $\Pob$ be the punctured $n$-gon obtained by removing $\redbullet$ from $\Po$. The \dfn{marked points} of $\Pob$ are the $n$ vertices together with the central point $\redbullet$.

\begin{definition}
    A \dfn{tagged arc} is an arc $\mathfrak{a}$ embedded within $\Pob$ without self-intersections, satisfying the following properties:
    \begin{itemize}
        \item The endpoints of $\mathfrak{a}$ are distinct marked points.
        \item If the endpoints of $\mathfrak{a}$ are adjacent points on the boundary of $\Pob$, then $\mathfrak{a}$ must encircle $\redbullet$.
    \end{itemize}
    Tagged arcs with an endpoint at $\redbullet$ are called \dfn{radial}. Radial arcs come in two flavors: \dfn{plain} and \dfn{notched}; notched arcs are indicated by the symbol $\includegraphics[height=0.22cm]{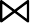}$.

    We say that two tagged arcs $\mathfrak{a},\mathfrak{b}$ are \dfn{compatible} if one of the following conditions holds:
    \begin{itemize}
        \item $\mathfrak{a}$ and $\mathfrak{b}$ are not both radial, and after isotopy fixing endpoints, there is an embedding of $\mathfrak{a}$ and $\mathfrak{b}$ so that they do not intersect except possibly at one endpoint. 
        \item $\mathfrak{a}$ and $\mathfrak{b}$ are both radial, and they have the same endpoints. One of them is plain, while the other is notched. 
        \item $\mathfrak{a}$ and $\mathfrak{b}$ are both radial, and they do not have the same endpoints. They are either both plain or both notched. 
    \end{itemize}
    A (type $D_n$) \dfn{triangulation} is a set of $n$ pairwise-compatible tagged arcs.
\end{definition}

Let us label the edges of $\Pob$ by $1,2,\ldots, n$ in clockwise order. We view these labels as living in $\ZZ/n\ZZ$. For $b\neq a-1$, let $\mathfrak{a}_{a,b}$ be the tagged arc that starts at the point immediately counter-clockwise of $a$ and moves clockwise around $\redbullet$ to the point immediately clockwise of $b$. Let $\mathfrak{a}_{a,a+n-1}$ be the plain radial arc with an endpoint at the vertex between $a-1$ and $a$. Let $\mathfrak{r}_a$ be the notched radial arc ending at the vertex immediately clockwise of $a$. Evidently, there is a bijection between the arcs $\gamma_{a,b}$ from \Cref{sec:Arcs_FTFSDL} with $a<b\leq n-1$ and plain tagged arcs $\mathfrak{a}_{a,b}$ (see \cref{fig:tagged_arcs}). 
We will think of the arc $\gamma_{a,b}$ as living in $\Pob$ by drawing it as an arc from the edge $a$ to the edge $b$. 

\begin{figure}[ht]
\begin{center}{\includegraphics[height=9cm]{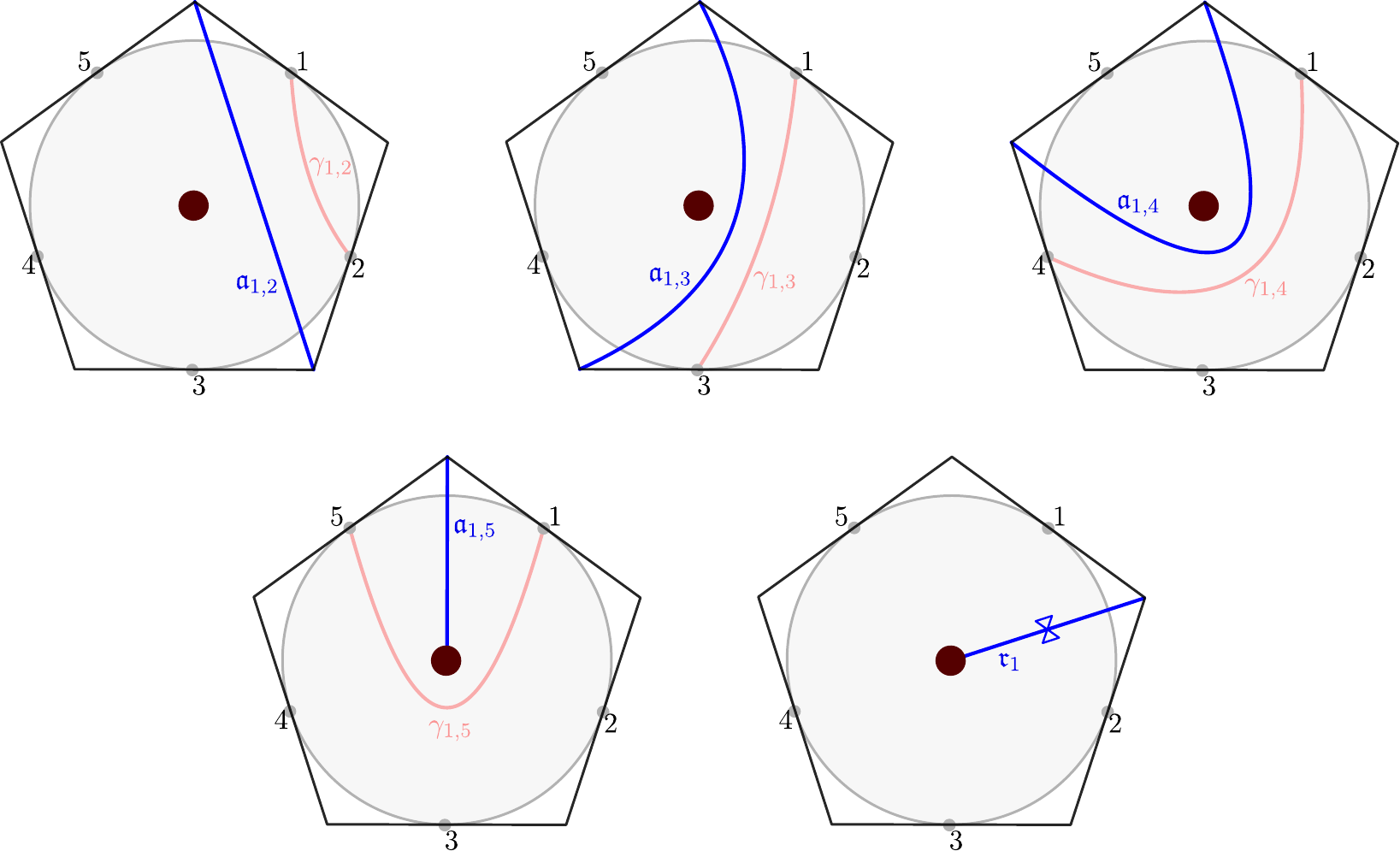}}
\end{center}
\caption{Some of the tagged arcs in $\Pob$ for $n=5$. Each tagged arc ({\color{blue}blue}) corresponds to an annular arc, which is overlaid in {\color{FaintRed}faint red}.} 
\label{fig:tagged_arcs} 
\end{figure} 


\begin{definition}\label{def:triangulation}
    Let $\cT$ be a torsion class in $\ATam_\Tors$. The \dfn{s$\tau$-tilting module} of $\cT$ is the unique $A_\cQ$-representation $M\in \cT$ such that every indecomposable summand of $M$ appears with multiplicity one, and for every representation in $N\in\cT$, we have that $\Ext^1_{A_\cQ}(M,N)=0$ and $N$ is a quotient of $M^{\oplus k}$ for some $k$. The \dfn{triangulation} of $\cT$ is the set of tagged arcs consisting of all $\mathfrak{a}_{a,b}$ such that $M_{a,b}$ is a summand of the s$\tau$-tilting module of $\cT$ and all $\mathfrak{r}_a$ such that the simple module $M_{a,a+1}$ is not a subquotient of any element of $\cT$. 
\end{definition}

Because $A_\cQ$ is a $\tau$-tilting finite algebra, each torsion class contains an s$\tau$-tilting module; this gives a bijection between torsion classes and s$\tau$-tilting modules \cite[Theorem~2.7]{Adachi}. We now indicate how this corresponds with the usual notion of triangulation associated to a cluster in a type $D_n$ cluster algebra. First, there are surface models for the type $D_n$ \emph{cluster category}; e.g. in \cite{Schiffler}. It is known that \emph{cluster-tilting objects} of the cluster category biject with clusters in a type $D_n$ cluster algebra; furthermore, each cluster-tilting object (equivalently, each cluster) has an associated triangulation. Now, \cite[Theorem~4.1]{Adachi} gives a bijection between the cluster-tilting objects in the cluster category and the s$\tau$-tilting modules for $A_\cQ$. Hence each s$\tau$-tilting module has an associated triangulation, which is the triangulation described in \Cref{def:triangulation}. Furthermore, two clusters are related by a mutation if and only if the associated torsion classes are related by an edge in the Hasse diagram of $\ATam_\Tors$, and this occurs if and only if the associated triangulations are related by a \emph{flip move} (i.e., differ by exactly one arc). It follows that the map from $\ATam_\Tors$ to triangulations gives an isomorphism from the Hasse diagram of $\ATam_{\Tors}$ to the exchange graph of the type $D_n$ cluster algebra. We give more direct proofs of this fact in \Cref{lem:triang_is_triang,cor:flipgraph}.

\begin{lemma}\label{lem:triang_is_triang}
    The map sending a torsion class $\cT$ to its triangulation is bijection from $\ATam_\Tors$ to the set of type $D_n$ triangulations.
\end{lemma}
\begin{proof}
    The map sending each torsion class to its triangulation is an injection from $\ATam_\Tors$ to the collection of sets of tagged arcs, since we can recover a torsion class $\cT$ as the torsion class generated by the $s\tau$-tilting module $M$, which we can read off from the unnotched arcs in the triangulation of $\cT$. Since there are $\Cat_{D_n}$ elements of $\ATam_\Tors$ and $\Cat_{D_n}$ type $D_n$ triangulations, it is enough to show that the triangulation of a torsion class $\cT$ is a type $D_n$ triangulation. 

    Let $\mathfrak{a},\mathfrak{b}$ be tagged arcs in the triangulation of $\cT$. We wish to show $\mathfrak{a}$ and $\mathfrak{b}$ are compatible. If $\mathfrak{a},\mathfrak{b}$ are both unnotched, then they correspond to summands $M_{a,b},M_{c,d}$ of the $s\tau$-tilting module $M$. Hence, $\Ext^1(M_{a,b},M_{c,d}) = 0$, 
    so by \Cref{lem:ext}, the integers $a,b,c,d$ cannot be chosen so that $a<c\leq b < d < a+n$. This implies the arcs $\mathfrak{a}_{a,b},\mathfrak{a}_{c,d}$ do not cross, so they are compatible. If instead $\mathfrak{a}=\mathfrak{a}_{a,b}$ and $\mathfrak{b}=\mathfrak{r}_c$, then the arcs are incompatible if and only if we can choose $c$ so that $a\leq c < b$. This occurs if and only if the simple module $M_{c,c+1}$ is a subquotient of $M_{a,b}$, so by the definition of the triangulation, this cannot happen. Finally, if $\mathfrak{a}=\mathfrak{r}_{a}$ and $\mathfrak{b}=\mathfrak{r}_b$, then the arcs are compatible since they are distinct. 
    
\end{proof}

We now give two ways of computing the triangulation associated to a torsion class, using the combinatorics we have developed so far. The first uses affine arc torsion classes, and the second uses TIBITs and linear extensions.

\subsection{Triangulations and Arc Torsion Classes} 
We first provide a correspondence between affine arc torsion classes and triangulations. 
See \cref{fig:triangulation_1} for an illustration. First, we need a computational lemma.

\begin{lemma}\label{lem:ext}
    We have $\Ext^1_{A_\cQ}(M_{a,b},M_{c,d})\neq 0$ if and only if the integers $a,b,c,d$ can be chosen so that $a<c\leq b < d < a+n$.
\end{lemma}
\begin{proof}
    Note that $P_x\coloneqq M_{x,x+n-1}$ is a projective $A_\cQ$-module for each $x$. Hence, there is a projective resolution
    \[ \cdots \to P_{a+n-1} \to P_{b} \to  P_a \to M_{a,b} \to 0. \]
    So our desired Ext group is the first cohomology of the complex
    \[ \Hom(P_a,M_{c,d}) \to \Hom(P_b,M_{c,d}) \to \Hom(P_{a+n-1},M_{c,d}) \to \cdots. \]
    There is a nonzero 1-cocycle if and only if $c\leq b < d\leq a+n-1$. That cocycle is a coboundary if and only if we also have $c\leq a <d$. We conclude that the first cohomology is nonzero if and only if $a<c\leq b < d \leq a+n-1$, as claimed.
\end{proof}

\begin{theorem}
    Let $D$ be an affine arc torsion class. The triangulation associated to the torsion class $\psi(D)$ consists of the following: 
    \begin{itemize}
    \item all tagged arcs $\mathfrak{a}_{a,b}$ such that $\gamma_{a,b}\in D$ and there do not exist $a<b'\leq b<c<a+n$ with $\gamma_{b',c}\in D$; 
    \item all notched radial arcs $\mathfrak{r}_a$ such that there do not exist $a'\leq a<b'$ satisfying $\gamma_{a',b'}\in D$. 
    \end{itemize} 
\end{theorem}
\begin{proof}
    First note that the simple module $M_{a,a+1}$ is a subquotient of a module in a torsion class $\cT$ if and only if $S_a$ is a subquotient of a brick in $\cT$. Furthermore, $M_{a,a+1}$ is a subquotient of the brick $M_{a',b'}$ if and only if (possibly after relabeling) $a'\leq a < b$. This proves the description of the notched radial arcs in the triangulation of $\psi(D)$.
    Now we need to check that 
    \[ \bigoplus \{M_{a,b} \mid \gamma_{a,b}\in D,~ \text{there do not exist } a<b'\leq b<c<a+n\text{ with }\gamma_{b',c}\in D\} \]
    is exactly the s$\tau$-tilting module of $\psi(D)$. 
    The s$\tau$-tilting module of $\psi(D)$ is the direct sum of the indecomposable modules $M_{a,b}\in \psi(D)$ such that $\mathrm{Ext}^1(M_{a,b},T)=0$ for all $T\in \psi(D)$. Let $M_{a,b}$ be such an indecomposable module in $\psi(D)$, and suppose instead that there exist $a<b'\leq b<c<a+n$ with $\gamma_{b',c}\in D$. Then $M_{b',c}\in \psi(D)$, and $M_{a,c}\oplus M_{b',b}$ is an extension of $M_{b',c}$ and $M_{a,b}$, so
    $\Ext^1(M_{a,b},M_{b',c})\neq 0$, which is a contradiction. 
    
    Conversely, assume $M_{a,b}$ is an indecomposable module in $\psi(D)$ such that there do not exist $a<b'\leq b<c<a+n$ with $\gamma_{b',c}\in D$. We claim that $\mathrm{Ext}^1(M_{a,b},T)=0$ for each $T\in \psi(D)$.  Indeed, it is enough to check this when $T=M_{c,d}$ is indecomposable, where it follows from \Cref{lem:ext}.
\end{proof}

\begin{figure}[ht]
\begin{center}{\includegraphics[height=4.256cm]{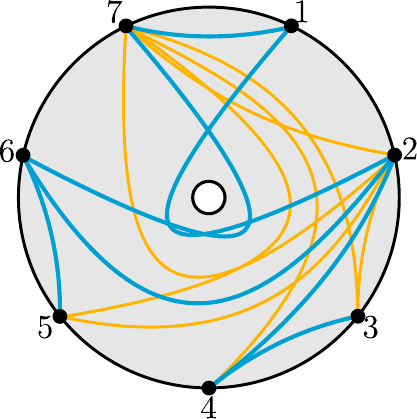}\qquad\qquad\includegraphics[height=4.5cm]{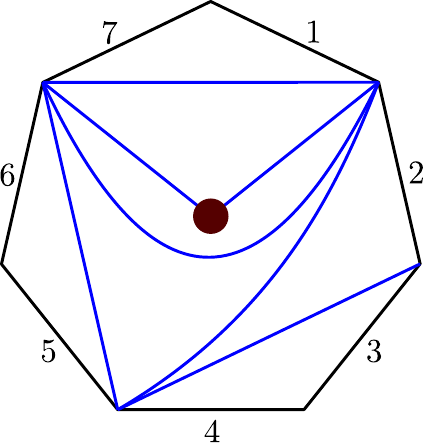}}
\end{center}
\caption{On the left is the affine arc torsion class from the left side of \cref{fig:arc_torsion_classes}. On the right is the corresponding triangulation. We have colored in blue the arcs on the left that correspond to the triangulation arcs on the right.}   
\label{fig:triangulation_1} 
\end{figure} 

\subsection{Triangulations and $\mathbf{g}$-vectors}\label{subsec:g-vectors}
We now give a way to compute the triangulation associated to an element of $\ATam$ using the \emph{stability fan} 
of $A_\cQ$. Given a vector $\theta = (\theta_1,\ldots, \theta_n)\in \RR^n$ and a representation $N$ of $A_\cQ$, let
\[ \langle \theta, N\rangle=\theta_1 \dim(N^1) + \theta_2 \dim(N^2) + \cdots + \theta_n \dim(N^n).  \]

Let $\cT_\theta$ be the set of representations $M$ of $A_\cQ$ such that $\langle \theta,N\rangle\geq 0$ for every quotient $N$ of $M$. According to \cite{BrustleFan}, $\cT_\theta$ is a torsion class of $A_\cQ$. 

\begin{definition}
The \dfn{cone} of a torsion class $\cT\in \ATam_{\Tors}$ is the polyhedral cone
    \[ \operatorname{cone}(\cT) \coloneqq \overline{\{\theta\in \RR^n \mid \cT_\theta = \cT 
    \}}. \]
\end{definition}
The relation between our definition of $\cone(\cT)$ and the definition given in \cite{BrustleFan} is explained in \cite[Lemma~3.12~and~Remark~3.28]{BrustleFan}.

\begin{proposition}[{\cite[Corollary 3.29]{BrustleFan}}]\label{prop:g-vector-fan}
    The cones of the form $\cone(\cT)$ for $\cT\in\ATam_{\Tors}$ form the chambers of a complete fan.
\end{proposition}

The complete fan from \cref{prop:g-vector-fan} is called the \dfn{stability fan} or \dfn{$\bg$-vector fan} of $A_\cQ$\footnote{In general, these are two different fans, but they coincide for $\tau$-tilting finite algebras like $A_\cQ$.}. 
To connect the combinatorics of this fan to the combinatorics we have been studying, we will give another perspective on the pairing $\langle \theta, N\rangle$. First, let $\RR^{\oplus \ZZ}$ be the vector space with basis $\{\mathbf{e}_i \mid i\in \ZZ\}$. Consider the $(n+1)$-dimensional quotient vector space $\RR^{\oplus\ZZ}/\langle \mathbf{e}_i-\mathbf{e}_{i+n} -(\mathbf{e}_j - \mathbf{e}_{j+n}) \mid i,j\in\ZZ \rangle$, and let $\widetilde{\mathbf{e}}_i$ be the image of $\mathbf{e}_i$ in this quotient. Let $\widetilde{V}$ be the $n$-dimensional vector subspace spanned by $\{\widetilde{\mathbf{e}}_{i}-\widetilde{\mathbf{e}}_{i+1}\mid i \in \ZZ\}$. There is a basis for $\widetilde{V}$ given by $\widetilde{\mathbf{e}}_1-\widetilde{\mathbf{e}}_2, \widetilde{\mathbf{e}}_2-\widetilde{\mathbf{e}}_3,\ldots,\widetilde{\mathbf{e}}_{n}-\widetilde{\mathbf{e}}_{n+1}$; let $\varpiv_1,\ldots,\varpiv_n$ be the dual basis of $\widetilde{V}^*$. As usual, we treat the indices of $\varpiv_i$ as living in $\ZZ/n\ZZ$, so that $\varpiv_0=\varpiv_n$, etc. The \dfn{dimension vector} of a representation $N$ is \[\underline{\dim}(N)\coloneqq \sum_{i=1}^n \dim(N^i)(\widetilde{\mathbf{e}}_{i}-\widetilde{\mathbf{e}}_{i+1}) \in \widetilde{V}.\] For example, $\underline{\dim}(M_{a,b}) = \widetilde{\mathbf{e}}_a-\widetilde{\mathbf{e}}_b$. Now, the pairing $\langle \theta, N\rangle$ is really the pairing $\langle \theta, \underline{\dim} N\rangle$ coming from the duality between $\widetilde{V}^*$ and $\widetilde{V}$, where $\theta = \theta_1 \varpiv_1 + \cdots + \theta_n\varpiv_n \in \widetilde{V}^*$. 

We now relate these vectors to TITOs. Each $\theta \in \widetilde{V}^*$ determines and is determined by the real numbers $\theta(i)\coloneqq \langle \theta, \widetilde{\mathbf{e}}_i - \widetilde{\mathbf{e}}_0\rangle$ for $i\in\ZZ$. Furthermore, the sequence $(\theta(i))_{i\in \ZZ}$ can be used to build TITOs. 
We say $\theta$ is \dfn{regular} if $\theta(i)\neq \theta(j)$ for all $i\neq j$. If $\theta$ is regular, then we define a total order $\prec_\theta$ on $\ZZ$ so that $i\prec_\theta j$ if and only if $\theta(i) < \theta(j)$. Evidently, $\prec_\theta$ is a TITO, since $\theta(j+n)-\theta(i+n) = \theta(j)-\theta(i)$. 

\begin{remark}
    Not all TITOs are of the form $\prec_\theta$ for some $\theta\in \RR^n$. For example, $[1][\underline{2,0}]$ is not, nor is any other TITO with both a waxing and a waning block. 
\end{remark}

We will also be interested in non-regular $\theta$. We say that $\theta$ is \dfn{compatible} with a TIBIT $T$ if for all nodes $\T_i,\T_j$ of $T$ such that $\T_i\leq_T \T_j$, we have $\theta(i)\leq \theta(j)$. If $\theta$ is regular, then $\theta$ is compatible with $T$ if and only if $\prec_\theta$ is a TILE of $T$.

\begin{lemma}\label{lem:TIBITcone}
    Let $T$ be a TIBIT. Then the set $\{\theta \in \widetilde{V}^* \mid \text{$\theta$ is compatible with $T$}\}$ is an $n$-dimensional polyhedral cone. 
\end{lemma}
\begin{proof}
First observe that $\theta$ is compatible with $T$ if and only if $\theta(i)\leq \theta(j)$ for all $ \T_i \leq_T \T_j$ such that $1\leq i \leq n$ and $|i-j| \leq n$. It follows that $\theta$ is compatible with $T$ if and only if it satisfies a finite list of linear inequalities, so $\{\theta \in \widetilde{V}^* \mid \text{$\theta$ is compatible with $T$}\}$ is a polyhedral cone. It remains to check that the cone is $n$-dimensional.  The only way the cone could fail to be $n$-dimensional is if it is contained in one of its supporting hyperplanes. Because each supporting hyperplane is of the form $\{\theta \in \widetilde{V}^* \mid \theta(i) = \theta(j) \}$, it is enough to show that there is some regular $\theta$ that is compatible with $T$. 

We first note that there is a TILE of $T$ that has exactly one block. Indeed, if the nodes in the spine of $T$ are all left children, then by the definition of post-order, $\po_T$ is a TILE of $T$ with exactly one block. On the other hand, if the nodes in the spine of $T$ are all right children, then by the definition of reverse-post-order, $\po^T$ is a TILE of $T$ with exactly one block. 

Now let $\preceq$ be a TILE of $T$ that has exactly one block. If the unique block of $\preceq$ is waxing, set $\preceq'$ to be $\preceq$ and set $\epsilon=1$; if the block is waning, set $\preceq'$ to be $\Psi_{\leftrightarrow}(\preceq)$ (the reverse of $\preceq$) and set $\epsilon=-1$. Since $\preceq'$ has a unique waxing block and no waning blocks, there is a unique affine permutation $\pi\colon\ZZ\to\ZZ$ such that for all $i,j\in\ZZ$, we have $i\preceq' j$ if and only if $\pi^{-1}(i)\leq\pi^{-1}(j)$. Now we set $\theta$ to be the unique element of $\widetilde{V}^*$ such that $\epsilon(\theta(i) - \theta(0)) = \pi^{-1}(i) - \pi^{-1}(0)$ for all $i\in \ZZ$. We claim that the vector $\theta$ is compatible with $T$; evidently it is regular, so this will finish the proof. To see compatibility, consider nodes $\T_i,\T_j$ of $T$ such that $\T_i \leq_T \T_j$. Then $i\preceq j$, so we have $\epsilon\pi^{-1}(i) < \epsilon\pi^{-1}(j)$, which implies that $\theta(i) < \theta(j)$. The claim follows. 
\end{proof}

\begin{lemma}\label{lem:cone1} 
Let $T$ be a TIBIT, and let $\cT$ be the corresponding torsion class for $A_\cQ$. A vector $\theta\in \widetilde V^*$ is in $\cone(\cT)$ if and only if it is compatible with $\pispine^{\downarrow}(T)$ or $\pispine^\uparrow (T)$. 
\end{lemma}
\begin{proof}
    Suppose first that $\cT_{\theta} = \cT$; we will show that $\theta$ is compatible with $\pispine^\uparrow (T)$ or $\pispine^\downarrow (T)$. Note that $\pispine^\uparrow (T)$ and $\pispine^\downarrow (T)$ are distinct if and only if the spine of $T$ consists of a single residue class modulo $n$. In this case, after replacing $T$ by $\pispine^\uparrow (T)$ or $\pispine^\downarrow (T)$, we may assume that $\theta(b) \leq \theta(a)$ whenever $v_b$ is a descendant of $v_a$ and both are in the spine. Under this assumption, we will show that $\theta$ is compatible with $T$. 
    
    Assume instead that $\theta$ is not compatible with $T$. Then there is some pair of nodes $v_{a}$, $v_{b}$ in $T$ such that $v_{b}$ is a child of $v_{a}$ and $\theta(b)>\theta(a)$. Assume $b>a$; the other case is similar. Then $b-a\leq n$. We cannot have $b-a=n$ since that would imply that $v_a$ and $v_b$ are both in the spine, contradicting our above assumption. Thus, $b-a<n$, and 
    \Cref{rem:arc_torsion_containment}
    implies that $M_{a,b} \in \cT$. But then $\langle \theta, M_{a,b}\rangle = \theta(a)-\theta(b) \geq 0$ by the definition of $\cT_{\theta}$, and this contradicts the fact that $\theta(b)>\theta(a)$. 

    It follows from the preceding argument that every vector in $\overline{\{\theta \in \RR^n \mid \cT_\theta=\cT\}}$ is compatible with $\pispine^\uparrow (T)$ or $\pispine^\downarrow (T)$. It remains to check that if $\theta$ is compatible with $T$, then $\theta\in \cone(\cT)$. Any vector compatible with $T$ is a limit of regular vectors that are compatible with $T$ by \Cref{lem:TIBITcone} and the fact that removing countably many codimension-1 hyperplanes from a codimension-0 polyhedral cone does not change its closure. Therefore, we may assume that $\theta$ is regular. 

    We will show that $\cT_\theta = \cT$. Certainly, we have $\cT\subseteq \cT_\theta$ since each quotient of a module in $\cT$ is in $\cT$ and is, hence, a direct sum of modules of the form $M_{a,b}$ with $(a,b)$ an inversion of $T$ (by \Cref{rem:arc_torsion_containment}). For such modules, $\langle \theta, M_{a,b}\rangle = \theta(a)-\theta(b) \geq 0$ since $v_b$ is a descendant of $v_a$. Conversely, let $M=M_{a,b}$ be a brick in $\cT_\theta$. Then, by the definition of $\cT_\theta$ and regularity of $\theta$, we have $\langle \theta, M_{a,b'}\rangle = \theta(a)-\theta(b') > 0$ for all $a< b' \leq b$.  We wish to show that $M_{a,b}\in \cT$, or equivalently that $(a,b)$ is an inversion of $T$. In other words, we wish to show that $v_b$ is a descendant of $v_a$. Suppose this is not the case, and let $v_c$ be the least common ancestor of $v_a$ and $v_b$. Then $a<c\leq b$, and $v_a$ is a descendant of $v_c$, so compatibility requires that $\theta(a)\leq \theta(c)$. But this contradicts the fact that $\langle \theta, M_{a,c}\rangle > 0$, so we conclude that $M_{a,b}\in \cT$. 
\end{proof}

We wish to describe the polyhedral geometry of $\cone(\cT)$. In order to do so, we will introduce a vector in $\widetilde{V}^* \cong \RR^n$, called a $\mathbf{g}$-vector, associated to each tagged arc. 

\begin{definition}
    The \dfn{$\mathbf{g}$-vector} of a tagged arc is defined as follows:
    \begin{align*}
        \mathbf{g}(\mathfrak{a}_{a,b}) &\coloneqq \begin{cases}
            \varpi_a^\vee - \varpi_b^\vee & \text{if $b-a<n-1$; } \\
            \varpi_a^\vee &\text{if $b-a=n-1$;}
        \end{cases} \\
        \mathbf{g}(\mathfrak{r}_b) &\coloneqq -\varpiv_b.
    \end{align*}
    We write $\mathbf{g}(\ATam)$ for the set of $\mathbf{g}$-vectors of all tagged arcs.
\end{definition}

\begin{proposition}\label{prop:gvecwalls}
    Let $T\in\ATam_{\mathsf{Tree}}$ and $\cT\in\ATam_\Tors$ be a  real TIBIT and its corresponding torsion class, with triangulation arcs $\mathfrak{b}_1, \ldots, \mathfrak{b}_n$. 
    \begin{itemize}
        \item The extremal rays of $\cone(\cT)$ are the rays
        \[  \RR_{\geq 0}\bg(\mathfrak{b}_1),\, \ldots,\, \RR_{\geq 0}\bg(\mathfrak{b}_n). \]

        \item Let $M_{a,b}^\pm \coloneqq \{ \theta\in \RR^n \mid \langle \pm\theta, M_{a,b} \rangle \geq 0 \}.$ Then the facet-defining inequalities of $\cone(\cT)$ are as follows: 
        \begin{itemize}
        \item If $(a,b)$ is a lower wall of $\po_T$, then $\cone(\cT)\subseteq M_{a,b}^+$. 
        \item If $(a,b)$ is an upper wall of $\po^{\pispine^\uparrow T}$, then $\cone(\cT)\subseteq M_{a,b}^-$. 
        \end{itemize}
    \end{itemize}
\end{proposition}
\begin{proof}
    First we check the facet-defining inequalities. The \dfn{brick label} of a cover relation $\cT_1\lessdot \cT_2$ in $\ATam_{\Tors}$ is the unique brick $M\in\cT_2$ such that $\Hom(B,M)=0$ for all $B\in \cT_1$. By \cite[Proposition~4.9]{Demonet}, the brick label of a cover relation coincides with the module described in \cite[Proposition~3.17]{BrustleFan}. Consequently, \cite[Corollary~3.18]{BrustleFan} implies that the facet-defining hyperplanes of $\cone(\cT)$ are $(\underline{\dim} M)^\perp$, as $M$ ranges over the brick labels of the upper and lower covers of $\cT$ in $\ATam_{\Tors}$. It follows from \cite[Theorem 5.1, Theorem 8.6]{FTFSDL} that the brick labeling of a cover relation in $\ATam_\Tors$ coincides with the labeling of cover relations by join-irreducibles described in \cref{subsec:lattices}. 
    Hence, the brick $M_{a,b}$ labels a lower cover of $\cT$ if and only if $\gamma_{a,b}$ is in the unique non-crossing arc diagram $D$ such that $\psi(\overline{D}) = \cT$, which occurs if and only if $(a,b)$ is a lower wall of $\po_T$. Similarly, $M_{a,b}$ labels an upper cover of $\cT$ if and only if $(a,b)$ is an upper wall of $\po^{\pispine^\uparrow (T)}$. We conclude that the list of facet-defining hyperplanes is the one claimed; it remains to check the direction of the inequalities. Consider a regular $\theta\in \RR^n$ compatible with $T$. If $(a,b)$ is a lower wall of $T$, then $\theta(b)< \theta(a)$ by compatibility, so $\langle \theta, M_{a,b} \rangle > 0$ and, therefore, $\theta \in M_{a,b}^+$. Consequently, $\cone(\cT) \subseteq M_{a,b}^+$ whenever $(a,b)$ is a lower wall of $T$. Similarly, $\cone(\cT)\subseteq M_{a,b}^-$ whenever $(a,b)$ is an upper wall of $\pispine^\uparrow (T)$.

    We now check the extremal vectors of $\cone(\cT)$. Let $M$ be the $s\tau$-tilting module associated to $\cT$. By \cite[Theorem 2.2]{BrustleFan}, there is a unique projective module $P$ such that $(M,P)$ is a \emph{$\tau$-tilting pair} $(M,P)$. Furthermore, since $\cone(\cT)$ coincides with the cone $\mathcal{C}_{(M,P)}$ defined in \cite[Section 3]{BrustleFan}, the extremal rays of   $\cone(\cT)$ coincide with the $\mathbf{g}$-vectors of the $\tau$-tilting pair $(M,P)$. Therefore, we need to verify that the $\mathbf{g}$-vectors of $(M,P)$ as defined in \cite[Definition 3.5]{BrustleFan} coincide with $\mathbf{g}(\mathfrak{b}_1),\ldots,\mathbf{g}(\mathfrak{b}_n)$. Indeed, the notched arc $\mathfrak{r}_a$ is in the triangulation if and only if the simple module $M_{a,a+1}$ is not a subquotient of $M$. By definition of a $\tau$-tilting pair, this occurs if and only if the projective module $P_a=M_{a,a+n-1}$ is a summand of $P$, which occurs if and only if $-\varpiv_a$ is a $\mathbf{g}$-vector of $(M,P)$. Now, the unnotched arc $\mathfrak{a}_{a,b}$ is in the triangulation if and only if $M_{a,b}$ is a summand of $M$. Because
    \[ \cdots \to P_{b} \to P_{a} \to M_{a,b} \to 0 \]
    is a minimal projective resolution for $M_{a,b}$, this occurs if and only if $\varpiv_a-\varpiv_b$ is a $\mathbf{g}$-vector of $(M,P)$, as desired.
\end{proof}

Because the stability fan is a fan whose chambers have (by \Cref{prop:gvecwalls}) extremal rays in $\mathbf{g}(\ATam)$,   we get the following corollary. 
\begin{corollary}\label{cor:compatible}
Let $T\in\ATam_{\mathsf{Tree}}$ and $\cT\in \ATam_{\Tors}$ be a real TIBIT and its corresponding torsion class with triangulation arcs $\mathfrak{b}_1,\ldots,\mathfrak{b}_n$. Then 
\[ \mathbf{g}(\ATam)\cap \cone(\cT) = \{\mathbf{g}(\mathfrak{b}_1), \ldots, \mathbf{g}(\mathfrak{b}_n) \}. \]
In particular, a tagged arc $\mathfrak{b}$ is in the triangulation of $\cT$ if and only if $-\mathbf{g}(\mathfrak{b})$ is compatible with $T$.
\end{corollary}

\begin{corollary}\label{cor:flipgraph}
    If $\cT_1,\cT_2\in\ATam_\Tors$, then the triangulations of $\cT_1$ and $\cT_2$ are related by a flip move if and only if $\cT_1$ and $\cT_2$ are related by a cover relation.
\end{corollary}
\begin{proof}
Since two type $D_n$ triangulations are related by a flip move if and only if their intersection has size $n-1$, it is enough to show that two torsion classes $\cT_1,\cT_2$ are related by a cover relation if and only if $\mathbf{g}(\ATam) \cap \cone(\cT_1) \cap \cone(\cT_2)$ has size $n-1$, or equivalently, if and only if $\cone(\cT_1)\cap \cone(\cT_2)$ has codimension 1.  This follows from \cite[Proposition 4.5]{BrustleFan}.
\end{proof}

\begin{remark}\label{inversetranspose}
    \Cref{prop:gvecwalls} shows that there is a duality between $\bg$-vectors and the lower and upper walls of a TIBIT $T$. There is a strengthening of this statement to a duality of integral bases. Define the \dfn{$\mathbf{c}$-matrix} $C$ of $T$ to be the matrix whose column vectors, written in the basis $\widetilde{\mathbf{e}}_1-\widetilde{\mathbf{e}}_2,\ldots,\widetilde{\mathbf{e}}_{n}-\widetilde{\mathbf{e}}_{n+1}$, are the elements of 
    \begin{align*}  \{\widetilde{\mathbf{e}}_a-\widetilde{\mathbf{e}}_b \mid \text{$(a,b)$ is a lower wall of $\po_T$}\}
        \cup \{ \widetilde{\mathbf{e}}_b-\widetilde{\mathbf{e}}_a \mid \text{$(a,b)$ is a upper wall of $\po^T$}\}.\end{align*} 
        Then the inverse transpose matrix $(C^{-1})^\top$ has columns given by the $\bg$-vectors of the triangulation associated to $T$ (in the basis $\varpi_1^\vee,\ldots,\varpi_n^\vee$)  \cite[Theorem 1.2]{Nakanishi}. This gives a way of computing the triangulation given the TIBIT.
\end{remark}

Given a real TIBIT $T\in \ATam_{\mathsf{Tree}}$, we write $\mathbf{g}(T)$ for the set of $\mathbf{g}$-vectors of the triangulation arcs associated to $T$.

\begin{theorem}\label{thm:gveccompatible}
    Let $T\in \ATam_{\mathsf{Tree}}$ be a real TIBIT. For $z\in \ZZ$, let $I_z=\{x\in \ZZ \mid \T_x\leq_T \T_z\}$, and define
    \[\varpiv_{T,z} \coloneqq \begin{cases}
        \varpiv_a - \varpiv_b &\text{if } I_z=\{x\in \ZZ \mid a < x \leq b\}, \\
        \varpiv_a &\text{if } I_z=\{x\in \ZZ \mid a < x\}, \\
        -\varpiv_b &\text{if } I_z=\{x \in \ZZ \mid x \leq b\}. 
    \end{cases} \]
Then 
    \[ \mathbf{g}(T) = \{ \varpiv_{T,z} \mid 1\leq z \leq n\}. \]
\end{theorem}
\begin{proof}
    First, we note that the three cases for the form of $I_z$ are exhaustive, since $I_z$ is order-convex and is evidently not all of $\ZZ$. 
    We claim that if $I_z = \{x\in \ZZ \mid a<x\leq b \}$, then $b-a< n$. Suppose instead that $b-a\geq n$. Then there exists $x\in\ZZ$ such that $v_x$ and $v_{x+n}$ are both descendants of $v_z$. Say $\pr^k(v_x) = v_z$ and $\pr^{k'}(v_{x+n}) = v_z$; without loss of generality, $k\leq k'$. But then $\pr^k(v_{x+n})=v_{z+n}$ by translation-invariance, so $\pr^{k'-k}(v_{z+n})=v_z$, implying that $v_z$ is in the spine. This contradicts the hypothesis that $v_z$ has finitely many descendants. This proves our claim. We deduce that $\varpiv_{T,z}$ is always a $\mathbf{g}$-vector.

    By \Cref{cor:compatible}, $\mathbf{g}(T)$ is exactly the set of $\mathbf{g}$-vectors that are compatible with $T$. Because there are exactly $n$ elements of $\mathbf{g}(T)$, it is enough to show that $\{ \varpiv_{T,z} \mid 1\leq z \leq n\}$ consists of $n$ distinct elements, each of which is compatible with $T$. 
    
    First, to show distinctness, observe that we can recover the pair $(a,b)$ up to simultaneous translation by $n$ from the vector $\varpiv_a-\varpiv_b$. Then if $\varpiv_{T,z} = \varpiv_a-\varpiv_b$, we can recover the residue class of $z$ modulo $n$ as the residue class of the (index of the) $\leq_T$-maximal element of $\{\T_x \mid a < x \leq b\}$. Since we know $1\leq z \leq n$, we can thus recover $z$. The other cases are similar.

    Now we show that the vector $\theta\coloneq\varpiv_{T,z}$ is compatible with $T$. First consider the case where
    $\theta = \varpiv_a - \varpiv_b$. 
    We have \[\theta(i) = \begin{cases}
        -1 & \text{if $a<i\leq b$;} \\ 
        0 & \text{otherwise}.
    \end{cases}\] Then $\theta$ is compatible with $T$ if and only if for all $i,j\in\ZZ$ with $\T_i \leq_T \T_j$, we have $\theta(i)\leq \theta(j)$. This fails if and only if $\T_i\leq_T \T_j$ and $a<j \leq b$ hold but $a<i\leq b$ does not hold. Because $I_z=\{x\in\ZZ\mid a<x\leq b\}$, this is equivalent to the statement that $\T_i\leq_T \T_j$ and $\T_j\leq_T \T_z$ hold but $\T_i\leq_T \T_z$ does not hold, which is evidently impossible. Hence, $\theta$ is compatible with $T$. 
    
    We also must consider the case where $\theta = -\varpi_b$ so that \[\theta(i) = \begin{cases}
        1 & \text{if $i>b$;} \\ 
        0 & \text{if $i\leq b$}
    \end{cases}\] and the case where $\theta=\varpi_a$ so that \[\theta(i)=\begin{cases}
        -1 & \text{if $i>a$;} \\
        0 & \text{if $i\leq a$.}
    \end{cases}\] Both of these cases are similar to the previous one.
\end{proof} 

\begin{example}
    We now apply \Cref{inversetranspose} and \Cref{thm:gveccompatible} to the running example from \Cref{fig:triangulation_1} with $n=7$. The TIBIT $T$ is shown in \cref{fig:one_tree_shaded}. Window notations of $\po_T$ and $\po^{\pispine^\uparrow(T)}$ are $[1,4,3,6,5][\underline{7,2}]$ and $[\underline{6,4,3,5,7,1,2}]$, respectively. The lower walls of the TITO $\po_T$ are $(0,2),(2,7),(3,4),(5,6)$, and the upper walls of $\po^{\pispine^\uparrow (T)}$ are $(1,2),(3,5),(5,7)$. Hence, the $\mathbf{c}$-matrix of $T$ has column vectors 
    \[ \tbe_0-\tbe_2,\, \tbe_2-\tbe_7,\, \tbe_3-\tbe_4,\, \tbe_5-\tbe_6,\, \tbe_2-\tbe_1,\, \tbe_5-\tbe_3,\, \tbe_7-\tbe_5, \]
    expressed in the basis $\tbe_1-\tbe_2,\ldots,\tbe_7-\tbe_8$. This matrix and its inverse transpose are 
    \[ C 
    = \begin{bmatrix}
        1 & 0 & 0 & 0 & -1 & 0 & 0 \\
        0 & 1 & 0 & 0 & 0 & 0 & 0 \\
        0 & 1 & 1 & 0 & 0 & -1 & 0 \\
        0 & 1 & 0 & 0 & 0 & -1 & 0 \\
        0 & 1 & 0 & 1 & 0 & 0 & -1 \\
        0 & 1 & 0 & 0 & 0 & 0 & -1 \\
        1 & 0 & 0 & 0 & 0 & 0 & 0
    \end{bmatrix}, \qquad (C^{-1})^\top 
     = \begin{bmatrix}
        0 & 0 & 0 & 0 & -1 & 0 & 0 \\
        0 & 1 & 0 & 0 & 0 & 1 & 1 \\
        0 & 0 & 1 & 0 & 0 & 0 & 0 \\
        0 & 0 & -1 & 0 & 0 & -1 & 0 \\
        0 & 0 & 0 & 1 & 0 & 0 & 0 \\
        0 & 0 & 0 & -1 & 0 & 0 & -1 \\
        1 & 0 & 0 & 0 & 1 & 0 & 0
    \end{bmatrix}. \]
    Now we see that the columns of $(C^{-1})^\top$ are the $\mathbf{g}$-vectors \[\varpiv_7,\, \varpiv_2,\, \varpiv_3-\varpiv_4,\, \varpiv_5-\varpiv_6,\, \varpiv_7-\varpiv_8,\, \varpiv_2-\varpiv_4,\, \varpiv_2-\varpiv_6,\]
    so we conclude that the triangulation associated to $T$ has arcs
    \[ \mathfrak{a}_{7,13},\, \mathfrak{a}_{2,8},\, \mathfrak{a}_{3,4},\, \mathfrak{a}_{5,6},\, \mathfrak{a}_{7,8},\, \mathfrak{a}_{2,4},\, \mathfrak{a}_{2,6}. \]
    We can also check that each of these $\mathbf{g}$-vectors is compatible with $T$. In \cref{fig:one_tree_shaded}, we have shaded the nodes indexed by the sets $I_1$ ({\color{MyGreen}green}), $I_{5}$ ({\color{DarkPink}dark pink}), and $I_7$ ({\color{LightPink}light pink}). We have \[I_1=\{x\in\ZZ\mid 0<x\leq 1\},\quad I_{5}=\{x\in\ZZ\mid 2<x\leq 6\},\quad I_{7}=\{x\in\ZZ\mid 2<x\},\] which (by \cref{thm:gveccompatible}) provides another way to see that the vectors $\varpi_0-\varpi_1=\varpi_7-\varpi_8$, $\varpiv_2-\varpiv_6$, and $\varpiv_2$ are in $\mathbf{g}(T)$. 
\end{example} 

\begin{figure}[ht]
\begin{center}{\includegraphics[height=8.443cm]{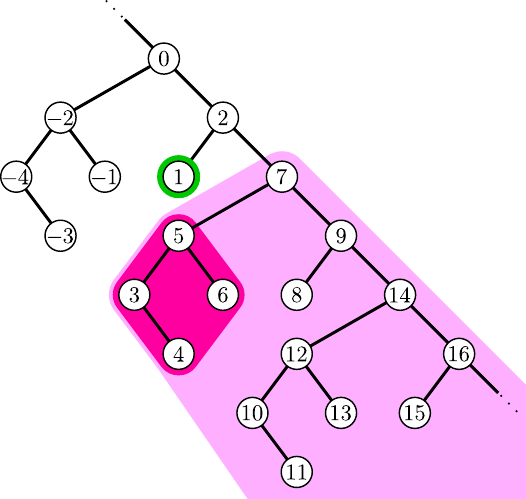}}
\end{center}
\caption{A TIBIT for $n=7$ with three shaded principal order ideals.} 
\label{fig:one_tree_shaded} 
\end{figure} 

\section{Maximal Green Sequences}\label{sec:green}
Let $\Lambda$ be an algebra, and let $\Tors(\Lambda)$ be the poset of torsion classes of representations of $\Lambda$. A \dfn{maximal green sequence} is a finite maximal chain in $\Tors(\Lambda)$. When $\Lambda$ is a cluster-tilted algebra (like $A_\cQ$), maximal green sequences can be identified with a sequence of mutations on the associated quiver. The origin of the name comes from work of Keller, who gave a way of coloring the vertices of the quiver either red or green; a maximal green sequence consists of a sequence of green mutations such that the final quiver has only red vertices. 

Specializing to our context, a maximal green sequence for $C_\cQ$ is a maximal chain in $\CTam_\Tors$, and a maximal green sequence for $A_\cQ$ is a maximal chain in $\ATam_\Tors$. Recall that the \dfn{length} of a chain is $1$ less than the number of elements in the chain. It is known that the maximum length of a maximal green sequence for $A_\cQ$ is $\binom{n+1}{2}-1$ \cite{Apruzzese}. Here we give a new proof of this fact using the combinatorics of $\ATam$. It was also shown by \cite{Garver} that the minimum length of a maximal green sequence for $A_\cQ$ is $2n-2$.\footnote{By \cite{CeballosPilaud}, this is also the diameter of the Hasse diagram of $\ATam$, viewed as an undirected graph.} It then follows from the  \emph{no-gap conjecture} (proven in this case in \cite{Garver2}) that the lengths of maximal green sequences for $A_\cQ$ are exactly the integers in the interval $[2n-2,\binom{n+1}{2}-1]$.
Using this, we give the first proof that the lengths of maximal green sequences for $C_\cQ$ are exactly the integers in the interval $[2n-1, \binom{n+1}{2}]$. 
Previously, it was known by \cite[Corollary~5.17]{GorskyWilliams} that the set of lengths of maximal green sequences for $C_\cQ$ is an interval, but its endpoints were unknown.

By determining the lengths of maximal green sequences for $C_\cQ$, we also determine the lengths for a large class of Nakayama algebras. In any finite dimensional quotient algebra $Q$ of the path algebra $k[\cQ]$ which has $C_\cQ$ as a quotient, the set of bricks of $Q$ coincides with the bricks of $C_\cQ$. This implies that the maximal green sequences of $Q$ have the same lengths as for $C_\cQ$, e.g., by describing maximal green sequences using bricks as in \cite[Theorem~A.3]{KellerSurvey}. By taking inverse limits, this also implies that maximal green sequences for the \dfn{completed path algebra} $k[\widehat{\cQ}]$ have the same lengths as for $C_\cQ$. The completed path algebra is the completion of $k[\cQ]$ at the maximal ideal generated by the arrows of $\cQ$.

Given a finite poset $P$, we let $\Gamma(P)$ denote the set of lengths of maximal chains in $P$. 

\begin{lemma}\label{lem:max_length_differ_by_1} 
For each integer $k$, we have $k\in\Gamma(\ATam)$ if and only if $k+1\in\Gamma(\CTam)$. 
\end{lemma}

\begin{proof}
Suppose first that $k+1\in\Gamma(\CTam)$. Then there is a maximal chain ${T_0\lessdot T_1\lessdot\cdots\lessdot T_{k+1}}$ in $\TIBIT$. There is a unique index $j\in[k+1]$ such that the TITOs $\po_{T_j},\ldots,\po_{T_{k+1}}$ have waning blocks while the TITOs $\po_{T_0},\ldots,\po_{T_{j-1}}$ do not have waning blocks. Moreover, we have ${\pispine^\downarrow(T_j)=T_{j-1}}$ and $\pispine^\uparrow(T_{j-1})=T_j$, and \[\pispine^\downarrow(T_0)\lessdot\cdots\lessdot\pispine^\downarrow(T_{j-1})\lessdot\pispine^\downarrow(T_{j+1})\lessdot\cdots\lessdot\pispine^\downarrow(T_{{k+1}})\] is a maximal chain in $\ATIBIT$. Therefore, $k\in\Gamma(\ATam)$. 

To prove the converse, suppose $k\in\Gamma(\ATam)$. Then there is a maximal chain $T_0'\lessdot\cdots\lessdot T_k'$ in $\ATIBIT$. Let $i$ be the largest index such that $\po_{T_i'}$ does not have a waning block. Then \[T_0'\lessdot\cdots\lessdot T_i'\lessdot\pispine^\uparrow(T_i')\lessdot T_{i+1}\lessdot\cdots\lessdot T_k'\] is a maximal chain in $\TIBIT$. Therefore, $k+1\in\Gamma(\CTam)$. 
\end{proof}

Let us introduce another combinatorial model for the affine Tamari lattice that will allow us to study its maximum-length chains. This new model was communicated to us by Andrew Sack, who was inspired by the \emph{ornamentation lattices} introduced in \cite{DefantSack} and studied further in \cite{AjranDefant}.  

As in \cref{sec:quivers}, let $\cQ$ be the oriented cycle quiver with node set $\ZZ/n\ZZ$. An \dfn{ornament} is a nonempty subset of $\ZZ/n\ZZ$ that induces a connected subgraph of $\cQ$. Let $\Orn(\cQ)$ be the set of ornaments. We say an ornament $\mathfrak o$ is \dfn{hung} at a vertex $a\in\ZZ/n\ZZ$ if $a\in\mathfrak o$ and either $a-1\not\in\mathfrak o$ or $\mathfrak o=\ZZ/n\ZZ$. Thus, the ornament $\ZZ/n\ZZ$ is hung at every vertex, while each proper ornament is hung at a unique vertex. An \dfn{ornamentation} is a function $\varrho\colon\ZZ/n\ZZ\to\Orn(\cQ)$ such that 
\begin{itemize}
\item for each $i\in\ZZ/n\ZZ$, the ornament $\varrho(i)$ is hung at $i$;
\item for all $i,j\in\ZZ/n\ZZ$, if $i\in\varrho(j)$, then $\varrho(i)\subseteq\varrho(j)$. 
\end{itemize}
Let $\cO$ be the set of ornamentations endowed with the partial order $\leq$ defined so that $\varrho\leq\varrho'$ if and only if $\varrho(i)\subseteq\varrho'(i)$ for all $i\in\ZZ/n\ZZ$. 

\begin{example}
Suppose $n=4$. The function $\varrho\colon\ZZ/4\ZZ\to\Orn(\cQ)$ given by $\varrho(1)=\{1,2,3\}$, $\varrho(2)=\{2,3\}$, $\varrho(3)=\{3\}$, and $\varrho(4)=\ZZ/4\ZZ$ is an ornamentation. We represent this ornament pictorially as 
\[\begin{array}{l}\includegraphics[height=2.274cm]{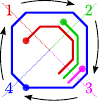}\end{array}.\] In this diagram, for each $i\in\ZZ/4\ZZ$, there is a colored ray $\mathfrak r_i$ starting at the central point and passing through $i$. We represent the ornament $\varrho(i)$ by a piecewise-linear curve with the same color as $\mathfrak r_i$ that crosses the rays $\mathfrak r_j$ for $j\in\varrho(i)$. We represent the ornament $\varrho(4)=\ZZ/4\ZZ$ as a closed loop. 
\end{example}

\begin{proposition}\label{prop:O_ATam}
The poset $\cO$ is isomorphic to $\ATam$. 
\end{proposition}
\begin{proof}
The map $\varrho\mapsto\{\gamma_{ij} \mid i\in \ZZ/n\ZZ,~j\in\varrho(i)\}$ is an isomorphism from $\cO$ to the poset of affine arc torsion classes ordered by containment. By \cref{rem:arc_torsion_containment}, the latter poset is isomorphic to $\ATam$. 
\end{proof} 

\begin{theorem}\label{thm:max_chains}
The maximum length of a maximal chain in the cyclic Tamari lattice $\CTam$ is $\binom{n+1}{2}$. The maximum length of a maximal chain in the affine Tamari lattice $\ATam$ is $\binom{n+1}{2}-1$. 
\end{theorem}

\begin{proof}
In light of \cref{lem:max_length_differ_by_1,prop:O_ATam}, it suffices to show that the maximum length of a chain in $\cO$ is $\binom{n+1}{2}-1$. We begin by constructing a chain of this length. 

Given integers $j,k$ satisfying $1\leq j\leq n$ and $1\leq k\leq \min\{n-j,n-2\}$, define ornamentations $\delta_j^*$ and $\delta_{j,k}$ by 
\[\delta_{j}^*(i)=\begin{cases}
     \ZZ/n\ZZ & \text{if $1\leq i\leq j$}; \\
    \{i\} & \text{if $j+1\leq i\leq n$}
\end{cases}\]
and 
\[\delta_{j,k}(i)=\begin{cases}
     \delta_j^*(i) & \text{if $i\neq j$}; \\
    \{j,j+1,\ldots,j+k\} & \text{if $i=j$}.
\end{cases}\]
Also, let $\delta_0^*$ be the minimum element of $\cO$, which is the ornamentation satisfying $\delta_0^*(i)=\{i\}$ for all $i\in\ZZ/n\ZZ$. Note that $\delta_n^*$ is the maximum element of $\cO$. Then 
\[\delta_0^*<\delta_1^*<\cdots<\delta_n^*.\] For $1\leq j\leq n$, we have 
\[\delta_{j-1}^*\lessdot\delta_{j,1}\lessdot\delta_{j,2}\lessdot\cdots\lessdot\delta_{j,\min\{n-j,n-2\}}\lessdot\delta_{j}^*.\] Thus, the set 
\[\{\delta_j^*\mid 0\leq j\leq n\}\cup\{\delta_{j,k}\mid 1\leq j\leq n,\, 1\leq k\leq \min\{n-j,n-1\}\}\] forms a chain of length $\binom{n+1}{2}-1$ in $\cO$. (See \cref{fig:max_chain}.) 

Now let $\varrho_0\lessdot\varrho_1\lessdot\varrho_2\lessdot\cdots\lessdot\varrho_m$ be a maximal chain in $\cO$; we wish to show that $m\leq\binom{n+1}{2}-1$. For each $r\in[m]$, there is a unique $i_r\in\ZZ/n\ZZ$ such that $\varrho_{r-1}(i_r)\neq\varrho_{r}(i_r)$. Moreover, there is a unique $q_r\in\ZZ/n\ZZ$ such that $q_r\in\varrho_{r-1}(i_r)$ and $q_r+1\not\in\varrho_{r-1}(i_r)$; note that $q_r+1\in\varrho_r(i_r)$. 

Define a permutation $z_1,\ldots,z_n$ of the elements of $\ZZ/n\ZZ$ as follows. First, let $z_1=q_1$. For $2\leq k\leq n$, let $z_k$ be the element of $(\ZZ/n\ZZ)\setminus\{z_1,\ldots,z_{k-1}\}$ that appears farthest to the left in the tuple $(q_1,\ldots,q_m)$. By the definition of $q_1$, the ornament $\varrho_1(z_1)$ is not a singleton. Therefore, for every $j\in[m]$, the ornament $\varrho_j(q_1)$ is not a singleton. It follows that $z_1$ appears exactly once in the tuple $(q_1,\ldots,q_m)$ (in the first position). 

If $r\in[m]$ is such that $q_{r}=z_2$, then we must have $i_r\in\{q_1,z_2\}$. It follows that $z_2$ appears at most $2$ times in the tuple $(q_1,\ldots,q_m)$. Similarly, if $r\in[m]$ is such that $q_{r}=z_3$, then we must have $i_r\in\{q_1,q_2,z_3\}$. This implies that $z_3$ appears at most $3$ times in the tuple $(q_1,\ldots,q_m)$. Continuing in this fashion, we find that each $z_k$ can appear at most $k$ times in the tuple $(q_1,\ldots,q_m)$. This shows that $m\leq 1+2+\cdots+n=\binom{n+1}{2}$. To show that $m\leq\binom{n+1}{2}-1$, we just need to show that $z_n$ appears at most $n-1$ times in $(q_1,\ldots,q_m)$. 

Suppose by way of contradiction that $z_n$ appears $n$ times in $(q_1,\ldots,q_m)$. Then there is some ${r\in[m]}$ such that $q_{r}=z_n$ and $i_r=z_n-1$. But then $\ZZ/n\ZZ=\varrho_{r-1}(i_r)\subseteq\varrho_{r-1}(i_r)$, so we have $\varrho_{r-1}(i_r)=\varrho_{r}(i_r)$. This contradicts the definition of $i_r$. 
\end{proof}

\begin{figure}[ht]
\begin{center}{\includegraphics[height=6.388cm]{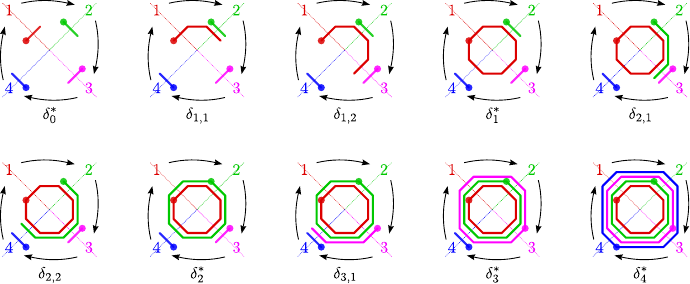}}
\end{center}
\caption{The maximum-length chain of $\cO$ constructed in the proof of \cref{thm:max_chains} for $n=4$.}   
\label{fig:max_chain} 
\end{figure}

\section{Rowmotion}\label{sec:rowmotion} 

\subsection{Dynamics} 

Consider a cyclic group $\cC_\omega=\langle g_\omega\rangle$ of order $\omega$ acting on a finite set $X$. For $F(q)\in\CC[q]$, we say the triple $(X,\cC_\omega,F(q))$ exhibits the \dfn{cyclic sieving phenomenon} if for all $k\in\ZZ$, we have 
\[|\{x\in X\mid g_\omega^k\cdot x=x\}|=F(e^{2\pi ik/\omega}).\] If this is the case, $F(q)$ is called a \dfn{sieving polynomial}; we can view it as a polynomial that encodes the entire orbit structure of the action of $\cC_\omega$. Such a sieving polynomial necessarily exists, so the cyclic sieving phenomenon is only interesting when one can write down the sieving polynomial in some reasonably nice, explicit form.

We say two invertible maps $f\colon X\to X$ and $f'\colon X'\to X'$ are \dfn{dynamically equivalent} if there exists a bijection $\varphi\colon X\to X'$ such that $f'\circ\varphi=\varphi\circ f$. 

\subsection{Rowmotion on Semidistributive Lattices}\label{subsec:rowmotion_semidistributive}

Let $L$ be a finite semidistributive lattice. Recall that $\cD(x)$ and $\cU(x)$ denote the canonical join representation and the canonical meet representation, respectively, of an element $x\in L$. Barnard \cite{Barnard} proved that there is a unique bijection $\Row_L\colon L\to L$ such that $\cD(x)=\cU(\Row_L(x))$ for all $x\in L$; the map $\Row_L$, which we denote simply by $\Row$ when the lattice $L$ is understood, is called \dfn{rowmotion}. Rowmotion on distributive lattices is very well studied, as the Fundamental Theorem of Finite Distributive Lattices allows one to view it combinatorially as an operator on the set of order ideals of a finite poset \cite{AST, HopkinsSurvey, StrikerSurvey, StrikerWilliams}. There are also notable connections between rowmotion on distributive lattices and homological algebra \cite{MTY, IyamaMarczinzik}. That said, there are some non-distributive semidistributive lattices, including Tamari lattices, on which rowmotion has interesting dynamics \cite{Barnard, DefantLin, HopkinsCDE, ThomasWilliams}. Our goal in this section is to completely describe the dynamics of rowmotion on the cyclic Tamari lattice $\CTam$ and the affine Tamari lattice $\ATam$. 

\begin{remark}\label{rem:trim}
In \cite{Thomas}, Thomas introduced the family of \emph{trim lattices}, which serve as a natural generalization of distributive lattices. Thomas and Williams \cite{ThomasWilliams} defined rowmotion on trim lattices, and they explained that for such lattices, rowmotion could be computed ``in slow motion'' (i.e., via a sequence of local \emph{flips}). The lattices $\CTam$ and $\ATam$ are not trim, so their rowmotion operators cannot be computed in slow motion. This makes the nice dynamics of rowmotion on these lattices all the more remarkable. As far as we are aware, these are the first examples of non-trim lattices with well behaved rowmotion dynamics. 
\end{remark} 

A \dfn{set partition} of a totally ordered set $X$ is a collection of nonempty pairwise-disjoint subsets of $X$ whose union is $X$; the subsets in the set partition are called \dfn{blocks}. Associated to a set partition $\rho$ of $X$ is an equivalence relation $\sim_\rho$ on $X$ defined so that $x\sim_\rho x'$ if and only if $x$ and $x'$ are in the same block of $\rho$. Say a set partition $\rho$ is \dfn{coarser} than a set partition $\rho'$ if every block of $\rho'$ is contained in a block of $\rho$. We say $\rho$ is \dfn{noncrossing} if for all $x_1,x_2,x_3,x_4\in X$ satisfying $x_1<x_2<x_3<x_4$, $x_1\sim_\rho x_3$, and $x_2\sim_\rho x_4$, we have $x_1\sim_\rho x_2$.   

\begin{figure}[ht]
\begin{center}{\includegraphics[height=7.35cm]{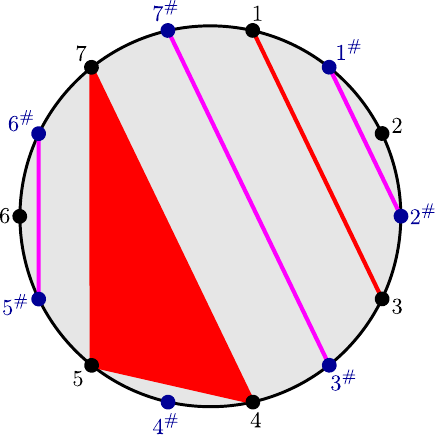}}
\end{center}
\caption{In {\color{red}red} is the type $A_{6}$ noncrossing partition $\{\{1,3\},\{2\},\{4,5,7\},\{6\}\}$, whose Kreweras complement is $\{\{1,2\},\{3,7\},\{4\},\{5,6\}\}$ ({\color{Pink}pink}).}   
\label{fig:type_A_noncrossing} 
\end{figure} 

We will need to consider three different types of noncrossing partitions. First, there are \dfn{type $A_{N-1}$ noncrossing partitions}, which are simply noncrossing set partitions of $[N]$ (with the usual total order on $[N]$). The \dfn{convex hull diagram} of a type $A_{N-1}$ noncrossing partition is obtained by drawing $N$ equally-spaced points numbered $1,\ldots,N$ in clockwise order around the boundary of a disk and then drawing the convex hulls of the blocks (see \cref{fig:type_A_noncrossing}). Second, a \dfn{type $B_N$} noncrossing partition is a type $A_{2N-1}$ noncrossing partition whose convex hull diagram is invariant under $180^\circ$ rotation (see \cref{fig:type_B_noncrossing}). Finally, a \dfn{translation-invariant noncrossing partition} (TINCP) is a noncrossing partition $\rho$ of $\ZZ$ such that for all $a,b\in\ZZ$, we have $a\sim_\rho b$ if and only if $a+n\sim_\rho b+n$ (see \cref{fig:type_B_noncrossing}). Let $\NC(A_{N-1})$, $\NC(B_N)$, and $\TINC(\ZZ)$ be the sets of type $A_{N-1}$ noncrossing partitions, type $B_N$ noncrossing partitions, and TINCPs, respectively. 

\begin{figure}[ht]
\begin{center}{\includegraphics[height=7.5cm]{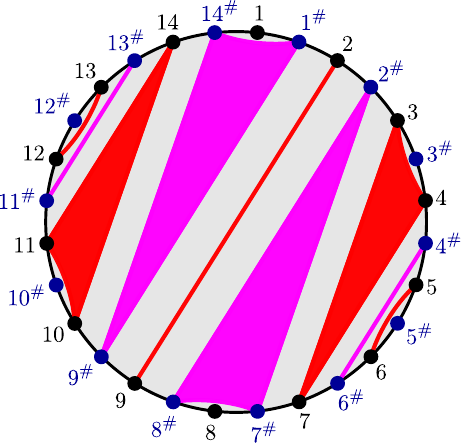}}\\ \vspace{0.8cm}{\includegraphics[width=13.806cm]{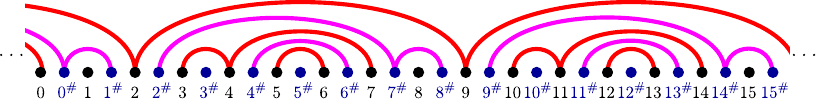}}
\end{center}
\caption{Let $n=7$. On the top is a type $B_{n}$ noncrossing partition $\rho$ ({\color{red}red}) together with $\Krew(\rho)^\#$ ({\color{Pink}pink}). On the bottom is the TINCP $\uw(\rho)$ ({\color{red}red}) together with $\Krew(\uw(\rho))^\#$ ({\color{Pink}pink}). }   
\label{fig:type_B_noncrossing} 
\end{figure} 

There is a bijection $\uw\colon\TINC(\ZZ)\to\NC(B_n)$ that simply reduces numbers modulo $2n$. More precisely, for $Y\subseteq\ZZ$, let $Y/2n\ZZ$ be the subset of $[2n]$ obtained by reducing the elements of $Y$ modulo $2n$ (deleting duplicates). If $\rho$ is a TINCP, then $\uw(\rho)=\{Y/2n\ZZ\mid Y\in\rho\}$. See \cref{fig:type_B_noncrossing}.  

Given a totally ordered set $X$, let $X^\#=\{i^\#:i\in\ZZ\}$ be an order-isomorphic copy of $X$ that is disjoint from $X$, where $i^\#\leq j^\#$ in $X^\#$ if and only if $i\leq j$ in $X$. Given a set partition $\rho$ of $X$, there is a corresponding set partition $\rho^\#$ of $X^\#$ such that $i^\#\sim_{\rho^\#}j^\#$ if and only if $i\sim_\rho j$. We view $X\sqcup X^\#$ as a totally ordered set in which $i<i^\#<j$ whenever $i,j\in X$ satisfy $i<j$. Define the \dfn{Kreweras complement} of a noncrossing partition $\rho$ of $X$, denoted $\Krew(\rho)$, to be the coarsest set partition of $X^\#$ such that $\rho\sqcup\Krew(\rho)^\#$ is a noncrossing partition of $X\sqcup X^\#$. \cref{fig:type_A_noncrossing} illustrates the Kreweras complement of a type $A_6$ noncrossing partition. It is straightforward to see that 
\begin{equation}\label{eq:RKrew}
\uw(\Krew(\rho))=\Krew(\uw(\rho))
\end{equation}
for every TINCP $\rho$; this is illustrated in \cref{fig:type_B_noncrossing}. 

It is known that rowmotion on the Tamari lattice is dynamically equivalent to Kreweras complement on noncrossing partitions of type $A_{n-1}$. We will see that rowmotion on the cyclic Tamari lattice is dynamically equivalent to Kreweras complement on noncrossing partitions of type $B_n$. We would like to thank Nathan Reading for drawing our attention to the relationship between TINCPs and type $B$ combinatorics.

\subsection{Rowmotion on the Cyclic Tamari Lattice} 

Recall from \cref{lem:arcdiagram312} that we have a bijection $\cA$ from the cyclic Tamari lattice to the set of noncrossing arc diagrams. We can ``unwrap'' a noncrossing arc diagram $D$ to obtain a TINCP $\rho_D$. More precisely, $\rho_D$ is the finest TINCP in which $a$ and $b$ are in the same block whenever $\gamma_{a,b}\in D$ (see \cref{fig:noncrossing1}). It is straightforward to see that $D\mapsto \rho_D$ is a bijection from the set of noncrossing arc diagrams to the set of TINCPs. Thus, we obtain a bijection $\cP$ from $\CTam_{312}$ to the set of TINCPs given by $\cP(\preceq)=\rho_{\cA(\preceq)}$. 

\begin{figure}[ht]
\begin{center}{\includegraphics[width=13.806cm]{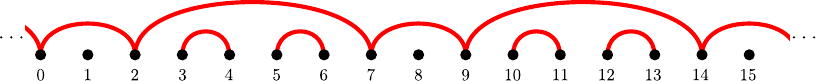}}\\ \vspace{0.8cm}{\includegraphics[width=13.806cm]{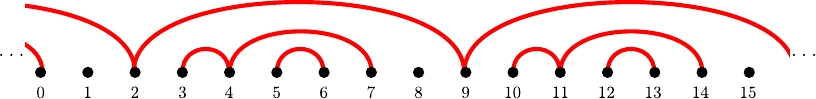}}
\end{center}
\caption{The TINCPs $\rho_D$ and $\rho_{D'}$, where $D$ and $D'$ are the noncrossing arc diagrams on the left and right, respectively, of \cref{fig:arc_diagrams}.}   
\label{fig:noncrossing1} 
\end{figure} 

\begin{proposition}\label{prop:PKrewRow}
For every $\preceq$ in $\CTam_{312}$, we have 
\[\cP(\Row(\preceq))=\Krew(\cP(\preceq)).\]
\end{proposition} 
\begin{proof}
Let $\preceq'\,=\cP^{-1}(\Krew(\cP(\preceq)))$. It is straightforward to see that \[|\cA(\preceq')|=n-|\cA(\preceq)|=n-|\cD(\preceq)|.\] In addition, since the Hasse diagram of $\CTam_{312}$ is a regular graph of degree $n$ (by \cref{prop:regular_cyclic}), we have \[|\cD(\preceq)|+|\cD(\Row(\preceq))|=|\cU(\Row(\preceq))|+|\cD(\Row(\preceq))|=n.\] Hence, $|\cA(\preceq')|=|\cD(\Row(\preceq))|$. Thus, to see that $\preceq'\,=\Row(\preceq)$, we just need to show that if $\gamma_{a,b}\in\cA(\preceq)$ and $\gamma_{c,d}\in\cA(\Row(\preceq))$, then the pairs $\{a,b\}\subseteq\ZZ$ and $\{c^\#,d^\#\}\subseteq\ZZ^\#$ do not cross in $\ZZ\sqcup\ZZ^\#$. That is, we need to show that we cannot have $a\leq c<b\leq d$ or $c<a\leq d<b$. To see this, we use the TIBIT $T=\sT(\Row(\preceq))$. Since $\gamma_{a,b}\in\cA(\preceq)$, the node $\T_a$ is a left child of $\T_b$ in $T$. Since $\gamma_{c,d}\in\cA(\Row(\preceq))$, the node $\T_d$ is a right child of $\T_c$ in $T$. If $a\leq c<b\leq d$, then $\T_b$ is in the right subtree of $\T_c$ in $T$, but this forces $\T_a$ to also be in the right subtree of $\T_c$, which is impossible. If $c<a\leq d<b$, then $\T_d$ is in the right subtree of $\T_a$ in $T$, but this forces $\T_c$ to also be in the right subtree of $\T_a$, which is impossible. 
\end{proof} 

\cref{prop:PKrewRow} allows us to give a complete description of the orbit structure of rowmotion on $\CTam$ via the cyclic sieving phenomenon. Indeed, consider the cyclic group $\cC_{4n}=\langle g_{4n}\rangle$ of order $4n$. There is an action of $\cC_{4n}$ on $\NC(B_n)$ given by $g_{4n}\cdot\rho=\Krew(\rho)$. Armstrong, Stump, and Thomas \cite{AST} proved that the triple $(\NC(B_n),\cC_{4n},\Cat_{B_n}(q))$ exhibits the cyclic sieving phenomenon. Together with \cref{prop:PKrewRow} and \eqref{eq:RKrew}, this implies the following theorem. 

\begin{theorem}\label{thm:CSP_cyclic}
There is an action of the cyclic group $\cC_{4n}=\langle g_{4n}\rangle$ on the cyclic Tamari lattice $\CTam$ satisfying ${g_{4n}\cdot x=\Row(x)}$ for all $x\in\CTam$. Moreover, the triple 
\[\left(\CTam,\cC_{4n},\Cat_{B_n}(q)\right)\] exhibits the cyclic sieving phenomenon. 
\end{theorem}

\begin{remark}\label{rem:actually_2n}
\cref{thm:CSP_cyclic} tells us that the order of rowmotion on $\CTam$ divides $4n$. However, the true order is actually $2n$. This is due to the fact that the order of Kreweras complement on $\NC(B_n)$ is $2n$. Indeed, $\Krew^{2n}$ acts on a type $B_n$ noncrossing partition by rotating the convex hull diagram by $180^\circ$, which leaves the diagram invariant. 
\end{remark} 

\subsection{Rowmotion on the Affine Tamari Lattice} 

If $L$ is a finite lattice and $\equiv$ is a lattice congruence on $L$, then we can consider the map $\pi_\equiv^\downarrow\colon L\to L$ that sends each element of $L$ to the minimum element of its congruence class. The image $\pi_\equiv^\downarrow(L)$ is a lattice quotient of $L$, and one can compute rowmotion on $\pi_\equiv^\downarrow(L)$ in terms of rowmotion on $L$. Namely,  
\[\Row_{\pi_\equiv^\downarrow(L)}=\pi_\equiv^\downarrow\circ\Row_L.\] 

To understand how rowmotion acts on the affine Tamari lattice, we use the fact that $\ATam$ is a lattice quotient of $\CTam$. More precisely, \cref{thm:affine_quotient} tells us that ${\pispine^\downarrow\colon\TIBIT\to\ATIBIT}$ is a surjective lattice homomorphism. We wish to define a related map \[\pi^\downarrow_{\NC}\colon\TINC(\ZZ)\to\TINC(\ZZ).\] Consider a TINCP $\rho$. If there does not exist an integer $i_0$ such that $i_0+n\ZZ$ is a block of $\rho$, then we simply define $\pi^\downarrow_{\NC}(\rho)=\rho$. Now suppose there is an integer $i_0$ such that $i_0+n\ZZ$ is a block of $\rho$. Let $\rho'$ be the TINCP obtained from $\rho$ by replacing the block $i_0+n\ZZ$ with the infinite collection of singleton blocks $\{i_0+kn\}$ for all $k\in\ZZ$. Now let $\pi^\downarrow_{\NC}(\rho)$ be the finest coarsening of $\rho'$ in which $i_0+kn$ is in the same block as $i_0+(k+1)n-1$ for each $k\in\ZZ$. Less formally, we obtain $\pi^\downarrow_{\NC}$ from $\rho$ by deleting the block $i_0+n\ZZ$ and then adding each integer of the form $i_0+kn$ into the block containing $i_0+(k+1)n-1$. (See \cref{fig:piNC}.)  

We have a bijection $\TIBIT\to\TINC(\ZZ)$ given by $T\mapsto\cP(\po_T)$. Note that $\cP(\po_T)$ is the finest partition of $\ZZ$ in which every $a$ is in the same block as $b$ whenever the vertex $\T_a$ is a right child of the vertex $\T_b$ in $T$. One can readily check that 
\begin{equation}\label{eq:poPpiNC}
\cP(\po_{\pispine^\downarrow(T)})=\pi^\downarrow_{\NC}(\cP(\po_T))
\end{equation} for every TIBIT $T$. 

\begin{example}
Let $n=5$. Let $T$ be the TIBIT on the right of \cref{fig:pispine}. Then $\pispine^\downarrow(T)$ is the TIBIT on the left of the same figure. The TINCPs $\cP(\po_T)$ and $\cP(\po_{\pispine^\downarrow(T)})$ are shown on the top and bottom, respectively, of \cref{fig:piNC}. This illustrates \eqref{eq:poPpiNC}. 
\end{example} 

\begin{figure}[ht]
\begin{center}{\includegraphics[width=13.806cm]{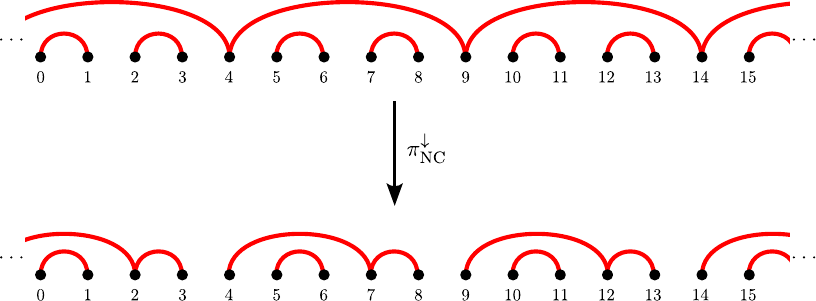}}
\end{center}
\caption{Applying $\pi^\downarrow_{\NC}$ to a TINCP with $n=5$. }   
\label{fig:piNC} 
\end{figure} 

\cref{prop:PKrewRow} and \eqref{eq:poPpiNC} together yield the following result. 

\begin{proposition}\label{prop:ATamKrew}
For every $\preceq$ in $\ATam_{312}$, we have 
\[\cP(\Row_{\ATam_{312}}(\preceq))=\pi^\downarrow_{\NC}(\Krew(\cP(\preceq))).\] 
\end{proposition}

According to \cref{prop:ATamKrew}, rowmotion on the affine Tamari lattice is dynamically equivalent to the map $\pi^\downarrow_{\NC}\circ\Krew\colon\pi^\downarrow_{\NC}(\TINC(\ZZ))\to\pi^\downarrow_{\NC}(\TINC(\ZZ))$. Thus, our goal in the remainder of this section is to describe the orbit structure of the latter map. 

For $i\in\ZZ$, let $\Xi_i$ be the set of TINCPs $\rho$ such that 
$\rho$ has no infinite blocks and such that $i$ is in the same block as $i+n-1$. Note that the sets $\Xi_0,\Xi_1,\ldots,\Xi_{n-1}$ are disjoint. Let $\Xi=\Xi_0\sqcup\Xi_1\sqcup\cdots\sqcup\Xi_{n-1}$. Let $\Xi'$ be the set of TINCPs that contain a block of the form $i_0+n\ZZ$ for some integer $i_0$. Note that $\pi^\downarrow_{\NC}(\TINC(\ZZ))=\TINC(\ZZ)\setminus\Xi'$. It is straightforward to check that $\Krew$ restricts to a bijection from $\Xi$ to $\Xi'$ and also restricts to a bijection from $\Xi'$ to $\Xi$. Thus, $\Krew$ is also a bijection from $\TINC(\ZZ)\setminus(\Xi\sqcup\Xi')$ to itself. Let us also observe that the map $\Krew^2\colon\TINC(\ZZ)\to\TINC(\ZZ)$ acts by simply shifting each TINCP by one step. More precisely, we have $i\sim_{\rho} j$ if and only if $i-1\sim_{\Krew^2(\rho)}j-1$ for all $i,j\in\ZZ$. This implies that $\Krew^2$ restricts to a bijection from $\Xi_i$ to $\Xi_{i-1}$ for each $i\in\ZZ$. 

\begin{lemma}\label{lem:CatAn-22n}
The map $\Krew\colon(\Xi\sqcup\Xi')\to(\Xi\sqcup\Xi')$ has exactly $\Cat_{A_{n-2}}$ orbits, each of which has size $2n$. 
\end{lemma}
\begin{proof}
Because $\Krew$ maps $\Xi$ to $\Xi'$ and vice versa and $\Krew^2$ maps $\Xi_i$ to $\Xi_{i-1}$ for each $i$, the orbits of $\Krew\colon(\Xi\sqcup\Xi')\to(\Xi\sqcup\Xi')$ must all have size $2n$. Moreover, every such orbit contains a unique element of $\Xi_0$. There is a natural bijection $\beta\colon\Xi_0\to\NC(A_{n-2})$ defined so that $i\sim_{\rho}j$ if and only if $i\sim_{\beta(\rho)}j$ for all $i,j\in[n-1]$. Hence, the number of orbits of $\Krew\colon(\Xi\sqcup\Xi')\to(\Xi\sqcup\Xi')$ is the cardinality of $\NC(A_{n-2})$, which is well known to be $\Cat_{A_{n-2}}$.   
\end{proof}

\begin{proposition}\label{prop:KrewXiXi'}
There is an action of the cyclic group $\cC_{4n}=\langle g_{4n}\rangle$ on $\TINC(\ZZ)\setminus(\Xi\sqcup\Xi')$ satisfying ${g_{4n}\cdot\rho=\pi^\downarrow_{\NC}(\Krew(\rho))}$ for all $\rho\in\TINC(\ZZ)\setminus(\Xi\sqcup\Xi')$. Moreover, the triple 
\[\left(\TINC(\ZZ)\setminus(\Xi\sqcup\Xi'),\cC_{4n},\Cat_{B_n}(q)-[2n]_{q^2}\Cat_{A_{n-2}}\right)\] exhibits the cyclic sieving phenomenon. 
\end{proposition}
\begin{proof}
We know that $\Krew$ is a bijection from $\TINC(\ZZ)\setminus(\Xi\sqcup\Xi')$ to itself. Because $\pi^\downarrow_{\NC}$ acts as the identity map on $\TINC(\ZZ)\setminus(\Xi\sqcup\Xi')$, the proof follows from \cref{thm:CSP_cyclic,lem:CatAn-22n}.  
\end{proof}

By \cref{prop:ATamKrew,prop:KrewXiXi'}, in order to understand the orbit structure of rowmotion on $\ATam$, we are left to understand the orbit structure of the map $\pi^\downarrow_{\NC}\circ\Krew\colon\Xi\to\Xi$. In what follows, let 
\[\epsilon(n)=\begin{cases}
    1 & \text{if $n$ is odd}; \\
    2 & \text{if $n$ is even}.
\end{cases}\]

\begin{lemma}\label{lem:Xi0CSP}
There is an action of the cyclic group $\cC_{2(n-1)}=\langle g_{2(n-1)}\rangle$ on $\Xi_0$ satisfying \[g_{2(n-1)}\cdot\rho=(\pi^\downarrow_{\NC}\circ\Krew^{-1})^n(\rho)\] for all $\rho\in\Xi_0$. Moreover, the triple 
\[\left(\Xi_0,\cC_{2(n-1)},\Cat_{A_{n-2}}(q^{\epsilon(n)})\right)\] exhibits the cyclic sieving phenomenon. 
\end{lemma}
\begin{proof}
We consider two actions of the cyclic group of order $2(n-1)$ on $\NC(A_{n-2})$; to distinguish the two actions, we consider a different (isomorphic) copy $\cC_{2(n-1)}'$ of this cyclic group with generator $g_{2(n-1)}'$. 
It is well known that there is an action of $\cC_{2(n-1)}'$ on $\NC(A_{n-2})$ given by $g_{2(n-1)}'\cdot\rho=\Krew^{-1}(\rho)$. The action of $\cC_{2(n-1)}$ on $\NC(A_{n-2})$ is given by $g_{2(n-1)}\cdot\rho=\Krew^{-n}(\rho)$. Armstrong, Stump, and Thomas \cite{AST} proved that the triple $(\NC(A_{n-2}),\cC_{2(n-1)}',\Cat_{A_{n-2}}(q))$ exhibits the cyclic sieving phenomenon. Since $\epsilon(n)=\gcd(n,2(n-1))$, it follows that the triple $(\NC(A_{n-2}),\cC_{2(n-1)},\Cat_{A_{n-2}}(q^{\epsilon(n)}))$ exhibits the cyclic sieving phenomenon. 

Now consider the bijection $\beta\colon\Xi_0\to\NC(A_{n-2})$ from the proof of \cref{lem:CatAn-22n}. It is straightforward to see that $\pi^\downarrow_{\NC}\circ\Krew^{-1}$ maps $\Xi_0$ to itself and that
\begin{equation}\label{eq:betapiNCKrew}\beta(\pi^\downarrow_{\NC}(\Krew^{-1}(\rho)))=\Krew^{-1}(\beta(\rho))
\end{equation} 
for all $\rho\in\Xi_0$ (see \cref{exam:betapiNCKrew}). This means that the maps \[(\pi^\downarrow_{\NC}\circ\Krew^{-1})^n\colon\Xi_0\to\Xi_0\quad\text{and}\quad\Krew^{-n}\colon\NC(A_{n-2})\to\NC(A_{n-2})\] are dynamically equivalent, so the desired instance of the cyclic sieving phenomenon follows. 
\end{proof} 

\begin{example}\label{exam:betapiNCKrew}
Let $n=8$. On the top of \cref{fig:betapiNCKrew} is a partition $\rho\in\Xi_0$, while on the bottom is the partition  $\pi^\downarrow_{\NC}(\Krew^{-1}(\rho))$. We have 
\[\beta(\rho)=\{\{1,2\},\{3,7\},\{4\},\{5,6\}\}\quad\text{and}\quad\beta(\pi^\downarrow_{\NC}(\Krew^{-1}(\rho)))=\{\{1,3\},\{2\},\{4,5,7\},\{6\}\};\] these are precisely the partitions shown in \cref{fig:type_A_noncrossing}. That figure illustrates the identity \eqref{eq:betapiNCKrew}.   
\end{example}

\begin{figure}[ht]
\begin{center}{\includegraphics[width=13.806cm]{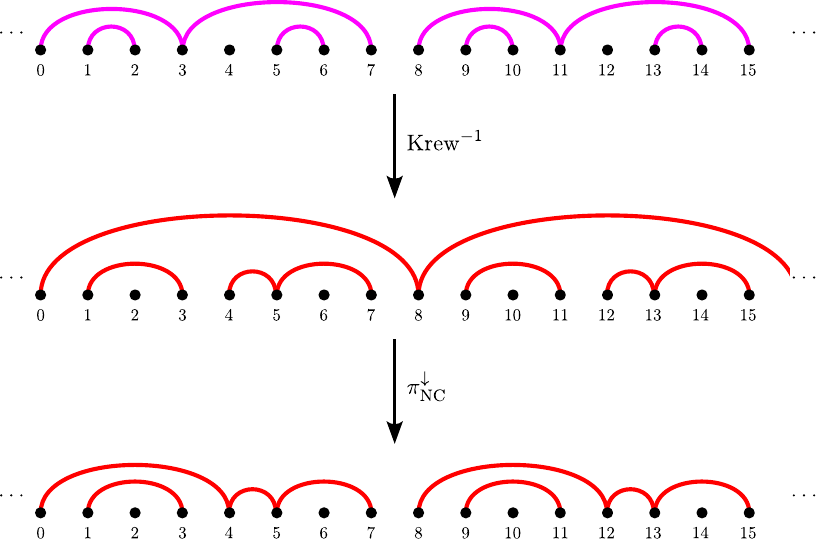}}
\end{center}
\caption{Applying $\pi^\downarrow_{\NC}\circ\Krew^{-1}$ to a partition in $\Xi_0$, where $n=8$. }   
\label{fig:betapiNCKrew}  
\end{figure} 

\begin{proposition}\label{prop:pi_Krew}
There is an action of the cyclic group $\cC_{2n(n-1)}=\langle g_{2n(n-1)}\rangle$ on $\Xi$ satisfying $g_{2n(n-1)}\cdot\rho=\pi^\downarrow_{\NC}(\Krew(\rho))$ for all $\rho\in\Xi$. Moreover, the triple 
\[\left(\Xi,\cC_{2n(n-1)},[n]_{q^{2(n-1)}}\Cat_{A_{n-2}}(q^{\epsilon(n)})\right)\] exhibits the cyclic sieving phenomenon. 
\end{proposition}
\begin{proof}
For each $i\in\ZZ$, it is straightforward to check that  $\pi^\downarrow_{\NC}\circ\Krew$ is a bijection from $\Xi_{i}$ to $\Xi_{i-1}$. It follows that every orbit of $\pi^\downarrow_{\NC}\circ\Krew\colon\Xi\to\Xi$ has size divisible by $n$. Therefore, there is a one-to-one correspondence between orbits of $\pi^\downarrow_{\NC}\circ\Krew\colon \Xi\to\Xi$ and orbits of the map $(\pi^\downarrow_{\NC}\circ\Krew)^n\colon\Xi_0\to \Xi_0$, where an orbit of the latter map is $n$ times as large as the corresponding orbit of the former map. 
Because $\Krew^2$ acts by shifting a partition, it commutes with $\pi^\downarrow_{\NC}$. In addition, $\Krew^{-2n}$ is the identity map. Hence, 
\begin{align*}
(\pi^\downarrow_{\NC}\circ\Krew)^n&=\Krew^{-2n}\circ(\pi^\downarrow_{\NC}\circ\Krew)^n \\ 
&=(\pi^\downarrow_{\NC}\circ\Krew^{-1})^n.
\end{align*}
In light of \cref{lem:Xi0CSP}, this completes the proof. 
\end{proof} 

The map $\cP$ restricts to a bijection from $\ATam_{312}$ to $\TINC(\ZZ)\setminus\Xi'$, so we can combine \cref{prop:ATamKrew,prop:KrewXiXi',prop:pi_Krew} to deduce the following theorem, which gives a complete description of the orbit structure of rowmotion on the affine Tamari lattice. 

\begin{theorem}\label{thm:CSP_affine}
There is an action of the cyclic group $\cC_{4n(n-1)}=\langle g_{4n(n-1)}\rangle$ on the affine Tamari lattice $\ATam$ satisfying $g_{4n(n-1)}\cdot x=\Row(x)$ for all $x\in\ATam$. Moreover, the triple 
\[\left(\ATam,\cC_{4n(n-1)},\Cat_{B_n}(q^{n-1})-[2n]_{q^{2(n-1)}}\Cat_{A_{n-2}}+[n]_{q^{4(n-1)}}\Cat_{A_{n-2}}(q^{2\epsilon(n)})\right)\] exhibits the cyclic sieving phenomenon. 
\end{theorem} 

\begin{remark}\label{rem:actually_2nn-1}
\cref{thm:CSP_affine} tells us that the order of rowmotion on $\CTam$ divides $4n(n-1)$. However, the true order is actually $2n(n-1)$ when $n\geq 4$ and is $6$ when $n=3$. This is due to \cref{rem:actually_2n} and our proof of \cref{thm:CSP_affine}. 
\end{remark}

\section{Future Work}\label{sec:future} 
Tamari lattices are extremely rich mathematical objects, and while we have discussed analogues of several of their properties for cyclic and affine Tamari lattices, there are still numerous directions left unexplored. 

\subsection{Associahedra}
Let $\widetilde U$ denote the real vector space freely generated by vectors $\tbe_1,\ldots,\tbe_n,\delta$. We define $\tbe_j$ for all $j\in\ZZ$ by taking $\tbe_{j+n}=\tbe_j-\delta$. Let $\widetilde V$ be the subspace of $U$ spanned by $\{\tbe_i-\tbe_{i+1}\mid i\in\ZZ\}$. 
For $i<j$, let \[ \widetilde{\Delta}_{ij} \coloneqq \operatorname{conv}(0, \tbe_i-\tbe_j, \tbe_{i+1}-\tbe_j, \tbe_{i+2}-\tbe_j,\ldots, \tbe_{j-1}-\tbe_j), \] where $\operatorname{conv}$ denotes convex hull. We remark that each polytope $\widetilde{\Delta}_{ij}$ is an example of a \emph{shard polytope} of type $\widetilde A_{n-1}$ \cite{PPR}. 

\begin{definition}
  The \dfn{cyclic Tamari associahedron} is the Minkowski sum
    \[ \sum_{1\leq i \leq n} \sum_{1\leq \ell \leq n} \widetilde{\Delta}_{i,i+\ell}. \]  
    The \dfn{affine Tamari associahedron} is the Minkowski sum
    \[ \sum_{1\leq i \leq n} \sum_{1\leq \ell \leq n-1} \widetilde{\Delta}_{i,i+\ell}. \]
\end{definition}

\begin{remark}
The cyclic Tamari associahedron should not be confused with the $n$-dimensional \emph{cyclohedron}, which is the generalized associahedron of type $B_n$ \cite{Chapoton}. Although these two polytopes have the same number of vertices, they are not isomorphic in general. For example, the Hasse diagram in \cref{fig:CTam3_TIBITs} is not isomorphic to the $1$-skeleton of the $3$-dimensional cyclohedron.  
\end{remark}

The cyclic and affine Tamari associahedra are special examples of the polytopes constructed by Fei in \cite[Proposition~9.5]{Fei}. It follows from Fei's work that the Hasse diagram of the cyclic (respectively, affine) Tamari lattice is isomorphic to the $1$-skeleton of the cyclic (respectively, affine) Tamari associahedron.

It would be very interesting to further understand the geometric properties of the cyclic and affine Tamari associahedra. 

\subsection{The Cyclic Tamari Lattice via Triangulations}
In \cref{sec:triangulations}, we showed how to describe the Hasse diagram of $\ATam$ using triangulations of type $D_n$. We believe it should also be possible to describe the Hasse diagram of $\CTam$ using centrally symmetric triangulations of type $D_{2n}$. 

\subsection{Maximal Green Sequences}
Numerical evidence suggests that the number of minimum-length maximal green sequences of $A_\cQ$ (i.e., the number of maximal chains of $\ATam$ of length $2n-2$) is $n!$. 

\subsection{Generalizations}
There are now several well studied generalizations of Tamari lattices such as $m$-Tamari lattices \cite{Bousquet, Bousquet2, DefantLin}, $\nu$-Tamari lattices \cite{vonBell,CeballosGeometry,PrevilleViennot}, and $m$-eralized Tamari lattices \cite{WhyTheFuss}. It could be fruitful to consider analogues of these lattices in our setting of cyclic and affine Tamari lattices. For example, is there a meaningful notion of the ``$m$-eralized cyclic Tamari lattice''? 

One way to construct the $\nu$-Tamari lattices is as intervals in Tamari lattices. Let $\mathsf{Tam}$ denote the $n$-th Tamari lattice, and let $\mathbb{A}_{\mathsf{Tam}}$ denote the collection of subsets of the set of atoms of $\mathsf{Tam}$. There is a map $@\colon\mathsf{Tam}\to\mathbb{A}_{\mathsf{Tam}}$ defined by $@(x)=\{a\in\mathbb{A}_{\mathsf{Tam}}\mid a\leq x\}$. The fibers of $@$ are $\nu$-Tamari lattices. There are obvious analogues of the map $@$ for $\CTam$ and $\ATam$; perhaps studying the fibers of these maps is worthwhile. 

\subsection{The Shard Intersection Order}
Let $L$ be a finite semidistributive lattice. For each $w\in L$, let $\Phi(w)=\{\mathfrak{j}(x\lessdot y)\mid \mathsf{Pop}(w)\leq x\lessdot y\leq w\}$, where $\mathsf{Pop}(x)=\bigwedge(\{u\in L\mid u\lessdot w\}\cup\{w\})$. The \dfn{shard intersection order} is the partial order $\leq_{\mathrm{shard}}$ on $L$ defined so that $u\leq_{\mathrm{shard}}w$ if and only if $\Phi(u)\subseteq\Phi(w)$. 
The shard intersection order was first defined by Reading for posets of regions of simplicial hyperplane arrangements \cite{ReadingShard}, where he proved that it is in fact a lattice with a geometric interpretation in terms of intersections of polyhedral cones called \emph{shards}. It was later generalized (under different names) to broader families of lattices (including semidistributive lattices) in \cite{DefantSemidistrim, Muhle}. The shard intersection order of the Tamari lattice is isomorphic to the lattice of type $A_{n-1}$ noncrossing partitions under reverse-refinement order. It would be interesting to study the shard intersection orders of $\CTam$ and $\ATam$.   

\subsection{Linear Intervals} 
Some recent articles have studied linear intervals in the Tamari lattice and its generalization \cite{CC, Cheneviere}. Studying linear intervals in the cyclic and affine Tamari lattices could be similarly fruitful.  

\subsection{Pop-Stack Operators}
It could be interesting to study the dynamics of the pop-stack operators of cyclic and affine Tamari lattices, as in \cite{AjranDefant, BarnardPop, ChoiSun, DefantCoxeterPop, DefantMeeting, Hong}. 

\section*{Acknowledgments}
Grant Barkley was partially supported by NSF grant DMS-1854512. Colin Defant was supported by the National Science Foundation under Award No.\ 2201907 and by a Benjamin Peirce Fellowship at Harvard University. We thank Andrew Sack for describing his lattice construction to us, and we thank Nicholas Williams for a helpful discussion. Grant Barkley would also like to thank Nathan Reading and David Speyer for helpful conversations related to this work.



\end{document}